\documentclass[letterpaper, 11pt,  reqno]{amsart}
\usepackage[margin=1.2in,marginparwidth=1.5cm, marginparsep=0.5cm]{geometry}

\usepackage{amsmath,amssymb,amscd,amsthm,amsxtra, esint, xcolor,mathtools, mathrsfs}

\usepackage[implicit=true]{hyperref}

 \allowdisplaybreaks[2]

\usepackage{color}

\newcommand{\wh}{\widehat}

\newtheorem{theorem}{Theorem} [section]

\newtheorem{lemma}[theorem]{Lemma}
\newtheorem{proposition}[theorem]{Proposition}
\newtheorem{remark}[theorem]{Remark}

\newtheorem{definition}[theorem]{Definition}
\newtheorem{corollary}[theorem]{Corollary}

\newcommand{\I}{\hspace{0.5mm}\text{I}\hspace{0.5mm}}
\newcommand{\II}{\text{I \hspace{-2.8mm} I} }
\newcommand{\noi}{\noindent}

\newcommand{\R}{\mathbb{R}}

\newcommand{\C}{\mathbb{C}}

\newcommand{\bul}{\bullet}

\newcommand{\bE}{\mathbb{E}}
\newcommand{\bP}{\mathbb{P}}
\newcommand{\bR}{\mathbb{R}}

\newcommand{\cB}{\mathcal{B}}
\newcommand{\cF}{\mathcal{F}}

\newcommand{\cP}{\mathcal{P}}

\newcommand{\bv}{\big\vert}

\newcommand{\W}{\mathcal{W}}

\newcommand{\E}{\mathbb{E}}

\newcommand{\ff}{\mathfrak{f}}

\newcommand{\F}{\mathcal{F}}
\newcommand{\cG}{\mathcal{G}}

\newcommand{\cN}{\mathcal{N}}

\newcommand{\bfN}{\mathbf{N}}
\newcommand{\scrN}{\mathscr{N}}

\newcommand{\al}{\alpha}
\newcommand{\be}{\beta}
\newcommand{\dl}{\delta}

\newcommand{\eps}{\varepsilon}

\newcommand{\s}{\sigma}

\newcommand{\wt}{\widetilde}

\newcommand{\dx}{\partial_x}
\newcommand{\dt}{\partial_t}

\renewcommand{\o}{\omega}
\renewcommand{\O}{\Omega}

\newcommand{\les}{\lesssim}

\newcommand{\jb}[1]
{\langle #1 \rangle}

\newcommand{\ind}{\mathbf 1}

\newcommand{\PP}{\mathbb{P}}

\newcommand{\N}{\mathbb{N}}

\newcommand{\Lip}{\textup{Lip}}

\newcommand{\Var}{\textup{Var}}
\newcommand{\fm}{\mathfrak{m}}
\newcommand{\fH}{\mathfrak{H}}
\newcommand{\dom}{\textup{dom}}

\numberwithin{equation}{section}
\numberwithin{theorem}{section}

\makeatletter
\@namedef{subjclassname@2020}{\textup{2020} Mathematics Subject Classification}
\makeatother

\begin{document}
\baselineskip = 14pt

\title[CLT for  SNLW with pure-jump L\'evy white noise]
{Central limit theorem  for stochastic nonlinear wave equation   with pure-jump L\'evy white noise}

 \author[R.M. Balan and G. Zheng]
{Raluca M. Balan and Guangqu Zheng}

\address{Raluca M. Balan,
Department of Mathematics and Statistics\\
University of Ottawa\\
150 Louis Pasteur Private\\
Ottawa, ON, K1N 6N5\\
Canada}

\email{Raluca.Balan@uottawa.ca}

\address{Guangqu Zheng,
Department of Mathematics and Statistics\\
Boston University\\
665 Commonwealth Ave\\
Boston, MA 02215\\
USA}
\email{gzheng90@bu.edu}



\vspace{-4mm}

\begin{abstract}  

In this paper, we study the random field solution to the stochastic 
nonlinear wave equation (SNLW) with constant initial conditions and
multiplicative noise $\sigma(u)\dot{L}$,
where the nonlinearity is encoded in a Lipschitz function 
$\sigma: \mathbb{R}\to\mathbb{R}$
and $\dot{L}$ denotes a  pure-jump L\'evy white noise on 
$\mathbb{R}_+\times\mathbb{R}$
with finite variance. Combining tools from It\^o calculus and Malliavin calculus,
we are able to establish the Malliavin differentiability of the solution with sharp 
moment bounds for the Malliavin derivatives. As an easy consequence, 
we obtain the spatial ergodicity of the solution to 
SNLW
that leads to a law of large number result for the  spatial integrals of the solution 
over $[-R, R]$ as $R\to\infty$. 
One of the main results of this paper is the obtention of the corresponding 
Gaussian fluctuation with rate of convergence in Wasserstein distance. 
To achieve this goal, we  adapt the discrete Malliavin-Stein bound 
from  Peccati, Sol\'e, Taqqu, and Utzet  ({\it Ann. Probab.}, 2010), 
and further combine it with the aforementioned 
moment bounds of Malliavin derivatives and It\^o tools.
Our work  substantially improves
  our previous results (\textit{Trans.~Amer.~Math.~Soc.}, 2024) 
  on  the linear equation that  heavily relied on the explicit chaos expansion of the  solution. 
In current  work, we also establish a functional version, an almost sure version
of the central limit theorems, and the (quantitative) asymptotic independence 
of spatial integrals from the solution. The asymptotic independence result
is established based on  an observation 
of L. Pimentel (\textit{Ann.~Probab.}, 2022)
and a further adaptation of Tudor's generalization  
(\textit{Trans.~Amer.~Math.~Soc.}, 2025) to the Poisson setting.

\end{abstract}

\date{\today; Corresponding author: R.M.~Balan.}
\maketitle

\vspace{-10mm}

\tableofcontents

\baselineskip = 14pt

\section{Introduction}

The study of stochastic partial differential equations (SPDEs)
  is an active research area,
which has been growing tremendously over the past four decades, 
providing the rigorous framework for modeling random phenomena
 that evolve in space and time. 
Among various approaches  in the literature,
 the random field approach initiated by Walsh in his Saint-Flour notes \cite{walsh86} 
 has been very influential, being embraced quickly by the scientific community.\footnote{A 
 random field  solution is a collection
of real-valued random variables indexed by space-time.} 
 Walsh's random field approach was built on It\^o's theory of stochastic integration, 
 and was further pushed forward by Dalang in \cite{dalang99}. As a result, 
 the random field approach also bears the name of Dalang-Walsh's approach. 
Two families of  equations that have been extensively studied using this approach are
   the {\em stochastic heat equations} (SHEs) and  the {\em stochastic wave equations} 
(SWEs).

Pragmatically speaking,  
an SPDE often describes a complex system with two main  inputs 
(i.e., the initial condition  and the noise).  
We will not discuss about the initial condition here, 
but we will make a few comments about the noise within the random field approach. 
The classical example is the stochastic (nonlinear) wave equation posed on 
$\R_+\times\R$:

\noi
\begin{align}
(\dt^2 - \dx^2) u(t,x) = \s(u (t,x)) \dot{W}(t,x), \label{gSWE}
\end{align}

\noi
where $\s: \R\to\R$ is  deterministic, Lipschitz continuous,
and $\dot{W}$ stands for the space-time Gaussian white noise on $\R_+\times\R$.
This is the running example in Walsh's notes \cite{walsh86}. 
The space-time Gaussian white noise $\dot{W}$  is a generalization of the Brownian motion,
and one can develop the corresponding It\^o integration theory to study
the above SWE \eqref{gSWE}. In fact,  the above SWE admits a unique   random field solution 
only in spatial dimension one, and the same comment applies 
to SHE (with $\dt^2 - \dx^2$ replaced by the heat operator 
$\dt - \frac{1}{2}\dx^2$). When the spatial dimension is strictly bigger than one, 
the above SPDEs fall into the realm of singular SPDEs and thus do not
admit a random-field solution.
In a seminal   work \cite{dalang99},
Dalang introduced the
 Gaussian noises  that are white in time and spatially homogeneous 
 such that  the spectral measure of spatial correlation function satisfies 
 certain mild integrability assumption (known as the Dalang's condition nowadays).
In Dalang's generalized framework, one retains the martingale features
necessary for It\^o's theory due to the noises being white-in-time.
Another popular example of Gaussian noises
 is the fractional Gaussian noise, which is a generalization of the fractional Brownian motion. 
In this case, due to the absence of martingale structures (when the Hurst index in the 
temporal direction
is different from $1/2$), one has to bring in other techniques to develop a solution  theory.
For example, when $\s(u)=u$ (i.e., we consider a linear equation),
one can apply the Malliavin calculus (particularly Wiener chaos expansion)
to establish a solution theory; see, e.g., \cite{HHNT15, BS17}.
For the general theory of  Malliavin calculus, we refer interested readers
to  Nualart's monograph \cite{Nua06};
see also  \cite{NN18} that presents 
 a pedagogical introduction to Malliavin calculus in  both Gaussian and Poisson settings.
 
 In this paper, we consider another family of driving noises,
 the space-time  L\'evy noises that generalize the classical  L\'evy processes;
 see \cite{B15} for the corresponding 
 It\^o integration theory and its applications to SPDEs. 
An in-depth analysis of various  solution properties  of such 
SPDEs driven by finite-variance L\'evy noise is possible using a discrete version of Malliavin calculus, with respect to the underlying Poisson random measure of the noise. More details about Malliavin calculus in the Poisson setting can be found in   \cite{Last16,LPS16,LP18,PSTU10}.

Among   many studies of various solution properties of SPDEs, a problem that comes persistently
in the recent literature is related to the asymptotic behaviour of the spatial integral of the solution, when the integration region tends to cover the whole space. 
Such a problem was initiated by   Huang, Nualart, and Viitasaari in their paper \cite{HNV},
which established the Gaussian fluctuation of the spatial averages of the SHE with
space-time Gaussian white noise. Since then, there has been a rapidly growing literature 
on the limit theorems for these spatial statistics. This research theme also includes
 SWEs and spatially homogeneous Gaussian noises and/or fractional Gaussian noises. 
 The following sample of relevant references will convince the reader of the order of magnitude of this growing area: \cite{HNVZ, DNZ20,  BNZ,   BNQSZ,  CKNP22, CKNPjfa,  NZ20a, NZ20b,  NSZ21, NZ22, NXZ22, PU22, Ebina23a, Ebina23b}. In a nutshell, by combining Stein's method for normal approximation with Malliavian calculus techniques, one can prove that with appropriate normalization, the spatial integral converges in distribution to the standard normal distribution, and can provide precise estimates for the rate of this central convergence, yielding a so-called {\em qualitative central limit theorem} (QCLT). From this, without a significant amount of effort, one can
 also  usually derive   a {\em functional central limit theorem} (FCLT) for the spatial integral process.

In a recent article \cite{BZ23}, we initiated a similar program 
for the stochastic wave equation with pure-jump L\'evy white noise $\dot{L}$. 
More precisely, we  studied the {\em hyperbolic Anderson model},  in which the noise
takes the multiplicative form $u \dot{L}$. Due to the linearity in $u$, one has 
an explicit Wiener chaos expansion of the solution, which is crucial 
for us to obtain fine estimates of Malliavin derivatives up to second order. 
Then, we combined them with Trauthwein's improved second-order 
Poincar\'e inequality \cite{Tara} to establish the QCLTs.

The goal of the present article is to investigate  the QCLT problem for the SWE with general nonlinearity. In addition, we study the asymptotic independence between the spatial integral and the solution, by developing a Poissonian  analogue of the multivariate Malliavin-Stein method  in \cite{tudor23} (see Section \ref{SEC4_2}).

\subsection{Stochastic nonlinear wave equation}

Consider
the following stochastic  nonlinear  wave equation driven
by a pure-jump space-time L\'evy white noise $\dot{L}$
on $\R_+\times\R$:

\noi
\begin{align}
\begin{cases}
 & \hspace{-4mm} ( \dt^2 - \dx^2) u   = \s(u) \dot{L} \\[0.5em]
&  \hspace{-4mm}  \big( u(0, \bul), \dt u(0, \bullet) \big) = (1, 0),
\end{cases}
\label{SWE}
\end{align}

\noi
where $\s: \R\to\R$ is a (deterministic) Lipschitz continuous function with $\s(1) \neq 0$,
and  $\dot{L}(s, y) = L(ds, dy)$ denotes the  space-time (pure-jump) L\'evy white 
noise on $\R_+\times\R$. To give the rigorous definition of the noise, we need to introduce several elements.

\noi
\begin{enumerate}
\item[(a)] Let $\R_0 := \R\setminus\{0\}$ be equipped with
the distance $d(x,y)= |x^{-1} - y^{-1}|$, and  $\nu$ be a $\s$-finite measure on $\R_0$
 satisfying the {\it L\'evy-measure  condition}\footnote{Sometimes, we will treat $\nu$
as a measure on $\R$ with $\nu(\{0\})=0$.
}

\noi  
\begin{align*} % \label{Levy_M}
\int_{\R_0} \min\{1, |z|^2\} \nu(dz) <\infty.
\end{align*}
In addition, throughout this paper, we will assume that $\nu$ satisfies
\begin{align}
m_2:= \int_{\R_0}|z|^2\nu(dz) \in(0, \infty).
\label{m2}
\end{align}
This guarantees that the noise has finite variance, which is essential for our purposes.

\item[(b)] Let $\fH=L^2(\mathbf{Z},\mathcal{Z},\fm)$, where
\begin{align*} %\label{bfZ}
\mathbf{Z}\coloneqq \R_+\times\R\times\R_0,
\,\,\,
\fm = {\rm Leb}\times   \nu,
\end{align*} 
$\mathcal{Z}$ is the Borel $\s$-field on $\mathbf{Z}$,  
and $ {\rm Leb}$ denotes the Lebesgue measure on $\R_+\times\R$.

\item[(c)] Let $N$ be a {\em Poisson random measure} on $(\mathbf{Z}, \mathcal{Z})$ 
with intensity  $\fm$, defined on a complete probability space
 $(\Omega,\cF,\bP)$, and $\widehat{N} = N - \fm$ be the compensated version of $N$;
  see Definition \ref{def:PRM} below. Let $\mathbb{F}=\{ \F_t\}_{ t\in\R_+}$ be the filtration induced by $N$, given as in \eqref{filtraF}.
    %\footnote{More %precisely, let $\F^0_t$ be the $\s$-algebra
%generated by the random variables
%$N([0,s]\times A \times B)$ with $s \in [0,t]$ and ${\rm Leb}([0,s]\times A)+ \nu(B)<\infty$.
%And let $\F_t = \s\big( \F_t^0 \cup \mathcal{N} \big)$ be the
%$\s$-algebra generated by $\F_t^0$ and  the set $\mathcal{N} $ of $\PP$-null sets.
%This gives us a filtration $\mathbb{F}=\{ \F_t\}_{t\in\R_+}$.}.

\item[(d)] For any Borel set $A\subset \R_+\times\R$ with ${\rm Leb}(A) <\infty$, let
\begin{equation}
\label{def-L}
L(A) = \int_{A\times \R_0} z \widehat{N}(ds, dy, dz),
\end{equation}
where the integral is defined in the It\^o sense. Then,
 $L(A)$ is an infinitely divisible random variable  with
characteristic function %L\'evy-Khintchine formula
\[
\E\big[ e^{i\lambda L(A)}\big]
= \exp\left\{{\rm Leb}(A)\int_{\R_0}  (e^{i\lambda z}  - 1 - i\lambda z  ) \nu(dz) \right\},
 \quad \lambda \in\R.
\]
Moreover, $L(A_1),\ldots,L(A_n)$ are independent for disjoins sets $A_1,\ldots,A_n$.
See \cite{B15} for more on the L\'evy noise.

\end{enumerate}

The product between the unknown $\s(u)$ and the L\'evy noise $\dot{L}$ is interpreted in the It\^o sense. By analogy with Duhamel's principle, we introduce the following definition. 
%Let $\mathbb{F}$ be the natural filtration generated by $N$, given by \eqref{filtraF} below.

\begin{definition}
 We say that a process $u=\{u(t,x):t\geq 0,x \in \bR\}$ is a  {\rm(}mild{\rm)} solution to   \eqref{SWE} if $u$ is $\mathbb{F}$-predictable 
and satisfies the following integral equation
\begin{align}\label{mild}
u(t,x) = 1 + \int_0^t \int_\R G_{t-s}(x-y) \s\big( u(s,y) \big)  L(ds, dy),
\end{align}

\noi
where
\begin{equation}
\label{def-G}
G_t(x)= \frac{1}{2}\ind_{\{ |x|< t\}} , \, \text{with $G_t = 0$ for $t\leq 0$},
\end{equation} 
denotes  the fundamental solution to the deterministic wave equation on $\bR_{+} \times \bR$.
\end{definition}

Here,  the stochastic integral in \eqref{mild} is understood in the It\^o sense:
\[
\int_{\R_{+}\times\R} X(s, y) L(ds, dy) 
= \int_{\R_{+}\times\R\times\R_0} X(s, y)z \widehat{N}(ds, dy, dz),
\]
whenever  the right hand-side is well-defined; see 
Section \ref{SEC2_1} for 
more details on the stochastic integration.

\smallskip

Note that if $\s(1) = 0$, then  $u(t, x) \equiv 1$ satisfies equation \eqref{mild},
and is indeed the unique solution to \eqref{SWE}. To exclude this trivial case,
we have assumed    $\s(1) \neq 0$ in this paper. 

\smallskip

From \cite[Theorem 1.1]{BN16}, we know that equation \eqref{SWE} has
a unique random field solution $u$ that is $L^2(\PP)$-continuous 
 with
$
\sup\big\{ \E [ | u(t,x)|^2  ]:  (t,x)\in [0,T]\times\R\} <\infty
$
 for any  finite $T > 0$. 
 Moreover, if 
\begin{equation}
\label{mp}
m_p := \int_{\R_0} |z|^p \nu(dz) <\infty 
\end{equation}
for some $p\geq 2$, then

\noi
\begin{align} \label{KTP}
K_{T, p} : = \sup_{(t,x)\in [0,T] \times \bR} \E\big[ | u(t,x) |^p \big] <\infty.
\end{align}

\noi
Because $\s$ is Lipschitz, this implies that
\begin{align}\label{KTPs}
K_{T,p,\s}:= \sup_{(t,x)\in [0,T] \times \bR} \E\big[ | \s( u(t,x) ) |^p \big] 
 \leq 2^{p-1} \big(  |\s(0)|^p +  \Lip^p( \s)   K_{T,p} \big),
\end{align}

\noi 
where  $\textup{Lip}(\s)$ denotes the  Lipschitz constant of $\s$.
For $h:\R^n\to\R$,
\begin{align}\label{def_Lip}
\textup{Lip}(h):=\sup\Big\{ \frac{|h(x)-h(y)|}{|x-y|}:  \, \text{ $x\neq y$ in $\R^n$} \Big\}.
\end{align}

Since the It\^o integral in \eqref{mild}
 has zero mean, $\bE[u(t,x)]=1$ for all $(t,x) \in \bR_+ \times \bR$. 
 For any $R>0$ and $t>0$, we consider the (centered) {\em spatial integral} of the solution:

\noi
\begin{align}\label{FRT}
F_R(t) =  \int_{-R}^R \big( u(t,x) - 1 \big) \, dx.
\end{align}

Let $\sigma_R^{2}(t)=\bE\big[F_R^2(t)\big]$ be the variance of $F_R(t)$.
The main goal of this article is to show that for any fixed $t>0$, 
$F_R(t)/\s_R(t)$ converges in distribution to $Z \sim \cN(0,1)$
as $R \to \infty$, 
where the rate of convergence can be   described using the Wasserstein distance.
We recall that the {\em Wasserstein distance} between two random vectors $F$ and $G$
 on $\R^n$ is given by
\begin{align}\label{WASS_def}
d_{\rm Wass}(F,G)=  \big|\bE[h(F)]-\bE[h(G)]\big|,
\end{align}

\noi
where the supremum  runs over all  Lipschitz functions $h:\bR^{n}\to \bR$ with $\textup{Lip}(h) \leq 1$ and $\textup{Lip}(h)$ defined as in \eqref{def_Lip}.

The study of asymptotic distribution of the spatial integrals \eqref{FRT}
was first performed in the work \cite{HNV}, which was partially motivated 
by the mathematical intermittency.  
In retrospect, one can proceed with the following heuristics.
For equations like \eqref{SWE},  the corresponding 
dynamics begin with the constant initial conditions and subject to 
the spatially homogeneous noises, and thus admits stationary solutions (if 
any). Then, a natural question to ask after the stationarity is that
\begin{center}
``{\it is the invariant $\s$-algebra generated by the spatial shifts trivial?}''
\end{center}
Or equivalently, one can ask whether or not the solution to \eqref{SWE} is {\it spatial ergodic.}
An affirmative answer to this question would lead to a law-of-large number  type result,
i.e., $F_R(t)/R \to 0$ almost surely and in $L^2(\PP)$ as $R$ tends to infinity. 
It is then very natural to investigate the second-order fluctuation of $F_R(t)$.
In the previous work \cite{BZ23}, we addressed these questions for the first
time for SPDEs driven by a L\'evy noise, and we focused
on the linear equation, with $\s(u) = u$ in \eqref{SWE}.
In the current paper, we work on the stochastic nonlinear wave equation \eqref{SWE}
with general nonlinearity.

\subsection{Main results}

We are now ready to present the first main result of this article.

\begin{theorem}
\label{thm_main} Recall the definition of $m_2$ from \eqref{m2}.
We assume 
 $m_2 \in(0,\infty)$ and $\s(1)\not=0$. 
 Then, the following statements hold true. 

%\smallskip
%\noi
%{\rm (i)} \textup{(\textsf{Malliavin differentiability})} For any $(t,x)\in\R_+\times\R$, %$u(t,x)\in\dom(D)$. 

\smallskip
\noi
{\rm (i)}  \textup{(\textsf{spatial ergodicity})} For any $t\in\R_+$, $\{u(t,x) :x\in \bR\}$ is strictly stationary and ergodic.  
Consequently, the following law of large numbers  holds:
\begin{align}\label{LLN}
\frac{F_R(t)}{R} \to 0 \quad \mbox{in $L^2(\PP)$ and almost surely as $R \to \infty$}.
\end{align}

\smallskip
\noi
{\rm (ii)}  \textup{(\textsf{limiting covariance})} For any $t,s\geq 0$,
\begin{equation}
\label{lim-cov}
\lim_{R \to \infty}\frac{1}{R}\bE[F_R(t)F_R(s)]=2\int_{\bR}{\rm Cov}\big(u(t,x),u(s,0) \big)dx=:K(t,s).
\end{equation}
Consequently, $\s_R^2(t)  \asymp R$ as $R\to\infty$. Moreover,  $\s_R^2(t) > 0$ for all $t>0$ and $R>0$.

\smallskip
\noi
{\rm (iii)}  \textup{(\textsf{QCLT})} Suppose that $m_4<\infty$. Then, for any $t>0$, there exists a constant $C_t>0$ depending on $t$, such that
\begin{equation}
\label{QCLT-rate}
d_{\rm Wass}\left( \frac{F_R(t)}{\s_R(t)}, \cN(0,1) \right) \leq  C_t R^{-1/2}.
\end{equation}

\smallskip
\noi
{\rm (iv)}  \textup{(\textsf{FCLT})} Suppose that $m_4<\infty$. Then,
the process $\{ \tfrac{1}{\sqrt{R}} F_R(t): t\in\R_+\}$
converges in law to a centered continuous Gaussian process $\cG=\{\cG(t)\}_{t\geq 0}$
with the covariance structure $K$ as in \eqref{lim-cov}.
%\begin{align}\label{Sigma}
%%\Sigma_{t,s} : = 
%\E\big[\cG(t) \cG(s) \big]= K(t,s).
%\end{align}

\end{theorem}

As mentioned above, the novelty of our result is the fact that we consider a general Lipschitz function $\s$. In the case $\s(u)=u$, the techniques of \cite{BZ23} rely heavily on the fact that the solution and its Malliavin derivative $Du$ have explicit   chaos expansions. This is no longer true when $\s$ is arbitrary. The techniques that we use in the present paper start with 
a careful analysis of the integral equation \eqref{eq_MD} satisfied by $Du$, paying special attention to some predictability issues that arise for giving a rigorous meaning to this equation.
As a result,   we substantially improve the main result of \cite{BN17}, which showed that $Du$ satisfies the integral equation \eqref{eq_MD} in the affine linear case $\s(u)=au+b$. 
 Next,  we move on to deriving the following {\em key estimate}:
\begin{equation}
\label{key}
\|D_{r,y,z}u(t,x)\|_p \leq C G_{t-r}(x-y)|z| 
\end{equation}
provided that $m_p<\infty$ for some $p\geq 2$
for $\fm$-almost every $(r,y,z)\in\mathbf{Z}$; see   \eqref{Z1} below. 
To prove \eqref{key}, we begin with the integral equation for $Du_n$ with   $u_n$  the $n$-th 
Picard iteration and then proceed with applications of  Rosenthal's inequality 
to obtain uniform-in-$n$ bound on $\|D_{r,y,z}u_n(t,x)\|_p$.
In order to pass $n\to\infty$ for obtaining \eqref{key},
we use a weak convergence argument and the technique of differentiating 
Radon measures; see Proposition \ref{prop_MD} for more details. 
Finally, we develop a variant of the Malliavin-Stein bound in the Poisson case (Theorem \ref{PSTU-new}), which together with the key estimate \eqref{key} yields the desired QCLT in 
Theorem \ref{thm_main}-(iii).

\begin{remark}
\rm In our previous work. we established the QCLT
under the weaker assumption $m_{p} + m_{p/2} <\infty$
for some $p\in(2, 4]$, when $\s(u)=u$.
As mentioned before, due to the linearity in $u$, we have explicit 
chaos expansions, from which we obtain bounds of Malliavin derivatives up to second order.
These bounds can be easily combined with a recent work of Trauthwein \cite{Tara}
so as to establish the QCLT under the said assumption. 
In the present paper, we are working with very general Lipschitz nonlinearity, 
then the solution $u(t,x)$ is only expected   to be Malliavin differentiable of first order. 
Even if we may assume more regularity on $\s$, our method of establishing \eqref{key}
would require us to work on   the recurrence equation satisfied by $D^2 u_n$,
which is much more complicated than that for $Du_n$ and cannot be easily iterated. 
For these reasons, we restrict ourselves to the assumption that $m_4 < \infty$
and develop a discrete Malliavin-Stein bound (Theorem \ref{PSTU-new})
that is tailor-made for our setting.  

\end{remark}

Soon after the appearance of our work \cite{BZ23},
together with P. Xia, we established an almost sure central limit theorem
in \cite{BXZ23}.  In the current setting,  we can establish the same result.

\begin{corollary}[\textsf{Almost sure central limit theorem}] 
Let the assumptions in Theorem \ref{thm_main} prevail
and assume $m_4 < \infty$. Then, for any fixed $t_0\in(0,\infty)$,
$\big\{ F_R(t_0)/\s_R(t_0) : R\geq 1\}$ satisfies an almost sure central limit theorem
meaning that for any bounded continuous function $\phi:\R\to\R$, 
we have 
\begin{align}\label{ASLT}
\frac{1}{\log T} \int_1^T \frac{1}{y} \phi \big(  F_y(t_0)/\s_y(t_0)    \big) dy
\xrightarrow[almost \,surely]{T\to+\infty}  \int_\R \frac{1}{\sqrt{2\pi}} e^{-z^2/2} \phi(z) dz.
\end{align}

\end{corollary}

\begin{proof}  It follows from Theorem \ref{thm_main} that 
$\s_y(t_0) \asymp \sqrt{y}$ as $y\to+\infty$. Then, following exactly
the same argument as in \cite[Section 3.1]{BXZ23} and using the key estimate \eqref{key},
we can prove the   almost sure central limit theorem in \eqref{ASLT}. 
We omit the details here. 
\qedhere
\end{proof}

For more details on the almost sure central limit theorem, we refer interested 
readers to \cite{BXZ23} and references therein. 

\smallskip

The second main contribution of this article
is an  asymptotic independence result for $(F_R(t), u )$
 as $R \to \infty$. This is motivated by  the recent developments   \cite{tudor23,TZ23} 
 in  the Gaussian case that stemmed from  an observation of L. Pimentel \cite{pimentel}.
 We present a brief account of this new development in Section \ref{SEC4_2}.
Now let us properly formulate the notion of asymptotic independence.
It is known, e.g., from \cite[Section 11.3]{Dudley}
that the set of bounded Lipschitz functions on $\R^k$
is a separating class for the weak convergence of probability
measures on $\R^k$.
Let $X, Y$ be two random variables with values in  
$\R^{k_1}, \R^{k_2}$ respectively,
and defined on a common probability space $(\O, \F, \PP)$.
Then, $X, Y$ are independent if and only if 
$\E[ f_1(X) f_2(Y) ]  = \E[ f_1(X) ] \E[f_2(Y)]$ for any bounded Lipschitz
functions $f_i: \R^{k_i} \to\R$.
The above condition is equivalent to
$\E[ e^{i  (\pmb{s} \cdot X +   \pmb{t} \cdot Y)}  ]  
=\E[ e^{i  \pmb{s} \cdot X }   ] 
\E[ e^{i     \pmb{t} \cdot Y}  ]   $
for any $\pmb{s}\in\R^{k_1}$ and $\pmb{t}\in\R^{k_2}$,
where $\pmb{s}\cdot X$ denotes the usual inner product
on the corresponding Euclidean space.
Furthermore, we say the families of random variables $(X_R), (Y_R)$
in $\R^{k_1}, \R^{k_2}$ (resp.) are {\it asymptotically independent} 
as $R\to+\infty$ if 
\begin{align} \label{adp}
{\rm Cov} \big( f_1(X_R),  f_2(Y_R) \big)   \xrightarrow{R\to+\infty} 0
\end{align}
for any bounded Lipschitz functions $f_i: \R^{k_i}   \to\R$, $i=1,2$.
With \eqref{WASS_def} in mind, one can easily see that 
the above condition
\eqref{adp} holds if 
\begin{align}\label{XYhat}
d_{\rm Wass}(  (X_R, Y_R), (\wh X_R, \wh Y_R)  ) \xrightarrow{R\to+\infty} 0,
\end{align}

\noi
where $(\wh X_R, \wh Y_R)$ has the same marginals as $ (X_R, Y_R)$
with   $\wh X_R,$ $\wh Y_R$   independent. 
See 
  \cite{DN21} for more on various notions of asymptotic independence.

 Now let us state our main result in this direction
 and postpone its proof to  Section \ref{SEC5_3}.

\begin{theorem}
\label{thm3}
 \textup{(\textsf{Asymptotic independence})} 
 Recall \eqref{mp} and assume  that $m_4<\infty$.
Then,
  for any $t>0$, $d\in\N_{\geq 1}$, and $(t_1,x_1),\ldots,(t_d,x_d) \in \R_{+} \times \R$, 
  there exists a finite  constant $C > 0$ depending on $t, d, t_1,\ldots,t_d$
   such that
 \begin{align}\label{qindep}
 d_{\rm Wass}\Big(  \big(\tfrac{F_R(t)}{\s_{R}(t)},   u(t_1,x_1),\ldots,u(t_d,x_d)  \big), 
 \big(Z', u(t_1,x_1),\ldots,u(t_d,x_d)\big)  \Big)  
 \leq C R^{-\frac12},
 \end{align}
where  $Z' \sim \cN(0,1)$ is independent of $\big(u(t_1,x_1),\ldots,u(t_d,x_d)\big)$.
Moreover,  any finite-dimensional marginal of the process $\{\tfrac{F_R(t)}{\sqrt{R} }\}_{t\in\R_+}$
is asymptotically independent 
from any finite-dimensional distribution of $u$, as $R\to+\infty$.
 \end{theorem}

\begin{remark}\rm
(i) The bound  \eqref{qindep}, together with a triangle inequality in the 
Wasserstein distance and the quantitative CLT in Theorem \ref{thm_main},
indicates that relation \eqref{XYhat} holds with $X_R=F_R(t)/\s_R(t)$ and $Y_R=(u(t_1,x_1),\ldots,u(t_d,x_d))$ and therefore the spatial integral $F_R(t)/\s_R(t)$ is asymptotically independent
from any finite-dimensional margins of the random field 
$\{u(t,x)\}_{(t,x)\in\R_+\times\R}$. 
Note that in the paper \cite{TZ23}, Tudor and Zurcher only established 
their asymptotic independence result for $F_R(t) / \s_{R}(t)$ and $u(t', x')$ (i.e., 
only for one-dimensional marginals); see Section \ref{SEC5_3} for more details.

\smallskip
\noi
(ii) Let $X_R = \frac{F_R(t)}{ \s_{R}(t)}$ for a fixed $t>0$ and
$Y_R=  (u(t_1,x_1),\ldots,u(t_d,x_d) )$. Then \eqref{adp} holds, by the remark above. Taking functions $f_1 \in \{x\mapsto \cos(s \cdot x), \sin(s\cdot x)\}$ and $f_2 \in \{x\mapsto \cos(t \cdot x), \sin( t \cdot x)\}$ in \eqref{adp}, we see that condition (9) of \cite{DN21} holds. By Theorem \ref{thm_main}, $(X_R)_{R>0}$ is a tight family. Then, Proposition 6 in \cite{DN21} indicates that the families $(X_R)_{R>0}$ and $(Y_R)_{R>0}$ satisfy the asymptotic independence condition {\bf AI-1} therein, that is,
$
\E[ h( X_R, Y_R)] - \E[ h( \wh X_R, \wh Y_R)] \to 0
$
for any $h: \R \times\R^{d}\to\R$ bounded uniformly continuous,
with $(\wh X_R, \wh Y_R)$  as  in \eqref{XYhat}.

\smallskip
\noi
(iii) To the best of our knowledge, 
the regularity of the solution to stochastic wave equation \eqref{SWE} 
with L\'evy noise $\dot{L}$ has not been addressed
in the literature; see, e.g., the paper \cite{CDH19} for the regularity result
on heat equations with L\'evy noise. We regard this as a separate 
research problem. 
For this reason, we could not 
establish the asymptotic independence between 
$\tfrac{F_R(t)}{\s_{R}(t)}\in\R$  and the random field $u\in\mathcal{M}$ 
on a suitable path space $\mathcal{M}$.  

\smallskip
\noi
(iv) In Theorem \ref{thm3}, we are only able to show the asymptotic independence 
for the finite-dimensional marginals of the spatial integral process 
$\{\tfrac{F_R(t)}{\sqrt{R} }\}_{t\in\R_+}$ and we can not pass the result 
to the whole process level.

\end{remark}

We end this section by collecting frequently used notations in this paper. 

\noi
$\bul$ {\bf Notations.} $\N_{\geq 1} := \{1,2, ...\}$ denotes
the set of positive integers. With   $G_t(x)$ as in \eqref{def-G},
we define 
\begin{equation}
\label{VTR}
\varphi_{t, R}(r, y) = \int_{-R}^R G_{t-r}(x-y)dx;
\end{equation}
%see, e.g., \eqref{VTR}.
  Let $\cB(\bR)$ be the Borel $\s$-field of $\bR$, and $\cB_b(\bR)$ the class of bounded sets in $\cB(\bR)$.
  Let $\cB(\bR_0)$ be the Borel $\s$-field of $\bR_0$, and $\cB_b(\bR_0)$ the class of bounded sets in $\cB(\bR_0)$.
  For any $p>0$, we let $\|X\|_p =(\bE[|X|^p])^{1/p}$ be the $L^p(\O)$-norm of $X$.
  For $a_R>0$ and $b_R>0$, the notation
$a_R\asymp b_R$ means that
$
0 < \liminf_{R\to+\infty} \frac{a_R}{b_R} \leq \limsup_{R\to+\infty}  \frac{a_R}{b_R} < +\infty.
$ 
  We denote by $\s\{X\}$ the $\s$-algebra generated
by the random object $X$. 
  We let $\|h\|_{\infty}=\sup_{x\in \bR^d}|h(x)|$ be the sup-norm of a
measurable  function $h:\bR^d \to \bR$.
  We write $\xi = (r, y,z)$ to denote a generic element in $\mathbf{Z} = \R_+\times\R\times\R_0$.

\smallskip

\noi
$\bul$ {\bf Organization of the paper.} 
This rest of the paper is organized as follows. Section \ref{section-prelim} includes some preliminaries about It\^o integration and Malliavin calculus in the Poisson setting, as well as
several useful facts on predictability. In Section \ref{section-Du}, we examine the Malliavin differentiability of the solution, and we include the proof of the key property \eqref{key}.
Section \ref{section-MS} gives a version of the Malliavin-Stein bound in the Poisson setting, and its extension to the multivariate case that is motivated 
by the desire of quantifying the asymptotic independence.
Section \ref{SEC5A} begins with the proof of the stationarity of the solution, and contains the  proofs of the main theorems (Theorem \ref{thm_main} and Theorem \ref{thm3}).

\section{Preliminaries}
\label{section-prelim}

 In this section, we collect a few preliminaries on It\^o integration
 and  Malliavin calculus in the Poisson setting. 
 When we apply It\^o-type 
 inequalities (e.g., Rosenthal's inequality), 
 the integrands are often required to be predictable.
 We address the predictability issue in Section \ref{SEC2_3}.

Recall from the Introduction that 
  $N$ is a Poisson random measure on $\mathbf{Z}=\R_+\times\R\times\R_0$
  that is generated by the pure-jump L\'evy white noise $\dot{L}$ in \eqref{SWE}.
In  Section \ref{SEC2_1}, we recall a few basics on 
the Poisson random measure $N$ and the associated 
It\^o integration;  we refer the reader to \cite[Section 4.2.1]{applebaum09}
 and \cite[Section 8.7]{PZ07} for more details.
 In Section \ref{SEC2_2}, we present basics on Malliavin calculus.
 In particular, we establish a variant of Heisenberg's commutation relation
 in Lemma \ref{lem_Heis}, which is crucial for us to derive
 the integral equation for the Malliavin derivative of the solution to \eqref{SWE}.

 \subsection{Basics on It\^o integration} 
 \label{SEC2_1}

Recall that $\mathbf{Z}=\bR_{+} \times \bR \times \bR_0$ is endowed with the Borel 
$\s$-field $\mathcal{Z}$ and the measure $\fm={\rm Leb} \times \nu$. 
The Poisson random measure $N$, over which 
our space-time pure-jump L\'evy white noise $\dot{L}$  is built,
is a set-indexed family $\{ N(A): A\in \mathcal{Z}\}$ of Poisson random variables. 
More precisely, let  $\mathbf{N}_\s$ be the set of all  $\s$-finite  measures $\chi$
on $(\mathbf{Z}, \mathcal{Z})$ with $\chi(B)\in \N_{\geq 0} \cup \{+\infty\}$
for each $B\in\mathcal{Z}$. Let $\mathscr{N}_\s$ be the smallest $\s$-algebra
that makes the mapping
$\chi\in \mathbf{N}_\s \mapsto \chi(B)\in[0,\infty]$
measurable for each $B\in\mathcal{Z}$. We recall the following definition.

\begin{definition}
\label{def:PRM}
A   Poisson random measure $N$ with intensity measure $\fm$ is a 
$(\mathbf{N}_\s,\mathscr{N}_\s)$-valued random element   
defined on a complete probability space 
$(\O, \F, \mathbb{P})$ such that the following two conditions hold:

\noi
\begin{itemize}
\item for each $A\in\mathcal{Z}$, $N(A)$ has a Poisson distribution  
with mean $\fm(A)$;\footnote{If $\fm(A)=\infty$,
we set $N(A)=\infty$ almost surely.}

\item  for any mutually disjoint sets $A_1, ... , A_k\in\mathcal{Z}$  {\rm(}$k\geq 1${\rm)}, 
$N(A_1), ... , N(A_k)$ are independent.

\end{itemize}
For $A\in\mathcal{Z}$ with $\fm(A)<\infty$, we define $\wh{N}(A) = N(A) - \fm(A)$. We
call $\wh{N}$  the  compensated Poisson random measure on $\mathbf{Z}$.

\end{definition}

\noi
$\bul$ {\bf A natural filtration.} We denote by $\mathcal{F}^0_t$ the $\s$-algebra
generated by the random variables
$N([0,s]\times A \times B)$ with $s \in [0,t]$ and 
$\textup{Leb}(A)+ \nu(B) < \infty$.
Let  $\mathcal{F}_t = \s\big( \mathcal{F}_t^0 \cup \mathcal{N} \big)$ be the
$\s$-algebra generated by $\mathcal{F}_t^0$ and  the set $\mathcal{N} $ of $\PP$-null sets.
This gives us a filtration
\begin{align}
\mathbb{F}:= \{ \mathcal{F}_t: t\in\R_+\}.
\label{filtraF}
\end{align}

\smallskip
\noi
$\bul$ {\bf It\^o integral with respect to $\wh{N}$.} 
Recall that a process $V=\{V(t,x,z):t\in \bR_{+},x\in \R,z\in \R_0\}$ is called {\em elementary},
 if it is of the form
\begin{equation}
\label{elem-V}
V(\o,t,x,z)=X(\omega) \ind_{(a,b]}(t) \ind_{B}(x)\ind_{C}(z), 
\end{equation}

\noi
where $0\leq a<b$, $B\in \cB_b(\R)$, $C \in \cB_b(\R_0)$, and $X$ is a bounded $\mathcal{F}_a$-measurable random variable.
The predictable  $\s$-field  $\overline{\cP}$ 
on $\Omega \times \bR_{+} \times \bR \times \bR_0$
is  generated by linear combinations of processes 
of form \eqref{elem-V}.\footnote{All  predictable $\s$-fields in this paper
 are considered with respect to the filtration $\mathbb{F}$ in \eqref{filtraF}.
  To simplify the language, we use the term ``predictable'' 
  instead of ``$\mathbb{F}$-predictable''.}

 For  an elementary  process $V$ as in \eqref{elem-V}, 
 we define its   It\^o integral
 as follows: for any $t\in\R_+$,
\[
I_t^{\wh N}(V):=\int_0^t \int_{\R}\int_{\R_0} V(s,x,z) \wh{N}(ds,dx,dz)
=X \cdot \wh{N}((t \wedge a,t\wedge b] \times B \times C).
\] 

\noi
Recall the notation  $\fH=L^2(\mathbf{Z}, \mathcal{Z},\fm)$. 
One can easily verify the {\em isometry property} for the above integral:

\noi
\begin{align}
\label{isometry}
\bE \big[|   I_t^{\wh N}  (V)|^2\big]=\bE \big[\|V \ind_{[0,t]} \|_{\fH}^2\big],
\end{align}

\noi
which  can be extended to all predictable process $V\in L^2(\Omega;\fH)$.
Note that  the process $\{I_t^{\wh N}(V)\}_{t\geq 0}$ is a zero-mean square-integrable 
$\mathbb{F}$-martingale that  has a c\`adl\`ag modification 
(still denoted by $I_\bul^{\wh N}(V)$). 
This   c\`adl\`ag modification  has 
%the following quadratic variation
%\[
%[I^{N}(V)]_{t}=\int_0^t \int_{\R} \int_{\R_0} V^2(s,x,z)N(ds,dx,dz),
%\]
%and it has
 the following (predictable) quadratic variation
\[
\langle I^{\wh N}(V) \rangle_{t}=\int_0^t \int_{\R} \int_{\R_0} V^2(s,x,z)  ds dx \nu(dz)
\]
and therefore,
$t\in\R_+\mapsto
 \big( I_t^{\wh N}(V) \big)^2 - \langle I^{\wh N}(V) \rangle_{t}$  is 
a (local) $\mathbb{F}$-martingale.

\smallskip

Let us record a basic property of It\^o's integral, which will be used in  
 the proof of Theorem \ref{thm_main}.

\begin{lemma} \label{Ito-int}
Let $s\in\R_+$ and $F$ be   $\F_s$-measurable.
Suppose  $V\in L^2(\Omega; \fH)$ is predictable, then  
the following equality holds almost surely:

\noi
\begin{align}
\label{Ito-int-F}
F \int_s^t \int_{\R} \int_{\R_0} V(r,x,z)\wh{N}(dr,dx,dz)
=
\int_s^t \int_{\R} \int_{\R_0} FV(r,x,z)\wh{N}(dr,dx,dz)
\end{align}
for any $t> s$. 
\end{lemma}

\begin{proof}
We only prove the result when $V$ is an elementary process. 
The general case will follow by approximation.
 Let $V$ be a process of the form \eqref{elem-V}
 and   $M_t=\wh{N}([0,t] \times A \times C\big)$.
 Then,  the left-hand side of \eqref{Ito-int-F} is equal to
\begin{equation}
\label{LHS-Ito}
FX \big[  (M_{t\wedge b}-M_{t \wedge a} )-
 (M_{s\wedge b}-M_{s \wedge a}  )\big].
\end{equation}
We consider three cases: 

\smallskip

\noi
(i) If $s\leq a$, then $FV$ is clearly  an elementary process.
Then,   the right-hand side of \eqref{Ito-int-F} coincides with \eqref{LHS-Ito}. 

\smallskip

\noi
(ii) If $s\in(a,b]$, then   $t\wedge a=s\wedge a=a$ and thus, 
we can rewrite  \eqref{LHS-Ito} as 
\[
F X  (M_{t\wedge b}-M_{s\wedge b} )=FX  (M_{t\wedge b}-M_{s} ).
\]
At the same time, $(r,x,z)\mapsto  FV(r, x, z)\ind_{(s, t]}(r) = FX\ind_{(s, b]}(r)\ind_B(x) \ind_C(z)$
is an elementary process,
so that the right-hand-side of \eqref{Ito-int-F} is almost surely equal to 
%$FY\ind_{(a,s]}(r)\ind_{B}(x) \ind_{C}(z)+FY\ind_{(s,b]}(r)\ind_{B}(x) \ind_{C}(z)$ and the first term vanishes if $r \in (s,t]$. So, the right-hand side of \eqref{Ito-int-F} equals
\begin{align*}
\int_0^t \int_{\R} \int_{\R_0}FX \ind_{(s,b]}(r)\ind_{B}(x) \ind_{C}(z)\wh{N}(dr,dx,dz)
&=FX (M_{t\wedge b}-M_{t\wedge s}) \\
&=F X (M_{t\wedge b}-M_{s}).
\end{align*}
That is, the equality \eqref{Ito-int-F} holds true in this case.

\smallskip

\noi
(iii) If $s >b$, then \eqref{LHS-Ito} and thus the left-side of  \eqref{Ito-int-F} 
is equal to zero, and it is also clear that the  right-side of \eqref{Ito-int-F} is  equal to $0$.

\smallskip

Hence, the proof is completed.
\qedhere
\end{proof}

\noi
$\bul$ {\bf It\^o integral with respect to $L$.}
We say that a process $g=\{g(t,x):t \in \bR_{+},x \in \bR\}$ is {\em elementary} if it is of the form:
\begin{equation}
\label{elem-X}
g(\omega,t,x)=X(\omega) \ind_{(a,b]}(t) \ind_{B} (x),
\end{equation}
with $a,b,B$ and $X$ as in \eqref{elem-V}. 
That is,
\[
(\o, t, x, z)\mapsto g(\o,t, x)\ind_C(z)
\]
is an elementary process as in \eqref{elem-V} for any $C\in\cB_b(\R_0)$.
We say that  $g$ is {\em predictable} if it is measurable 
with respect to the predictable $\s$-field $\widetilde{\cP}$ 
on $\Omega \times \bR_{+} \times \bR$, i.e., the $\s$-field generated by 
linear combinations of processes of form \eqref{elem-X}. 
%If $\cP$ denotes the classical predictable $\s$-field on $\O \times \bR_{+}$, then
%\[
%\widetilde{\cP}=\cP \times \cB(\bR) \quad \mbox{and} \quad \overline{\cP}=\widetilde{\cP} \times \cB(\bR_0).
%\] 
From the above discussion on predictability of random fields indexed by $\R_+\times\R$, 
we can easily define the It\^o integral with respect to the pure-jump L\'evy white noise by relating 
it to the It\^o integral with respect to $\wh{N}$:

\noi
\begin{align} \label{ito4L}
\int_{\R_+} \int_\R  g(s,x) L(ds,dx) := \int_{\R_+} \int_{\R} \int_{\bR_0} g(s,x)z \wh{N}(ds,dx,dz)
\end{align}
for any  predictable random field  $g$ satisfying 
\[
\E \bigg[ \int_{\R_+\times\R} | g(s, x) |^2 dsdx \bigg] < \infty.
\]

\noi
From \eqref{isometry} and   \eqref{m2}, it follows that the above integral (with respect to $L$)
satisfies a modified isometry property:
 
\noi
\begin{align}\label{m_iso}
\bE\Big[ \big|\int_{\R_+\times \R}g(s,x)L(ds,dx) \big|^2\Big] 
= m_2\cdot  \bE \Big[\int_{\R_+\times\R} |g(s,x)|^2 \, dsdx \Big].
\end{align}

\begin{remark}\label{REM23}
\rm
(i) The It\^o integral \eqref{ito4L} 
with respect to $L$ coincides with the stochastic integral with respect to the  
{\em martingale measure} $\{L_t(B): t \in\R_+,B\in \cB_b(\R)\}$
with $L_t(B)=L([0,t] \times B)$.
We refer interested readers to \cite{walsh86}
 for stochastic integration with respect to martingale measures.

\smallskip

\noi
(ii) For any $B \in \cB_b(\R)$, $\{L_t(B)\}_{t\geq 0}$ is  an $\mathbb{F}$-martingale and a zero-mean L\'evy process with L\'evy measure $\nu$ and variance
$
{\rm Var} (L_t(B)   )  =m_2 t \, {\rm Leb}(B).
$
%By \cite[Theorem 2.1.8]{applebaum09}, this process has a c\`adl\`ag modification. 
%We work with this modification.  

\smallskip

\noi
(iii) The martingale measure $\{L_t(B): t\in\R_+, B\in \cB_b(\R)\}$
  is {\em spatially homogeneous}  in Dalang's sense (see, e.g., \cite{dalang99}).
  This is shown  in  Lemma \ref{spat-homog} below.

\end{remark}

\subsection{Basic Malliavin calculus} \label{SEC2_2}

In this section, we recall some basic facts about Malliavin calculus in the Poisson setting. 
 We refer the readers to \cite{PSTU10, Last16,LPS16,LP18, NN18} for more details.

\smallskip

Let us begin with the well-known Wiener-Poisson-It\^o chaos decomposition
that asserts that the space $L^2(\O,\s\{N\},\mathbb{P})$ can be 
decomposed into a direct sum of mutually orthogonal subspaces (called 
Poisson chaoses); see, e.g., \cite[Section 18.4]{LP18}
and \cite[Section 9.9]{NN18}.

\smallskip

\noi
$\bullet$ {\bf Poisson chaos expansion.}
Any random variable $F\in L^2(\O,\s\{N\},\mathbb{P})$ can be written as:
\begin{align}
 F = \E[F] + \sum_{n=1}^\infty I_n( f_n),
 \label{CD3}
 \end{align}
 where $f_n\in \fH^{\odot n}$, $n\in\N_{\geq 1}$, are uniquely determined by $F$
 up to a $\fm$-null set.
 Here, $\fH^{\odot n}$ is the set of symmetric elements in $\fH^{\otimes n}$.
% see also \cite[Section 4]{Last16}. 
Let $J_n(\bul)$ denote the projection operator from $L^2(\O,\s\{N\},\mathbb{P})$ 
onto the $n$-th Poisson-Wiener chaos $\C_n=\{ I_n(f): f\in\fH^{\odot n}\}$, 
i.e., for $F$ as in \eqref{CD3},

\noi
 \begin{align}\label{def_JK}
 J_0 F= \E[ F] \quad{\rm and}\quad J_n F =  I_n( f_n), \,\,\, n\in\N_{\geq 1}.
 \end{align}

\noi
As in the Gaussian case (\cite[Proposition 2.7.5]{blue}), 
the Poisson chaoses are mutually orthogonal with 

\noi
\begin{align}\label{orth_rel}
\E[ I_n(f) I_m(g)] = n! \ind_{\{n=m\}} \langle f, g \rangle_{\fH^{\otimes n}}
\end{align}
for any  $f\in\fH^{\odot n}$ and $g\in\fH^{\odot m}$ with $m,n\in\N_{\geq 1}$.
As a consequence, we have, with $F$ as in \eqref{CD3}, that 
 \begin{align}
{\rm Var}(F) = \sum_{n\geq 1} {\rm Var}( J_n F) =
  \sum_{n=1}^\infty n! \| f_n \|^2_{\fH^{\otimes n}} < \infty.
 \label{CD4}
 \end{align}

\smallskip

\noi
$\bul$ {\bf Malliavin derivative.} Let $\dom(D)$ denote  the set of random variables
$F$ as in \eqref{CD3} with the  symmetric kernels $\{f_n\}_n$ satisfying
\begin{align}
\sum_{n=1}^\infty n! n \| f_n\|^2_{\fH^{\otimes n}} < \infty.
\notag %\label{CD5}
\end{align}

\noi
For such a random variable $F\in\dom(D)$, we define the {\em Malliavin derivative}
$DF$ of $F$ to be a $\fH$-valued random variable  given by
\begin{align} \notag %\label{CD5b}
D_\xi F = \sum_{n=1}^\infty D_\xi J_n F    
= \sum_{n=1}^\infty n I_{n-1}(f_n(\xi, \bul) ), \,\, \, \xi\in \mathbf{Z},
\end{align}

\noi
where for fixed $\xi\in \mathbf{Z}$, $f_n(\xi, \bul)\in\fH^{\odot (n-1)}$. 
In view of \eqref{orth_rel} and Fubini's theorem, 
we can easily derive that for $k\in\N_{\geq 1}$,

\noi
\begin{align} \label{lem_XY}
\begin{aligned} 
\E\big[ \|D_\xi J_k F \|^2_\fH \big]
&=   k^2 (k-1)!  \int_{\mathbf{Z}} \|f_k(\xi, \bul)\|^2_{\fH^{\otimes (k-1)}}\fm(d\xi) \\
&= k k! \| f_k\|^2_{\fH^{\otimes k}} = k   \E\big[ \| J_k F \|^2_2  \big].  
\end{aligned}
\end{align}

\noi
As a result, we have 

\noi
\begin{align}  \label{lem_XY2}
\E\big[ \| DF \|_{\fH}^2  \big] 
 &=  \sum_{n=1}^\infty  \E\big[ \| DJ_n F \|_{\fH}^2  \big]  =  \sum_{n=1}^\infty  n \E\big[ \| J_n F \|^2_2  \big]  
= \sum_{n=1}^\infty n n!  \| f_n\|^2_{\fH^{\otimes n}} <\infty.
\end{align}
 
 \noi
Comparing \eqref{lem_XY2} with \eqref{CD4} yields
the following Poincar\'e inequality: % [\textsf{\small needed for spatial ergodicity}]
\begin{align}
{\rm Var}(F) \leq \E\big[ \| DF \|_{\fH}^2  \big]
\label{Poi1}
\end{align}
for any $F\in\dom(D)$,
with equality when and only when $F - \E[F] \in \C_1$.

\smallskip

As we only study distributional properties,
we can assume that the Poisson
random measure $N$ (from Definition \ref{def:PRM}) is a 
proper simple point process of the form
\[
N = \sum_{n=1}^\kappa \dl_{Z_n},
\]
where $\{Z_n\}_{n\geq 1}$ are independent random variables with values in $ \mathbf{Z}$,  
$\kappa$ is a
random variable with values in $\N_{\geq 1}\cup\{+\infty\}$,
and $\dl_\xi$ is the Dirac mass at $\xi \in  \mathbf{Z}$. With probability one, 
these points are distinct (since $\fm$ 
is diffusive); see, e.g.,  \cite[Corollary 3.7]{LP18}.
For any real $\s\{N\}$-measurable random variable $F$, by Doob-Dynkin lemma, 
we can write $F = \ff(N)$ for some $\scrN_\s$-measurable function $\ff: \bfN_\s\to \R$. We define 
the {\em add-one cost operator} by

\noi
\begin{align}\label{addcost}
D^+_\xi F :=  \ff(N+\dl_\xi) - \ff(N).
\end{align}
 
The following lemma asserts  that the add-one cost operator coincides 
with the Malliavin derivative on $\dom(D)$.
For the proof of this result, we refer to \cite[Theorem 3]{Last16} for part (i)
and refer to \cite[Lemma 3.1]{PT13} for part (ii).
 
\begin{lemma}
\label{rem_BZ-ii}
{\rm (i)}  If $F\in\dom(D)$, then $D^+F = DF$.

\smallskip
\noi
{\rm (ii)} If $\E [\| D^+_\bul F\|_{\fH}^2 ] < \infty$, then 
$F\in\dom (D)$ and $DF = D^+F$.
\end{lemma}

In the Poisson case, the Malliavin derivative satisfies a ``{\it difference rule}'', instead of the chain rule encountered in the Gaussian case (e.g., \cite[Proposition 2.3.8]{blue}). 
This difference rule is stated in the following.

\begin{lemma} \label{chain}
Let $F_1,\ldots,F_d \in \dom(D)$ and $\phi:\bR^d \to \bR$ be a Lipschitz  function.
Then $\phi(F_1,\ldots,F_d) \in \dom(D)$ and

\noi
\begin{align} \notag % \label{chain_eq}
D_\xi \phi(F_1,\ldots,F_d)=\phi(F_1+D_\xi F_1,\ldots, F_d+D_\xi F_d)-\phi(F_1,\ldots,F_d)
\end{align}
for $\fm$-almost every $\xi\in\mathbf{Z}$.

\end{lemma}

\begin{proof}  
We write $F_i=\ff_i(N)$ for some $\scrN_\s$-measurable function $\ff_i$,
 $i=1,\ldots,d$. Then, we can write
 $\phi(F_1,\ldots,F_d)= \phi \circ (\ff_1,\ldots,\ff_d)(N)$, and  then
 we can deduce from \eqref{addcost} and Lemma \ref{rem_BZ-ii} that 
 
 \noi
\begin{align*}
D_{\xi}^+\phi(F_1,\ldots,F_d)
&= \phi \circ (\ff_1,\ldots,\ff_d)(N+\delta_{\xi})
   -  \phi \circ (\ff_1,\ldots,\ff_d)(N)\\
&=\phi\big(\ff_1(N+\delta_{\xi}), \ldots,\ff_d(N+\delta_{\xi})\big)-\phi\big(\ff_1(N),\ldots,\ff_d(N) \big)\\
&=\phi\big(F_1+D_{\xi}F_1,\ldots, F_d+D_{\xi}F_d\big)-\phi(F_1,\ldots, F_d).
\end{align*}

\noi
Using the Lipschitz property of $\phi$, we have
\begin{align}\label{2_LIP}
|D_{\xi}^+\phi(F_1,\ldots,F_d)|^2
 \leq 
\textup{Lip}^2(\phi)\sum_{j=1}^{d} |D_{\xi}F_j|^2,
\end{align}
 and thus,
 
 \noi
\begin{align} \notag %\label{2_LIP}
\bE \left[\int_{\mathbf{Z}}|D_{\xi}^+\phi(F_1,\ldots,F_d)|^2\fm(d\xi)\right] 
\leq \textup{Lip}^2(\phi)\sum_{j=1}^{d} \bE[\|DF_j\|_{\fH}^2]<\infty.
\end{align}

\noi
By Lemma \ref{rem_BZ-ii}-(ii), $\phi(F_1,\ldots,F_d) \in \dom(D)$ and
$D^{+}_\bul \phi(F_1,\ldots,F_d)=D_\bul \phi(F_1,\ldots,F_d)$.
 \qedhere

\end{proof}

%{\color{orange}
%$\GZ$: the proof of \eqref{chain_eq} is clear, but why in case (ii), 
%one gets $\phi(F_1, ... , F_d)\in\dom(D)$?
%In this paper, we only need to consider the case (i), for which
%the result is clear ?? }

\begin{remark}
 \label{rem_BZ-iii}
 \rm
If $F  \in\dom(D)$ and $\phi: \R\to\R$ is Lipschitz, then Lemma \ref{chain} implies that
 $\phi(F)\in\dom(D)$ with
 
 \noi
\begin{align}   \notag %\label{chain2}
D_\xi \phi(F)=\phi(F+D_\xi F)-\phi(F) 
\end{align}

\noi
and 
\begin{align}\label{add1b}
|D_\xi \phi(F)| \leq  \Lip(\phi)| D_\xi F|.
\end{align}
Combining this inequality with \eqref{Poi1}, we obtain a generalization of \eqref{Poi1}:
\begin{align}   \notag %\label{Poi1b}
\Var  ( \phi(F)  ) \leq  \bE[\|D \phi(F)\|_{\fH}^2]\leq \textup{Lip}^2(\phi) \E [ \| DF \|_\fH^2  ].
\end{align}
We refer  interested readers to \cite[Remark 2.7]{BZ23} for more details.

%\textcolor{blue}{I moved the inequality $|D_\xi \phi(F)  - \phi'(F) D_\xi F| \leq   \tfrac{1}{2}  %\Lip(\phi')( D_\xi F)^2$ in the proof of Theorem \ref{PSTU-new}.}

%If in addition, $\phi$ is differentiable and $\phi'$ is Lipschitz, then
%\begin{align}\
%\label{add1c}
%|    D_\xi \phi(F)  - \phi'(F) D_\xi F| \leq   \tfrac{1}{2}  \Lip(\phi')( D_\xi F)^2.
%\end{align}
%To see this, note that using \eqref{chain2} and fundamental theorem of calculus, we have:
%\begin{align*}
% D_\xi \phi(F)  - \phi'(F) D_\xi F
% &= \phi(F + D_\xi F) - \phi(F)  - \phi'(F) D_\xi F \\
% &= (D_\xi F) \, \int_0^1 \Big( \phi'(F + t D_\xi F) -  \phi'(F) \Big)  dt,
%\end{align*}
%and the above expression is bounded by $ \tfrac{1}{2}  \Lip(\phi') (D_\xi F)^2$.

\end{remark}

We also record below a useful fact. 

\begin{lemma}

\label{rem_BZ-iv}
Let $F\in\dom(D)$ with the chaos expansion \eqref{CD3} with $f_n \in \fH^{\odot n}$.
 Then $F$ is $\mathcal{F}_t$-measurable if and only if for any $n\geq 1$, 
 \begin{equation}
 \label{fn-zero}
 f_n( t_1, x_1, z_1, ..., t_n, x_n, z_n) = 0 \quad\text{if $t_j > t$ for some $j=1,\ldots,n$.}
 \end{equation}
 %Moreover, if $F$ is $\mathcal{F}_t$-measurable, then 
In this case,  $D_\xi F$ is $\mathcal{F}_t$-measurable for any $\xi \in \mathbf{Z} $, and

\noi
\begin{align} \label{D-zero}
D_{r,x,z}F=0  \,\, \mbox{almost surely} \,\, 
\mbox{for $\fm$-almost all $(r,x,z) \in (t,\infty) \times \R \times \R_0$}.
\end{align}

\end{lemma}

\begin{proof}
Clearly, $F$ is $\mathcal{F}_t$-measurable if and only if $F=\E[F|\mathcal{F}_t]$. 
By \cite[Lemma 2.5-(i)]{BZ23}, $\E[I_n(f_n)|\mathcal{F}_t]=I_n(f_n^t)$, where
\[
f_n^t(\xi_1,\ldots,\xi_n)=f_n(\xi_1,\ldots,\xi_n)\ind_{[0,t]^n}(t_1,\ldots,t_n)
 \,\, \mbox{with $\xi_i=(t_i,x_i,z_i)$}.
\]
Therefore, $F$ is $\mathcal{F}_t$-measurable if and only if 
$\sum_{n\geq 1}I_n(f_n)=\sum_{n\geq 1}I_n(f_n^t)$, 
which is equivalent to saying that $f_n=f_n^t$ for all $n\geq 1$, by the uniqueness of the chaos expansion. Note that $f_n=f_n^t$ is equivalent to \eqref{fn-zero}, due to the symmetry of $f_n$.

Next, we  prove that $D_{\xi}F$ is $\mathcal{F}_t$-measurable.
 Recall first  that $D_\xi F = \sum_{n\geq 1} n I_{n-1} ( f_n(\xi, \bul))$.
By the construction of multiple integrals, one can approximate the kernel $f_n(\xi, \bul)$ by 
simple functions, i.e., linear combination of indicators of the form 
$\ind_{A_j\times B_j\times C_j}$
with $A_j\subset [0,t]$, $B_j\times C_j\subset \R\times\R_0$, so that 
$I_{n-1} ( f_n(\xi, \bul)   )$ is the $L^2(\O)$-limit of multilinear polynomials 
in $\wh{N}(A_j\times B_j\times C_j)$, which are clearly $\mathcal{F}_t$-measurable;
see   (2.12), (2.13), and (2.16) in \cite{BZ23}. 

Finally, relation \eqref{D-zero} was proved in \cite[Lemma 2.5-(ii)]{BZ23}.
\end{proof}

In the following, we recall from \cite[Proposition 2.10]{BZ23}
the following result, which is a consequence of the Rosenthal's inequality.
 
 \begin{proposition}\label{prop_Rose}
 Let $\{\Phi(s,y)\}_{ (s, y) \in \R_+ \times \R}$ be a predictable process such that

\noi
\begin{align}  \notag% \label{Rosen1}
\bE \bigg[ \int_0^t \int_{\R}G_{t-s}^2(x-y)|\Phi(s,y)|^2 dyds \bigg] <\infty,
\end{align}

\noi
where $G$ is given as in \eqref{def-G}.
 Suppose that the condition \eqref{mp} holds  for some finite $p\geq 2$. Then,

\noi
\begin{align}
\begin{aligned}
&\bE \bigg[   \Big| \int_0^t \int_{\R}G_{t-s}(x-y)\Phi(s,y)L(ds,dy) \Big|^p \bigg]  \\
& \qquad\quad
\leq C_{p}(t) \int_0^t \int_{\R}G_{t-s}^p(x-y)\bE \big[  |\Phi(s,y)|^p \big] dsdy,
\end{aligned}
 \notag% \label{Rosen2}
\end{align}
where $C_{p}(t)=2^{p-1}B_p^p\big( m_2^{\frac p2} t^{p-2} + m_p \big)$
with  $B_p$   the constant in the  Rosenthal's inequality.
 \end{proposition}

\smallskip

\noi
$\bul$ {\bf Kabanov-Skorohod integral $\dl$.} This is an adjoint operator of $D$,
characterized by the following duality relation:

\noi
\begin{align}
\E[ \langle DF, V \rangle_\fH  ] = \E[ F \dl(V) ] 
\label{dualR}
\end{align}

\noi 
for any $F\in\dom(D)$. In other words, the domain $\dom(\dl)$
is the set of random vectors $V\in L^2(\O;\fH)$ satisfying the following
property:
there is some finite constant $C= C(V) >0$ such that 
for any $F\in\dom(D)$, $| \E\jb{DF, V}_\fH  | \leq C \| F\|_2$.
It is clear that the   relation \eqref{dualR} holds for any 
$F\in\dom(D)$ and for any $V\in\dom(\dl)$.
It is also clear that   $\dom(\dl)$ is  dense in $L^2(\O; \fH)$.  
See also Lemma \ref{lem_dl} below for a characterization of $\dom(\dl)$
using the chaos expansions. 

\smallskip

We recall a few properties of the Kabanov-Skorohod integral
from  \cite[Lemma 2.5]{BZ23}.

\begin{lemma}\label{lem25}

{\rm (i)}   
 Suppose $F\in\dom(D)$ is $\mathcal{F}_t$-measurable
for some fixed $t \in(0,\infty)$.
Then, the following Clark-Ocone formula holds:
\begin{align}
 \notag%\label{Clark-Ocone}
F = \E[F] + \dl(V),
%\notag
\end{align}

\noi
where $  (r,y,z)\in \mathbf{Z} \mapsto  V(r, y, z) := \E\big[ D_{r, y, z} F | \mathcal{F}_r \big]$ 
belongs to $\dom(\dl)$.

\smallskip
\noi
{\rm (ii)}  Suppose $V\in L^2(\Omega; \fH)$ is  predictable.
 Then,
$V\in\dom(\dl)$ and
$\dl(V)$ coincides with the  It\^o integral of $V$ against the compensated Poisson random
measure $\wh{N}$:
\begin{align}  \notag% \label{EXT0}
\dl(V) = \int_0^\infty \int_{\R} \int_{\R_0} V(r, x, z) \wh{N}(dr, dx, dz).
\end{align}

\smallskip
\noi
{\rm (iii)} \textup{(Skorohod isometry)} Suppose  $V\in L^2(\Omega; \fH)$ satisfies

\noi
\begin{align}\label{intD}
\E\int_{\mathbf{Z}^2} \big| D_{\xi_0} V(\xi) \big|^2 \fm(d\xi) \fm(d\xi_0) <\infty.
\end{align}

\noi
Then,  $V\in\dom(\dl)$  and 
\begin{align}\label{Skiso}
\E\big[ \dl(V)^2 \big] = \| V\|^2_{L^2(\O;\fH)} 
+ \E\int_{\mathbf{Z}^2} \big(  D_\xi V(\xi_0) \big)\big(  D_{\xi_0} V(\xi) \big) \fm(d\xi_0) \fm(d\xi).
\end{align}

\noi
Moreover, if $V_1, V_2 \in L^2(\O;\fH)$ satisfy \eqref{intD}, then
\begin{align} \label{Skiso_1}
\begin{aligned}
\E\big[ \dl(V_1)\dl(V_2) \big] = \langle V_1,V_2 \rangle_{L^2(\O;\fH)} 
&+ \frac{1}{2}\E\int_{\mathbf{Z}^2} \big(  D_\xi V_1(\xi_0) \big)\big(  D_{\xi_0} V_2(\xi) \big) 
\fm(d\xi_0) \fm(d\xi) \\   
& + \frac{1}{2}\E\int_{\mathbf{Z}^2} \big(  D_\xi V_2(\xi_0) \big)\big(  D_{\xi_0} V_1(\xi) \big) \fm(d\xi_0) \fm(d\xi).
\end{aligned}
\end{align}

\smallskip
\noi
{\rm (iv)} \textup{(It\^o isometry)} 
Let $V$, $V_1$, and $V_2$ be as in {\rm (iii)}. If they are 
 additionally   $\mathbb{F}$-adapted, then \eqref{Skiso} and \eqref{Skiso_1}
 reduce to 
 
\noi
\begin{align}
\begin{aligned} \label{Skiso_2}
\E\big[ \dl(V)^2 \big] &= \bE\big[\| V\|^2_{\fH}\big] \\
\E\big[ \dl(V_1)\dl(V_2)  \big] &= \bE\big[\langle V_1,V_2 \rangle_{\fH}\big].
\end{aligned}
\end{align}
\end{lemma}

Note that if $V$ is $\mathbb{F}$-adapted (not assumed to be  predictable) and   \eqref{intD} holds, 
$\dl(V)$ is {\it still} well-defined in the Skorohod sense
but not necessarily  in the It\^o sense, although we have the It\^o isometry.
%see Remark  \ref{rem_predict} for more discussions. 

 \begin{proof}[Proof of Lemma \ref{lem25}]
 Note that part (i) and part (ii) are taken from \cite[Lemma 2.5]{BZ23}.
 For the obtention of \eqref{Skiso}, one can refer to \cite[Theorem 5]{Last16}.
 Relation \eqref{Skiso_1} follows from \eqref{Skiso} by polarization, i.e.,
using the identity $4ab=(a+b)^2-(a-b)^2$.

To prove (iv), we note that for 
  $V$  being    $\mathbb{F}$-adapted, 
\[
\fm\otimes\fm\big\{ (\xi_0, \xi) \in\mathbf{Z}^2: 
 \big(  D_\xi V(\xi_0) \big)\big(  D_{\xi_0} V(\xi) \big)  \neq 0 \big\}
=0
\]
in view of Lemma \ref{rem_BZ-iv}, 
so that the second term in \eqref{Skiso} vanishes. 
This implies that  the first equality in  \eqref{Skiso_2},
whilst  the second equality in   \eqref{Skiso_2}
follows by polarization.
\qedhere

\end{proof}

The following result  is a Poincar\'e-type inequality
and  will be used when we establish the limiting covariance structure 
of $F_R(t)$ in \eqref{FRT}.
Recall from \eqref{def_JK} that $J_k$ denotes the projection operator onto the 
$k$-th Wiener chaos $\C_k$  and we will use this notation 
in the following two lemmas.

\begin{lemma}[Poincar\'e inequality]
\label{2-Poincare}
For any $F,G \in \dom(D)$, we have:
\[
|{\rm Cov} (F,G)| \leq \int_0^{\infty}\int_{\bR}\int_{\bR_0}\|D_{r,y,z}F\|_2\|D_{r,y,z}G\|_2\, drdy \nu(dz).
\]
\end{lemma}

\begin{proof} 
One can  apply  the same argument as in the Gaussian case (\cite[Proposition 6.3]{CKNP21})
that is based on the Clark-Ocone formula. 
Here, we  {\it sketch}  a different proof that would {\it not} need the It\^o structure and may be 
of independent interest.  

With $J_k$ as in \eqref{def_JK}, 
we define 
\[
L F = \sum_{k=1}^\infty -k J_k F \quad {\rm and}\quad
L^{-1} F = \sum_{k=1}^\infty -\frac{1}{k} J_k F,
\]
whenever the above series are well defined in $L^2(\O)$. The operators $L$ and $L^{-1}$
are called Ornstein-Uhlenbeck operator and its pseudo-inverse, while the associated
semigroup is given by $P_t = e^{tL}$ and it satisfies 
the contraction property on $L^2(\O)$:
\[
\|P_t F'\|_2 \leq \|F'\|_2
\]
 for any $F'\in L^2(\Omega)$, which can be easily verified by using the chaos expansion. 
Using the identity $L = -\dl D$, 
we can write  $G=\delta(-DL^{-1}G)$ for a centered $G\in\dom(D)$; 
see, e.g., \cite{PSTU10} for more details
on these facts.

Without loss of generality, we assume that $G$ is centered. Then,
using \eqref{dualR} with  $G=\delta(-DL^{-1}G)$, 
we have 
 \[
 {\rm Cov}(F, G) = \E[  F \delta(-DL^{-1}G)    ] = \E[ \langle DF, - DL^{-1}G \rangle_{\fH}],
 \] 
 and thus, 
 
 \noi
  \begin{align} \label{CFG1}
  |{\rm Cov}(F, G)|\leq \int_{\mathbf{Z}} \|D_{\xi}F\|_2 \| D_{\xi}L^{-1}G\|_2 \fm(d\xi).
  \end{align}
  
  \noi
 Note that one can deduce from the chaos expansion that $-DL^{-1}G = \int_0^\infty e^{-t} P_t DGdt$
 and therefore, using Minkowski's inequality and the above contraction property,
 we can obtain 
 $\| D_{\xi} L^{-1}G \|_2 \leq \|D_{\xi} G\|_2$.
 Finally,   the conclusion follows from \eqref{CFG1} immediately.
 \qedhere

\end{proof}

In what follows, we present a useful criterion on the Malliavin differentiability.

\begin{lemma} \label{lem_1}
Let $\mathcal{A}$ be a nonempty index set and let $\{F_{n, \al}, F_\al : n\in\N_{\geq 0}, \al\in\mathcal{A}\}$
be a family of random variables in $L^2(\O,\s\{N\}, \PP)$.

{\rm (i)}
Suppose $F_{n,\al}$ converges uniformly in $L^2(\O)$ to $F_\al$  as $n\to\infty$:

\noi
\begin{align}  \notag% \label{unicA1}
\lim_{n\to\infty}\sup_{\al\in \mathcal{A}} \| F_{n,\al} - F_\al \|_2  = 0.
\end{align}

\noi
 Then for any $k\in\N_{\geq 0}$,
we have  $D J_k F_{n,\al}$   converges uniformly 
in $L^2(\O; \fH)$ to $DJ_k F_\al$:

\noi
\begin{align}   \notag% \label{unicA2}
\lim_{n\to\infty}\sup_{\al\in \mathcal{A}} \| DJ_k F_{n,\al} - DJ_k F_\al \|_{L^2(\O;\fH)}  = 0.
\end{align}

\smallskip
\noi
{\rm (ii)} 
Assume additionally that  $F_{n,\al}\in\dom(D)$ satisfies 

\noi
\begin{align}\label{MDB1}
\sup_{\al\in\mathcal{A}} \sup_{n\geq 1} \| DF_{n, \al}\|_{L^2(\O;\fH)} <\infty. 
\end{align}

\noi
Then, $F_\al\in\dom(D)$ for each $\al\in\mathcal{A}$ 
with
\begin{align}\label{MDB1b}
\sup_{\al\in\mathcal{A}}  \| DF_{\al}\|_{L^2(\O;\fH)} <\infty. 
\end{align}

\noi
 Moreover, for each $\al\in\mathcal{A}$,
 we have  $DF_{n,\al}$ weakly converges to $DF_\al$ in $L^2(\O;\fH)$
as $n\to\infty$. 

\end{lemma}

\begin{proof}  
The convergence in part (i) follows immediately from 
the fact that $\E\big[ \| DJ_k F \|_\fH^2 \big] = k \E\big[ ( J_k(F) )^2 \big]$
for any $F\in L^2(\O)$ (see, e.g., \eqref{lem_XY}):

\noi
\begin{align*}
\E\big[ \| DJ_k (F_{n,\al} - F_\al )\|^2_\fH  \big]
= k \| J_k (F_{n,\al} - F_\al )\|^2_2  \leq k \| F_{n,\al} - F_\al \|^2_2 \xrightarrow{n\to\infty} 0.
\end{align*}

\smallskip

Now let us prove (ii).\footnote{The 
usual proof (as done in the classic book \cite{Nua06})
starts from extracting a subsequence. }
Recall that $F\in\dom(D)$ if and only if 
\[
 \sum_{k\geq 1} k  \|  J_k F \|_2^2 < \infty,
\]
which is also equivalent to $ \sum_{k\geq 1}   \| D J_k F \|^2_{L^2(\O;\fH)} < \infty$.
Meanwhile, we can deduce from part (i) and the condition \eqref{MDB1},
together with Fatou's lemma,
 that

\noi
\begin{align*}
\sup_{\al} \| DF_\al \|^2_{L^2(\O;\fH)} 
 &=  \sup_{\al} \sum_{k\geq 1}   \| D J_k F_\al \|^2_{L^2(\O;\fH)}  
 =  \sup_{\al}  \sum_{k\geq 1}  \lim_{n\to\infty}  \| D J_k F_{n,\al} \|^2_{L^2(\O;\fH)}   \\
 &\leq \sup_{\al} \liminf_{n\to\infty} \sum_{k\geq 1}  \| D J_k F_{n,\al} \|^2_{L^2(\O;\fH)}   
 = \sup_{\al, n}  \| D  F_{n,\al} \|^2_{L^2(\O;\fH)}   < \infty.
\end{align*}

\noi
That is, we just proved that $F_\al\in\dom(D)$ and \eqref{MDB1b}. 
Next, we need to show the weak convergence
of $DF_{n, \al}$ to $DF_\al$ in $L^2(\O;\fH)$: as $n\to\infty$,
\begin{align}\label{MDB2}
\E\big[ \jb{ DF_{n,\al}, V}_\fH \big] \to \E\big[ \jb{ DF_\al, V}_\fH \big] 
\end{align}
for any $V\in L^2(\O; \fH)$. Note first that the above convergence \eqref{MDB2} holds for
any $V\in\dom(\dl)$. Indeed, using the relation \eqref{dualR},
we have 

\noi
\begin{align*}
\E\big[ \jb{ DF_{n, \al}, V}_\fH \big]
= \E\big[ F_{n_\al} \dl(V) \big]  
&\xrightarrow{n\to\infty}   \E\big[ F_\al \dl(V) \big] 
\quad \text{(as $F_{n,\al}\to F_\al$ in $L^2(\O)$)} \\
&  = \E\big[ \jb{ DF_\al, V}_\fH \big] .
\end{align*}
The convergence \eqref{MDB2} for any $V\in L^2(\O; \fH)$
follows from the density of $\dom(\dl)$ in $L^2(\O;\fH)$
and the uniform bound \eqref{MDB1}. 
Hence, the proof  is completed. \qedhere

\end{proof}

Suppose $V\in L^2(\O;\fH)$.
Then, one can deduce from the Fubini's theorem
that for $\fm$-almost every $\xi\in\mathbf{Z}$, 
$V(\xi)\in L^2(\O)$, and thus admits the following chaos expansion

\noi
\begin{align}\label{227}
V(\xi) = h_0(\xi) + \sum_{n=1}^\infty I_n\big(  h_n(\xi, \bul)  \big),
\end{align}

\noi
where $h_0(\xi) = \E[ V(\xi)]$, and  $h_n(\xi, \bul)\in \fH^{\odot n}  $ may not
be symmetric in all of its $(n+1)$ arguments.
Moreover, due to $V\in L^2(\O; \fH)$, we get 

\noi
\begin{align}\label{227b}
\E\big[ \| V\|_\fH^2 \big] = \sum_{n\geq 0} n! \| h_n \|^2_{\fH^{\otimes (n+1)}} < +\infty.
\end{align}

The following lemma is the Poisson  analogue of \cite[Proposition 1.3.7]{Nua06} in the Gaussian setting.

\begin{lemma} \label{lem_dl}
Let  $V\in L^2(\O;\fH)$ and $V(\xi)$ have the chaos expansion as in  \eqref{227}-\eqref{227b}.
Then, $V\in\dom(\dl)$
if and only if 

 \noi
\begin{align}  \label{227d}
\sum_{n=1}^\infty n! \| \wt{h}_{n-1} \|_{\fH^{\otimes n}}^2 < \infty.
\end{align}

\noi
For $V\in\dom(\dl)$, we have

 \noi
\begin{align}   \notag%  \label{227c}
  \dl(V) = \sum_{n=1}^\infty I_n( \wt{h}_{n-1} ).
\end{align}
Here, $\wt{h}_n$ is the symmetrization of $h_n$ in all the $(n+1)$ arguments:

\noi
\begin{align}  \label{def-wth}
\begin{aligned}
\wt{h}_n(\xi_1,\ldots,\xi_n,\xi_{n+1})
&=
\frac{1}{n+1}\sum_{i=1}^{n+1}h_n(\xi_i, \pmb{\xi_{\neq i} } ),
%&=\frac{1}{n+1}\sum_{i=1}^{n+1}h_n(\xi_1,\ldots,\xi_{i-1},\xi_{i+1},\ldots,\xi_n,\xi_{n+1},\xi_i)
\end{aligned}
\end{align}

\noi
where $\pmb{\xi_{\neq i} } $ stands for the $n$ arguments $(\xi_1, ... , \xi_{i-1}, \xi_{i+1}, ... ,\xi_{n+1})$
when $i\geq 2$,  and for $(\xi_2, ... ,\xi_{n+1})$ when $i=1$.\footnote{Note that the canonical 
symmetrization of $h_n$ reduces to the above form \eqref{def-wth}, since
$h_n(\xi, \xi_1, ... , \xi_n)$ is symmetric in the $n$ arguments $(\xi_1, ... , \xi_n)$
for any fixed $\xi$.}

\end{lemma}

\begin{proof}
For the ``if'' part, see,  e.g., Lemma 2.4 in \cite{BZ23}. Next, we show the ``only-if'' direction. 

\smallskip

Assume $V\in\dom(\dl)$. Then, for any $F\in\dom(D)$,
we deduce from duality relation \eqref{dualR}
and orthogonality relation  that 
for $k\geq 1$,

\noi
\begin{align*}
\E\big[ \jb{D J_k F, V}_\fH \big] 
=  \E\big[ (J_k F) \dl(V) \big]  = \E\big[ F J_k (\dl(V)) \big],
\end{align*}
while, with $D_\xi J_k F\in\C_{k-1}$, we also deduce from orthogonality relation  that

\noi
\begin{align*}
\E\big[ \jb{D J_k F, V}_\fH \big] 
&= \int_{\mathbf{Z}} \E\big[  (D_\xi J_k F) V(\xi) \big]\, \fm(d\xi) \\
&=  \int_{\mathbf{Z}} \E\big[  (D_\xi  F)  J_{k-1}V(\xi) \big]\, \fm(d\xi) \\
&=  \E \int_{\mathbf{Z}}    (D_\xi  F)  I_{k-1} ( h_{k-1}(\xi)) \, \fm(d\xi) \\
&= \E\big[ F I_k(\wt h_{k-1}) \big],
\end{align*}

\noi
where the last equality follows from  \cite[Lemma 2.4]{BZ23}. In particular, the above 
two equations imply $J_k ( \dl (V) ) = I_k(\wt h_{k-1}) $, and thus that 
\[
\dl(V) = \sum_{k=1}^\infty  I_k(\wt h_{k-1})
\]
with 
\[
\| \dl(V)\|_2^2 = \sum_{k=1}^\infty k!  \| \wt h_{k-1} \|^2_{\fH^{\otimes k}} <\infty,
\]
which is exactly the condition \eqref{227d}. Hence, the proof is completed. \qedhere

\end{proof}

In the following, we provide the {\it Heisenberg's commutation} relation in 
the Poisson setting, which is crucial in deriving 
the integral equation \eqref{eq_MD} for the Malliavin derivative 
of the solution to \eqref{SWE};
see further discussions in Remark \ref{rem_MD}.

\begin{lemma}\label{lem_Heis}\textup{(Heisenberg's commutation relation)} 
Let $V\in\dom(\dl)$ as in Lemma \ref{lem_dl}.
Assume $\dl(V)\in\dom(D)$ and $V(\xi_0)\in\dom(D)$
for $\fm$-almost every $\xi_0\in\mathbf{Z}$ such that 

\noi
\begin{align}\label{cond_H1}
\E\int_{\mathbf{Z}^2} \big| D_{\xi_0} V(\xi) \big|^2 \fm(d\xi) \fm(d\xi_0) <\infty.
\end{align}

\noi
 Then,     $D_{\xi_0}V  \in \dom(\dl)$ for $\fm$-almost 
every $\xi_0\in\mathbf{Z}$ with

\noi
\begin{align}\label{com_rel}
D_{\xi_0}  \dl(V)=  \dl(  D_{\xi_0} V) + V(\xi_0) .
\end{align}

\end{lemma}

\begin{proof}

  Let $V\in\dom(\dl)$  have the expression \eqref{227}.
 Then,   the assumption $V(\xi)\in \dom(D)$ 
 implies 
 
 \noi
 \begin{align}\label{DxV}
 D_{\xi_0} V(\xi) = \sum_{n=1}^\infty n I_{n-1}\big( h_n(\xi, \xi_0, \bul) \big).
 \end{align}

\noi
Since 
the condition \eqref{cond_H1}    implies

\noi
\begin{align}
\int_{\mathbf{Z}^2} \E \Big[ \big| D_{\xi_0} V(\xi) \big|^2 \Big]  \fm(d\xi) \fm(d\xi_0) 
&=  \int_{\mathbf{Z}^2}  \sum_{n=1}^\infty n^2 (n-1)! \| h_n(\xi, \xi_0, \bul) \|^2_{\fH^{\otimes (n-1)}}  
\fm(d\xi) \fm(d\xi_0) \notag   \\
&= \sum_{n=1}^\infty n  n! \| h_n \|^2_{\fH^{\otimes (n+1)}},
\label{cond_H2}
\end{align}  

\noi
we get 

\noi
\begin{align}
\sum_{n=1}^\infty n  n! \| h_n \|^2_{\fH^{\otimes (n+1)}} < \infty.
 \notag% \label{cond_H2b}
\end{align}  

Before proving  $D_{\xi_0} V\in\dom(\dl)$,
we point out that the assumption $\dl(V)\in\dom(D)$ leads to 

\noi
\begin{align}
\sum_{n=1}^\infty n n! \| \wt{h}_{n-1}\|^2_{\fH^{\otimes n}} <\infty
\quad\text{or equivalently}
\quad
\sum_{n=1}^\infty (n+1)^2 n! \| \wt{h}_{n}\|^2_{\fH^{\otimes (n+1)}} <\infty.
\label{cond_H3}
\end{align}

\noi
And it is also immediate that 

\noi
\begin{align}
D_{\xi_0} \dl(V) = \sum_{n=1}^\infty n I_{n-1}\big(  \, \wt{h}_{n-1}(\xi_0) \big), 
\label{Dxdl}
\end{align}

\noi
where $ \wt{h}_{n-1}(\xi_0) $ is obtained by first symmetrizing $h_{n-1}$
and then fixing the first argument to be $\xi_0$.

Now let us show $D_{\xi_0} V\in\dom(\dl)$.
For this, we apply Lemma \ref{lem_dl}.
We first rewrite \eqref{DxV} as follows:

\noi
\begin{align*}
 D_{\xi_0} V(\xi)  =  \sum_{n=0}^\infty I_n\big( g_{\xi_0, n}(\xi, \bul) \big)
 \quad
 {\rm with}
 \quad
 g_{\xi_0, n}(\xi, \bul) := (n+1) h_{n+1}(\xi, \xi_0, \bul).
\end{align*}
 
\noi
It is not difficult to see that the canonical symmetrization of $ g_{\xi_0, n}$ is given by 

\noi
\begin{align}\label{syms}
\begin{aligned}
\wt  g_{\xi_0, n}(\xi_{n+1},  \xi_1, ... , \xi_n) 
&= \sum_{j=1}^{n+1} h_{n+1}(\xi_j,  \pmb{\xi_{\neq j}}) \\
&= \bigg(  \sum_{j=0}^{n+1} h_{n+1}(\xi_j,  \pmb{\xi_{\neq j}})  \bigg) - h_{n+1}(\xi_0, \xi_1, ... , \xi_{n+1}) \\
&= (n+2) \wt{h}_{n+1}(\xi_0, \xi_1, ... , \xi_{n+1}) - h_{n+1}(\xi_0, \xi_1, ... , \xi_{n+1}),
\end{aligned} 
\end{align}

 \noi
 where 
 $\pmb{\xi_{\neq j}}$ stands for the $(n+1)$ arguments 
 $(\xi_0, ...,  \xi_{j-1},\xi_{j+1}, ...,  \xi_{n+1})$.
 Therefore, it follows from \eqref{syms} that 
 
 \noi
 \begin{align*}
 \sum_{n\geq 1} n! \| \wt  g_{\xi_0, n-1} \|^2_{\fH^{\otimes n}}
 \leq 
 2  \sum_{n\geq 1} n! 
 \Big[   (n+1)^2 \|  \wt{h}_{n}(\xi_0, \bul) \|^2_{\fH^{\otimes n}}
  + \| h_{n}(\xi_0, \bul) \|^2_{\fH^{\otimes n}} \Big]   ,
 \end{align*}
 
 \noi
 which is integrable with respect to $\fm(d\xi_0)$ in view of \eqref{cond_H2} and \eqref{cond_H3}.
 Thus, $D_{\xi_0} V\in\dom(\dl)$
 for $\fm$-almost every $\xi_0\in\mathbf{Z}$ and
 \begin{align}\label{Dx0V}
 \dl\big( D_{\xi_0}V \big) = \sum_{n=1}^\infty I_n\big( \wt  g_{\xi_0, n-1}  \big).
 \end{align}
Hence the desired commutation relation $D_{\xi_0} \dl(V) = \dl( D_{\xi_0}V) + V(\xi_0)$
follows immediately 
from \eqref{Dxdl}, \eqref{Dx0V}, \eqref{227}, and \eqref{syms}. \qedhere
 
\end{proof}

\begin{remark}\rm \label{rem_LMS}
 
The authors of   \cite{LMS23} established the  commutation relation \eqref{com_rel} 
under a different set of conditions, among which 
there is (in our notation)
\[
\E \int_{\mathbf{Z}^3} \big( D^+_{\xi_1}  D^+_{\xi_2} V(\xi_3) \big)^2 \fm(d\xi_1) \fm(d\xi_2) \fm(d\xi_3)  <\infty;
\]
see Lemma 3.3 therein. 
In our Lemma \ref{lem_Heis}, we do not need the above condition 
on the second-order difference  operator  and thus provide 
a set of  more relaxed conditions for  \eqref{com_rel} to hold. 
This minor difference is important
 in our application of Lemma \ref{lem_Heis},
when we only know that  the solution $u(t,x)$ to \eqref{SWE} is Malliavin differentiable
(but not necessarily  twice Malliavin differentiable).
See, in particular, the application of Lemma \ref{lem_Heis}
in the proof of Proposition \ref{prop:diff}.
  
\end{remark}

\subsection{Predictability}\label{SEC2_3}

In this section, we include some auxiliary  results related to predictability.
 We start by recalling the following definition of modification.

\begin{definition}

If $X=\{X(t,x)\}_{(t,x)\in \R_{+} \times \R^d}$ and $Y=\{Y(t,x)\}_{(t,x)\in \R_{+} \times \R^d}$
 are two processes defined on the same probability space $(\O, \F,\bP)$, 
 we say that $X$ and $Y$ are   modifications of each other if
$\bP\big(X(t,x)=Y(t,x)\big)=1$ 
for     all $(t,x) \in \R_{+} \times \R^d$.
 \end{definition}

\begin{lemma}
\label{pred-all}
Let  $X=\{X(t,x)\}_{(t,x)\in \R_{+} \times \R^d}$ be a  process 
defined on a probability space $(\O,\F,\bP)$ such that 
 for any $T>0$ and for any compact set $K \subset \R^d$, $X\bv_{[0,T]\times K}$ 
 has a jointly measurable {\rm(}respectively, predictable{\rm)} modification.
 Then, $X$ has a jointly measurable 
 {\rm(}respectively, predictable{\rm)} modification on $\R_{+}\times \R^d$.
\end{lemma}

\begin{proof}
We  only prove  the case where $X$ is jointly measurable, 
while  the same argument works  for the case where  $X$ is predictable.

Let $E_m = \{ \| x\| \leq m\}$ for $m\geq 1$
and   let $X_m'$ be the jointly measurable modification of $X$ on $[0,m]\times E_m$. 
Then $\wt{X}_m:=X_m' \ind_{[0,m]\times E_m}$ is a jointly measurable modification of 
$X_m:=X \ind_{[0,m]\times E_m}$. 
Therefore, for   all $(t,x)\in \R_+ \times \R^d$, the event 
\[
\O_{t,x} = \bigcap_{m\geq 1}\{\wt{X}_m(t,x) = X_m(t,x)  \}   \]

\noi
 has probability $1$
 and on $\O_{t,x}$, one has
 
 \noi
 \begin{align} \label{lim_Otx}
 \lim_{m\to \infty}\wt{X}_m(t,x) = \lim_{m\to \infty}X_m(t,x)=X(t,x).
\end{align}

\noi
That is, the above limit \eqref{lim_Otx} takes place almost surely
for every $(t,x)\in\R_+\times\R^d$.

Now define 
$A=\{(\o,t,x)\in \O \times \R_{+}\times \R^d: 
\lim_{m\to \infty} \wt{X}_m(\o,t,x) \, \mbox{exists}\}$.
 Then, for   all $(t,x)\in \R_{+}\times \R^d$, the event  
 $\{\o \in \O: (\o,t,x)\in A^c\} \subset \O^c_{t,x}$ has probability $0$,
 so that  by Fubini theorem, the set $A$ has the full measure
 with $\ind_A(\o,t,x) = 1$ almost surely for every $(t,x)$,
  and moreover,  we can deduce from \eqref{lim_Otx} that 
 
 \noi
 \begin{align}
 \wt{X}(\o, t, x) 
:&=\lim_{m\to \infty} \wt{X}_m(\o, t, x) \ind_A(\o, t, x) \label{Otx1} \\
&=  X(\o, t, x) \ind_A(\o, t, x),  \label{Otx2}
 \end{align}
 
 \noi
where  the equality \eqref{Otx2} holds for every $(t,x)\in\R_+\times\R^d$ 
and for  $\bP$-almost every $\o\in\O$.
By definition \eqref{Otx1}, we easily see that    
$\wt{X}$ is jointly measurable and $\wt{X}(t,x)=X(t,x)$ almost surely
for   all $(t,x) \in \R_{+}\times \R^d$. 
Hence, $\wt{X}$ is a jointly measurable
modification of $X$.
\qedhere

\end{proof}

\begin{lemma}
\label{lem_pred}
Let $X= \{ X(t,x)\}_{(t,x)\in\R_+\times\R^d}$ 
be a collection of   random variables in $\dom(D)$ 
such that for any finite $T >0$,

\noi
\begin{align} \label{sup_DX}
\sup_{t\leq T}\sup_{x\in\R^d}  \| DX(t,x) \|_{L^2(\O; \fH)} < \infty.
\end{align}

\noi
Assume that  the map $(t,x) \mapsto X(t,x)$ is $L^2(\O)$-continuous 
on $\bR_{+}\times\R^d$. Then the following statements hold.

\smallskip
\noi
{\rm (i)} For $\fm$-almost every $\xi\in\mathbf{Z}$,
  $D_\xi X$ has a jointly measurable modification.
  
  \smallskip
\noi
{\rm (ii)}
   If in addition, $X$ is $\mathbb{F}$-adapted {\rm(}i.e., $X(t,x)$ is $\F_t$-measurable 
   for any $(t,x)\in\R_+\times \R^d${\rm)}, 
   then this modification is $\mathbb{F}$-predictable {\rm(}and adapted as well{\rm)}.

  \smallskip
\noi
{\rm (iii)}    $\{D_\xi X(t,x): (t, x, \xi)\in   \R_+\times\R^d \times \mathbf{Z} \}$
is jointly measurable.

\end{lemma}

Part (iii) is used in the estimation of $\Var(\mathbf{B_2})$ in Section \ref{SEC5_2},
while the other two results are utilized in the proof of Proposition \ref{prop:diff}.

\begin{proof}[Proof of Lemma \ref{lem_pred}]
We use the same approximation as in the proof of \cite[Proposition B.1]{BQS}.
Note that this reference  studies the adaptedness and predictability of a 
(stochastically continuous) random field,
while we focus on the Malliavin derivative of a field $\{D_\xi X(t,x)\}_{(t,x)\in\R_+\times\R^d}$,
and the random field 
 $\{D_\xi X(t,x)\}_{( t,x, \xi )\in   \R_+\times\R^d \times \mathbf{Z}}$.

Fix $T>0$ and a compact set $K \subset \bR^d$. 
By Lemma \ref{pred-all}, it suffices to show that $D_{\xi}X$ has a jointly measurable (or predictable) modification on the set $[0,T] \times K$.
In the following, we split the proof in two steps.

\smallskip

$\bul$  {\bf Step 1.} Suppose first that $X(t,x) \in \mathbb{C}_k$ for all $(t,x)$
and  for some fixed $k\geq 1$.  
Due to compactness, 
the map $(t,x)\in [0,T]\times K \mapsto X(t,x)\in L^2(\O)$ is uniformly continuous. 
Thus, for any $n\geq 1$, there exists $\dl_n>0$ such that 
for all $(t,x),(s,y)\in  [0,T]\times K$ with $|t-s|^2+|x-y|^2 \leq \dl_n$,
it holds that

\noi
\begin{align}\label{pred1}
\E\big[|X(t,x)-X(s,y)|^2\big] \leq \frac{1}{2^n}.
\end{align}

\noi
Let $\{0=t_0^{(n)}<t_1^{(n)}<\ldots< t_{J_n}^{(n)}=T\}$ 
be a partition of $[0,T]$ into subintervals of length smaller than $\dl_n$ 
and let  $(U_{\ell}^{(n)})_{\ell=1,\ldots, L_n}$ be a partition of $K$
 into (nonempty) Borel subsets with diameter less than $\dl_n$. 
 Let $x_{\ell}^{(n)} \in U_{\ell}^{(n)}$ be arbitrary for $\ell=1,\ldots,L_n$
 and we define 
  
\noi
\begin{align}\label{pred2}
X_n(t,x)=   X(0,x)\ind_{\{ t=0 \}} + 
\sum_{j=1}^{J_n}\sum_{\ell=1}^{L_n}X(t_j^{(n)},x_{\ell}^{(n)})
\ind_{(t_{j}^{(n)},t_{j+1}^{(n)}] \times U_{\ell}^{(n)}}(t,x).
\end{align}

\noi
It is immediate that $X_n$ is a simple and  jointly measurable  process    such that 

\noi
\begin{align}
\label{E-Xn}
\E\big[|X_n(t,x)-X(t,x)|^2\big]   \leq \frac{1}{2^n}
\end{align}
for every $(t,x)\in [0, T]\times K$,
in view of \eqref{pred2} and \eqref{pred1}. 
It is also clear that when   $X$ is adapted, 
then the process  $X_n$ is predictable.

 Note that for any $\xi \in \mathbf{Z}$,
\[
D_{\xi}X_n(t,x)= D_\xi X_0(t,x) + 
\sum_{j=1}^{J_n}\sum_{\ell=1}^{L_n}D_{\xi}X(t_j^{(n)},x_{\ell}^{(n)}) 
\ind_{(t_{j}^{(n)},t_{j+1}^{(n)}] \times U_{\ell}^{(n)}}(t,x) 
\]
and thus,
 $D_{\xi}X_n$ is jointly measurable. 
 If  $X$ is additionally assumed to be  adapted, 
 then for any $t \in (t_j^{(n)},t_{j+1}^{(n)}]$,  
 $X(t_j^{(n)},x_{\ell}^{(n)})$ is $\F_t$-measurable, and thus, 
 $D_{\xi}X(t_j^{(n)},x_{\ell}^{(n)})$ is   $\F_t$-measurable 
 by Lemma \ref{rem_BZ-iv}.
 As a result,   $D_{\xi}X_n$ is predictable for any $\xi \in \mathbf{Z}$
and
 $\{D_\xi X_n(t,x)\}_{(t, x, \xi)\in \R_+\times\R^d\times\mathbf{Z}}$
is clearly  jointly measurable.

Since $X_n(t,x)$ and  $X(t,x)$ live in  $\mathbb{C}_k$,  
one can apply  \eqref{lem_XY} (with $F\in\mathbb{C}_k$)     
and \eqref{E-Xn} to obtain 
\[
\E\big[\|DX_n(t,x)-DX(t,x) \|_{\fH}^2\big]=k \E\big[|X_n(t,x)-X(t,x)|^2\big]\leq \frac{k}{2^n}
\]

\noi
for any $(t, x)\in [0, T]\times K$,
and furthermore, 
\[
\sum_{n\geq 1}\int_0^T \int_{K}\E\big[\|DX_n(t,x)-DX(t,x) \|_{\fH}^2\big]dxdt <\infty.
\]
By Fubini's theorem, we can rewrite the above inequality as follows:
\[
\int_{\mathbf{Z}}\sum_{n\geq 1}
\E\left(\int_0^T \int_{K} |D_{\xi}X_n(t,x)-D_{\xi}X(t,x)|^2dxdt \right) \fm(d\xi)
<\infty. 
\]
It follows that for $\fm$-almost all $\xi \in \mathbf{Z}$,

\noi
\begin{align*} 
\text{$D_{\xi}X_n$ converges to   $D_{\xi}X$ 
in $L^2(\Omega \times [0,T]\times K)$ as $n\to \infty$}. 
\end{align*}

\noi
Then, along a subsequence $(n_k\uparrow +\infty)$,
 we have for almost every $(\omega,t,x) \in \Omega \times [0,T]\times K$,
\[
D_{\xi}X_{n_k}(\o,t,x) \xrightarrow{n_k\to\infty} D_{\xi}X(\o,t,x).
\]
 
\noi
Then $D_{\xi}X$ has a jointly measurable modification 
(and a predictable modification, if $X$ is additionally  adapted). 
Moreover, 
$\{D_\xi X(t,x): (t, x, \xi)\in \R_+\times\R^d\times\mathbf{Z}\}$
is jointly measurable
as an almost everywhere limit 
of a sequence of jointly measurable random fields. 

\smallskip

$\bul$  {\bf Step 2.}   Consider now the general case. 
For any $(t,x)\in [0,T]\times K$, we have the chaos expansion 
$X(t,x)=\sum_{k\geq 0}X_k(t,x)$ with $X_k(t,x)\in \mathbb{C}_k$, and
$D_{\xi}X(t,x)=\sum_{k\geq 1}D_{\xi}X_k(t,x)$. 
The terms are orthogonal in both series. 
By  the assumption \eqref{sup_DX} and Fubini's theorem,
we have 
\[
\int_{\mathbf{Z}} \int_0^T \int_{K} 
\E\big[|D_{\xi}X(t,x)|^2\big] dxdt\fm(d\xi)=
\int_0^T \int_{K} 
\E\big[\|DX(t,x)\|_{\fH}^2\big] dxdt<\infty.
\]
It follows from orthogonality  relation \eqref{orth_rel}  that 
for $\fm$-almost all $\xi \in \mathbf{Z}$, 
\[
\sum_{k\geq 1}
\int_0^T \int_{K} 
\E\big[|D_{\xi}X_k(t,x)|^2\big] dxdt=\int_0^T \int_{K} 
\E\big[|D_{\xi}X(t,x)|^2\big] dxdt<\infty,
\]
 
 \noi
 and therefore, 
\begin{align} \label{pred3}
 \text{$\sum_{k=1}^{n}D_{\xi}X_k$ converges to $D_\xi X$ 
 in $L^2(\O \times [0,T]\times K)$
as $n\to \infty$}. 
\end{align}

By {\bf Step 1}, $\sum_{k=1}^{n}D_{\xi}X_k$ 
has a jointly measurable modification 
(and a predictable modification, if $X$ is adapted),
and $\{\sum_{k=1}^{n}D_{\xi}X_k(t,x): (t, x, \xi)\in    \R_+\times\R^d \times   \mathbf{Z}\}$
has a jointly measurable modification.
Hence, the same thing can be said about $D_{\xi}X$ 
by taking an almost everywhere convergent subsequence 
in \eqref{pred3}.

\smallskip

Hence, the proof is completed. 
\qedhere

\end{proof}

\section{Malliavin differentiablity of the solution}
\label{section-Du}

In this section, we present several results about the Malliavin differentiability 
 of the solution to  \eqref{SWE}, including the key estimate  \eqref{key}.
More precisely,  we prove in Proposition  \ref{prop:diff}
the Malliavin differentiability  of the solution $u(t,x)$
and then show the Malliavin derivative satisfies 
 a certain integral equation,
 which serves as the very first step to prove our  Theorem \ref{thm_main}.
 Note that Proposition  \ref{prop:diff} substantially  improves  
 the main result of \cite{BN17}, which was obtained only
  for  $\s(u)=au+b$. 
  The key ingredients that lead to this improvement
  is the  commutation relation \eqref{com_rel} in   Lemma \ref{lem_Heis}
and  the utilization  of the integral equation  \eqref{mild}
  or equivalently \eqref{eq_u} below.

\begin{proposition} \label{prop:diff}
Assume $m_2\in(0,\infty)$ as in  \eqref{m2}.
Let $u$ solve the equation \eqref{SWE}. 
Then, $u(t,x)\in\dom(D)$ for any $(t,x)\in\R_+\times\R$
with
\begin{align}\label{uniD2}
 \sup_{t\leq T}\sup_{x\in\R}  \| Du(t,x) \|_{L^2(\O;\fH)}   < \infty
\end{align}

\noi
for any finite $T> 0$. 
Moreover, for $\fm$-almost every $\xi =(r, y, z)\in\mathbf{Z}$,
$D_\xi u$ admits a predictable modification
{\rm(}still denoted by $D_\xi u${\rm)}
and 
  the following  commutation relation holds:

\noi
\begin{align} 
\begin{aligned} \label{eq_MD}
D_{r, y, z} u(t,x) 
&= G_{t-r}(x-y)z  \s( u(r, y)) \\
&\qquad + \int_0^t \int_{\R} \int_{\R_0} G_{t-s_1}(x-y_1)  z_1 D_{r, y, z} \s\big( u(s_1, y_1) \big)   
\wh{N}(ds_1, dy_1, dz_1).
\end{aligned}
\end{align}

\end{proposition}

\begin{proof}%[Proof of Proposition \ref{prop:diff}]

%Fix $(t,x)\in\R_+\times\R$ and 
Let  $(u_n)_{n\geq 0}$ be the sequence of Picard iterations, defined as follows:
$u_0(t,x) =1$ and 
\noi
\begin{align}
\label{nPicard}
u_{n+1}(t,x) = 1 + \int_0^t\int_\R G_{t-s}(x-y) \s\big( u_{n}(s, y) \big) L(ds, dy), 
\quad \mbox{for all} \quad n \geq 0. %n\in\N_{\geq 0}.
\end{align}

\noi
Via an induction argument on\footnote{In \cite[Theorem 9]{BN17}, (v) is not included in the induction 
step (\pmb{P}), but this addendum would not affect the proof therein.  }
\begin{align}\label{propP}
\textup{(\pmb{P})}
\quad
\begin{cases}
&{\rm(i)} \,\,\, \text{$u_n(t,x)\in  L^2(\O)$   for any $(t,x)\in\R_+\times\R$} \\
&{\rm(ii)} \,\,\, \text{$\sup\{ \| u_n(t,x)\|_2 : t\leq T, x\in\R\} <\infty$ for any finite $T$    } \\
&{\rm(iii)} \,\,\, \text{$(t,x)\in\R_+\times\R\mapsto u_n(t,x)\in L^2(\O)$ is continuous} \\
&{\rm(iv)} \,\,\, \text{$u_n(t,x)$ is $\mathcal{F}_t$-measurable for any $(t,x)\in\R_+\times\R$ } \\
&{\rm(v)} \,\,\, \text{$u_n$ has a $\mathbb{F}$-predictable modification, } 
\end{cases}
\end{align}

\noi
one can   show that there is a (predictable) random field $\{u(t,x)\}_{(t, x)\in\R_+\times\R}$ such that 
for any finite $T>0$,

\noi
\begin{align}\label{unic3}
\lim_{n\to\infty} \sup_{t\leq T} \sup_{x\in\R} \| u_n(t,x) - u(t,x) \|_2  = 0,
\end{align}

\noi
and moreover, $u$ is the unique solution to \eqref{SWE};
see \cite[Theorem 9]{BN17} for more details. 
By running another induction argument on %\footnote{(iii) \color{red}X}

\noi
\begin{align}\label{propQ}
\textup{(\pmb{Q})}
\quad
\begin{cases}
&{\rm(i)} \,\,\, \text{$u_n(t,x)\in  \dom(D)$   for any $(t,x)\in\R_+\times\R$} \\
&{\rm(ii)} \,\,\, \text{$\displaystyle 
\sup\big\{ \| Du_n(t,x)\|_{L^2(\O;\fH)} : (t, x)\in[0,T]\times\R\big\}  <\infty$
 for any finite $T$,    } % \\
%  &{\rm(iii)} \,\,\, \text{for any fixed $\xi \in \mathbf{Z}$, 
%  $ D_{\xi}u_n$ has a $\mathbb{F}$-predictable modification,}
\end{cases}
\end{align}
 
 \noi
 one can   show that for any finite $T>0$,
 \begin{align}\label{uniD1}
\sup_{n\geq 1} \sup_{t\leq T}\sup_{x\in\R} \big\| Du_n(t,x)\big\|_{L^2(\O;\fH)}   < \infty;
\end{align}

\noi
see \cite[Lemma 12]{BN17}.\footnote{It is labelled as Lemma 4.1 in its arXiv version.}
It follows from the bound \eqref{uniD1},  properties (iii) and (iv) in \eqref{propP}, with Lemma \ref{lem_pred}
that for $\fm$-almost every $\xi\in\mathbf{Z}$,

\noi
\begin{align} \label{D_xMOD}
\text{$D_\xi u_n$ admits a predictable modification (still denoted by $D_\xi u_n$).} 
\end{align}

Next,  we can combine the uniform convergence  in \eqref{unic3} 
and uniform bound in \eqref{uniD1}
to obtain  $u(t,x)\in \dom(D)$ with \eqref{uniD2},
which is an easy  consequence of our Lemma \ref{lem_1}. 
As a consequence of Lemma \ref{lem_pred} again, 
$D_\xi u$ admits a predictable modification for $\fm$-almost every $\xi\in\mathbf{Z}$.

\smallskip

Finally, we   show that the integral equation \eqref{eq_MD} holds. 
First, we write 
\begin{align}\label{eq_u}
u(t,x) = 1 + \dl( V_{t,x}),
\end{align}

\noi
where $V_{t,x}$ is a predictable random field\footnote{Hence, $V_{t,x}\in \dom(\dl)$ 
according to Lemma \ref{lem25}.}
given by $V_{t,x}(r, y, z) = G_{t-r}(x-y)z \s\big( u(r, y) \big)$. In particular, 
we have $\dl( V_{t,x}) = u(t,x) - 1\in \dom(D)$. In order to apply Lemma \ref{lem_Heis},
we need to verify the condition \eqref{cond_H1}.
Indeed, 
\begin{itemize}
\item[(a)] for $u(r,y)\in\dom(D)$ and $\s$ Lipschitz, we have 
$ \s\big( u(r, y) \big)\in\dom(D)$ in view of Remark \ref{rem_BZ-iii},
and moreover, we have 
\[
| D_{r_0, y_0, z_0} \s\big( u(r, y) \big) | \leq \textup{Lip}(\s) | D_{r_0, y_0, z_0} u(r, y) |
\]
with $ \textup{Lip}(\s) $ denoting the Lipschitz constant of $\s$;

\item[(b)]

 with $\xi_0 = (r_0, y_0,z_0)$ and $\xi = (r, y, z)$,
\[
D_{\xi_0} V_{t,x}(\xi) = G_{t-r}(x-y)z D_{r_0, y_0, z_0} \s\big( u(r, y) \big),
\]
we can deduce from the above point (a) and the uniform bound \eqref{uniD2}
that 

\noi
\begin{align} \notag%  \label{cond_H1b}
\begin{aligned} 
&\quad \E\int_{\mathbf{Z}^2} \big| D_{\xi_0} V_{t,x}(\xi) \big|^2 \fm(d\xi) \fm(d\xi_0)  \\
&\leq   \textup{Lip}^2(\s) 
\E  \int_{\mathbf{Z}^2} | G_{t-r}(x-y)z  |^2  | D_{\xi_0}  u(r, y) |^2 m(d\xi_0) drdy \nu(dz) \\
&\leq   \textup{Lip}^2(\s)  \bigg( \sup_{r\leq t} \sup_{y\in\R}  \| D u(r, y) \|^2_{L^2(\O;\fH)} \bigg)
\int_{\mathbf{Z}} | G_{t-r}(x-y)z  |^2 drdy \nu(dz)
 < +\infty .
\end{aligned}
\end{align}

\end{itemize}
That is, the condition \eqref{cond_H1} is   verified, so that 
Lemma \ref{lem_Heis} implies   
$D_{\xi} V_{t,x}\in \dom(\dl)$ and 
\[
D_{\xi} u(t,x) = D_{\xi} \dl( V_{t,x}) = \dl (D_{\xi} V_{t,x}  ) + V_{t,x}(\xi),
\]
which is exactly the desired integral equation \eqref{eq_MD}.
\qedhere

\end{proof}

\begin{remark}\rm \label{rem_MD}

(i)
If in addition $m_p<\infty$ for some $p>2$, then a similar argument shows that
%it can be proved that
\[
\sup_{t\leq T} \sup_{x\in\R}  \big\| u_n(t,x)-u(t,x) \big\|_p \to 0  
\]
as $n\to\infty$, and consequently with \eqref{KTP}, we have
\[
A_{T,p}:=\sup_{n\geq 1}\sup_{(t,x) \in [0,T]\times \R}\E\big[|u_n(t,x)|^p\big]<\infty.
\]
In this case, using the Lipschitz property of $\s$, we have
\begin{equation}
\label{ATPs}
A_{T,p,\s}:=\sup_{n\geq 1}\sup_{(t,x) \in [0,T]\times \R}\E\big[|\s(u_n(t,x))|^p\big] \leq 2^{p-1}(|\s(0)|^p+\Lip^p(\s)A_{T,p}).
\end{equation}

\smallskip
\noi
(ii) As indicated in \eqref{D_xMOD},
 $D_\xi u_n$ is predictable.   By the same argument leading to \eqref{eq_MD},
 we have

\noi
\begin{align} 
\begin{aligned} \label{eq_MDN}
D_{r, y, z} u_{n+1}(t,x) 
&= G_{t-r}(x-y)z  \s( u_n(r, y)) \\
&\quad + \int_0^t \int_{\R} \int_{\R_0} G_{t-s_1}(x-y_1)  z_1 D_{r, y, z} \s\big( u_n(s_1, y_1) \big)   
\wh{N}(ds_1, dy_1, dz_1).
\end{aligned}
\end{align}

\noi
And for any $(t,x)\in(0,\infty)\times\R$
and  for $\fm$-almost every 
$(r,y,z)\in\mathbf{Z}$, the process $ \{\mathfrak{M}_s\}_{s\in[r,t]}$, with

\noi
 \begin{align}   \notag%  \label{eq_mart}
 \mathfrak{M}_s := \int_r^s\int_{\R}\int_{\R_0} G_{t-r_1}(x-y_1)z_1  D_{r, y, z}\s( u_n(r_1,y_1) )  
  \wh{N}(dr_1, dy_1, dz_1),
 \end{align}
 has a c\`adl\`ag {\rm(}i.e., right continuous with left limits{\rm)} modification,
 which is an $\mathbb{F}$-martingale with {\rm(}predictable{\rm )} quadratic variation
 \[
 \jb{\mathfrak{M}}_s
 =  \int_r^s\int_{\R}\int_{\R_0} \big[G_{t-r_1}(x-y_1)z_1  D_{r, y, z}\s( u_n(r_1,y_1) )  \big]^2   
 dr _1dy_1\nu(dz_1).
 \]
This will be used in the proof of the following Proposition \ref{prop_MD}.

\end{remark}

Finally, we are ready to state the key technical estimate on the Malliavin derivative
of the solution $u(t,x)$, which will be combined with a variant of 
the discrete Malliavin-Stein bound to obtain the quantitative CLT
in Theorem \ref{thm_main}.

\begin{proposition}\label{prop_MD}
Assume $m_2\in(0,\infty) $ as in \eqref{m2}.
Let $u$ be the unique mild solution to \eqref{SWE}
and $D_\xi u$ denote the predictable random field from Proposition \ref{prop:diff}
for $\fm$-almost every $\xi\in\mathbf{Z}$.
Suppose that   $m_p < \infty$ for some finite $p\geq 2$.
 Then, for   every $(t,x)\in \bR_{+} \times\R$ and 
 for $\fm$-almost every $(r,y,z)\in [0,t] \times \R \times \R_0$, we have

\noi
\begin{align} \label{Z1}
\| D_{r, y, z} u(t,x) \|_p \leq \mathcal{C}_{t,p, \s}  G_{t-r}(x-y) |z| ,
\end{align}

\noi
where  
\begin{align}
\label{CTPS}
\mathcal{C}_{t,p, \s} = 2^{1-\frac1p}  A_{t,p,\s}^{\frac1p}   
 \exp\Big( \frac{1}{p} t^2 \Lip^p(\s) C_p(t) \Big)
\end{align}

\noi
with $A_{t,p,\s}$ as in  \eqref{ATPs},
 $\textup{Lip}(\s)$ denoting the Lipschitz constant of $\s$, 

\noi
\begin{align}
 \notag%\label{CPT}
C_{p}(t)=2^{p-1}B_p^p\big( m_2^{\frac p2} t^{p-2} + m_p \big),
\end{align}

\noi
and  $B_p$ the constant in the  Rosenthal's inequality.

\end{proposition}

\begin{proof}

By Remark \ref{rem_MD}, for $\fm$-almost every $\xi =(r, y,z) \in \mathbf{Z}$,
  $D_\xi u_{n+1}$ is predictable for each $n \geq 0$
  with  the integral equation \eqref{eq_MDN}.
Using Minkowski's inequality and  Proposition  \ref{prop_Rose}, 
\eqref{ATPs}, and \eqref{add1b}, we have

\noi
\begin{align}   \notag%  \label{MD1}
\begin{aligned}
& \| D_{r, y, z} u_{n+1}(t,x)  \|_p 
 \leq  G_{t-r}(x-y)   \|\s(u_n(r,y))\|_p |z|\\
&\qquad  + \bigg(  C_{p}(t)   \int_r^t \int_{\R}  G_{t-s_1}^p(s-y_1)
     \big\|  D_{r, y, z} \s(u_n(s_1, y_1))  \big\|_p^p \, ds_1dy_1 \bigg)^{\frac1p}\\
    & \leq   a_t^{\frac1p}|z|G_{t-r}(x-y)+ \bigg( b_t  \int_r^t \int_{\R} 
    G^p_{t-s_1}(x-y_1)
    \|  D_{r, y, z} u_n(s_1, y_1) \|_p^p \, ds_1dy_1 \bigg)^{\frac1p},
\end{aligned}
\end{align}

\noi
where $a_t=A_{t, p,\s}$ is given by \eqref{ATPs} and 
$b_t=\Lip^p(\s)   C_p(t)$.
Using  
$|a+b|^p \leq 2^{p-1} (|a|^p + |b|^p)$
and $2^{p-1} G_t^p = G_t$ in view of \eqref{def-G},
we can further get 

\noi
\begin{align}\label{MD2}
\begin{aligned}
 \| D_{r, y, z} u_{n+1}(t,x)  \|_p^p 
& \leq  a_{t}  |z|^p G_{t-r}(x-y)\\
&\quad +    b_t  \int_r^t \int_{\R}  G_{t-s_1}(x-y_1)
    \big\|  D_{r,y,z} u_n(s_1,y_1) \big\|_p^p  \,  ds_1dy_1.
\end{aligned}
\end{align}

\noi
Iterating \eqref{MD2} for  $n$ times
with
$D_{r,y, z} u_0(t,x)  = 0$,
 we  can obtain

\noi
\begin{align}\label{MD3}
\begin{aligned}
& \| D_{r, y, z} u_{n+1}(t,x)  \|_p^p   \\
&   \leq  a_{t}  |z|^p   
\sum_{j=0}^{n} b_t^j
\int_{r< s_j <  \ldots < s_1 < t } 
\int_{\R^j}  \prod_{k=0}^{j} G_{s_k-s_{k+1}}(y_k-y_{k+1})   dy_1\cdots dy_j ds_1 \cdots ds_j 
\end{aligned}
\end{align}

\noi
with $(s_0, y_0) = (t, x)$ and $(s_{j+1}, y_{j+1}) = (r, y)$.
Next, using   $G_t(x) = \frac{1}{2}\ind_{\{ |x| < t \}}$ 
and    the triangle inequality, we have

\noi
\begin{align*}
 \prod_{k=0}^{j} G_{s_k-s_{k+1}}(y_k-y_{k+1})
 & =\frac{1}{2^{j+1}}\prod_{k=0}^{j}\ind_{\{ |y_k-y_{k+1} |< s_k-s_{k+1}\}} \\
 &\leq \frac{1}{2^{j+1}}
\ind_{ \{ |x-y| < t -r  \}}
\prod_{k=0}^{j-1}  \ind_{\{ |y_k-y_{k+1} |<  s_k-s_{k+1}  \}}.
\end{align*}
Hence, the integrals in \eqref{MD3} can be estimated as follows:

\noi
\begin{align*}
& \int_{r< s_j  < \ldots < s_1 < t }
\int_{\R^j} 
\prod_{k=0}^j  G_{s_k-s_{k+1}}(y_k-y_{k+1}) dy_1\cdots dy_j   ds_1 \cdots ds_j \\ %\ind_{\{ |y_k-y_{k+1} |< s_k-s_{k+1}\}} \\
&\qquad \leq   \frac{1}{2^{j+1}} \ind_{ \{ |x-y| < t -r  \}}
\int_{r< s_j <\ldots < s_1 < t } 
\left(\int_{\R^j}  \prod_{k=0}^{j-1}  \ind_{\{ |y_k-y_{k+1} |<  s_k-s_{k+1}   \}} 
  dy_1\cdots dy_j \right) ds_1 \cdots ds_j  \\
&\qquad \leq   \frac{1}{2}  \ind_{ \{ |x-y| < t -r  \}}  \frac{(t-r)^{2j}}{ j!} 
                =G_{t-r}(x-y) \frac{ (t-r)^{2j}}{j!} , 
\end{align*}

\noi
where we performed the integration on $\R^j$ in the order $dy_j, dy_{j-1}, \ldots, dy_1$
that leads to a bound by $2^j(t-r)^j$.
Therefore, we deduce from \eqref{MD3} that

\noi
\begin{align*}
\| D_{r,y,z} u_{n+1}(t, x) \|^p_p 
&\leq a_t |z|^p G_{t-r}(x-y) 
\sum_{j=0}^{n}
b_t^j      \frac{ (t-r)^{2j}}{j!} \\
&\leq  a_t |z|^p  e^{ t^2 b_t} G_{t-r}(x-y) .
\end{align*}

\noi
From  $G_t^p(x)2^{p-1} =G_t(x)$ again,  we deduce that

\noi
\begin{align}\label{MDN}
\| D_{r,y,z} u_{n+1}(t, x) \|_p 
\leq \mathcal{C}_{t,p, \s} G_{t-r}(x-y)|z|,
\end{align}

\noi
where 
$\mathcal{C}_{t,p, \s} := 2^{1-\frac{1}{p}}  a_{t}^{\frac1p} \exp\big(  \frac{1}{p}t^2 b_t \big) 
= 2^{1-\frac{1}{p}}  A_{t, p,  \s}^{\frac1p}  \exp\big(  \frac{1}{p}t^2 \Lip^p(\s) C_p(t) \big) $  
%with $C_p(T)$ as in \eqref{CPT}, 
does not depend on $n$.\footnote{From \eqref{MDN}, we can provide another proof of
the uniform bound \eqref{uniD1}.}

\smallskip

 Finally, we proceed with a weak convergence argument 
 to establish  \eqref{Z1}.

 For $p\in[2,\infty)$, let $p^\ast = \frac{p}{p-1}\in (1, 2]$ be the 
H\"older conjugate exponent. 
Let $\phi\in\fH$ be  \emph{deterministic} and {\it nonnegative},  and let $\W$ be a real
 random variable with finite second moment and
$\| \W \|_{p^\ast} \leq 1$
so that the mapping
\[
(\o, r, y, z)\in\O\times\R_+\times\R\times\R_0 \mapsto \W(\o) \phi(r,y,z)
\]
belongs to $L^2(\O; \fH)$.
Then, thanks to Fubini's theorem and  the weak convergence of  $Du_n(t,x)$to  $Du(t,x)$ in $L^2(\O; \fH)$ (due to Lemma \ref{lem_1}),
we get 
\begin{equation}
\label{lim-WD}
\int_{\mathbf{Z}}  \phi(r,y,z) \E\big[ \W D_{r, y, z} u_n(t,x) \big] drdy \nu(dz) 
\to \int_{\mathbf{Z}}  \phi(r,y,z) \E\big[ \W D_{r, y, z} u(t,x) \big] drdy \nu(dz) 
\end{equation}
as $n\to+\infty$. 
It follows from \eqref{MDN} and H\"older inequality that, for any $n\ge 1$,

\noi
\begin{align}   \notag%  \label{Radon11}
\begin{aligned}  
&\bigg| \int_{\mathbf{Z}}  \phi(r,y,z) \E\big[ \W D_{r, y, z} u_n(t,x) \big] drdy \nu(dz)  \bigg| \\
& \qquad\qquad \leq
\mathcal{C}_{t,p,\s}  \int_{0}^t \int_{\R}\int_{\R_0}  \phi(r,y,z)  G_{t-r}(x-y) |z|   \nu(dz) dy dr. 
\end{aligned}
\end{align}

\noi
Hence, using \eqref{lim-WD}, we infer that
\begin{align}\label{Radon1}
\begin{aligned}  
&\bigg| \int_{\mathbf{Z}}  \phi(r,y,z) \E\big[ \W D_{r, y, z} u(t,x) \big] drdy \nu(dz)  \bigg| \\
& \qquad\qquad \leq
\mathcal{C}_{t,p,\s}  \int_{0}^t \int_{\R}\int_{\R_0}  \phi(r,y,z)  G_{t-r}(x-y) |z|   \nu(dz) dy dr. 
\end{aligned}
\end{align}

Now we put 
\[
q(r, y, z) := \E\big[ \W D_{r, y, z} u(t,x) \big], 
\]
which is an element in $\fH$, since $u(t,x) \in {\rm dom}(D)$.
Then, 
\[
\mu_q(d\xi) = q(\xi) \fm(d\xi)
\]
is a  signed Radon measure with the Hahn decomposition
$
\mu_q = \mu^+_q - \mu_q^-,
$
where
\[
\mu^+_q(K) = \int_K \max\{ q(\xi), 0 \} \fm(d\xi)
\quad
{\rm and}
\quad
\mu^-_q(K) = \int_K \max\{ - q(\xi), 0 \} \fm(d\xi).
\]
By the basic property of {\it Hahn decomposition},
one can find 
a partition $\{P^{+},  P^{-}\}$ of  $\mathbf{Z}$  such that
$\mu_q^{+}(K_1) \geq 0$ for all $K_1 \subseteq P^{+}, K_1 \in \mathcal{Z}$ and 
$\mu_q^{-}(K_2) \leq 0$ for all $K_2\subseteq P^{-}, K_2 \in \mathcal{Z}$;
see, e.g., \cite[Theorem 5.6.1]{Dudley}.\footnote{To be more precise, 
we shall  proceed by truncating the measure $\mu_q$ on  $E_m$ 
with finite $\fm$-measure and $E_m\uparrow \mathbf{Z}$, 
and then, the truncated measure $\mu_q(\bul \cap E_m)$
satisfies the definition of signed measure on \cite[page 178]{Dudley} and 
the resulting Hahn decomposition holds: 
$\mu_q(\bul \cap E_m) =
 \mu^+_q(\bul \cap E_m) - \mu^-_q(\bul \cap E_m)$.
 The rest of the proof can be done for the truncated measures and 
 the resulting constants (e.g., the ones in \eqref{Radon2}-\eqref{Radon3})
  do not depend on $m$ as one can see easily,
 and a final passage of $m\to\infty$ concludes the proof. 
}
 Moreover,
\[
\mu_q^+(K) = \mu_q(P^{+} \cap K)
\quad
{\rm and}
\quad
\mu_q^-(K) = -\mu_q(P^{-} \cap K)
\]
for any $K\in\mathcal{Z}$.

Now one can observe that 
each of the Radon measures  $\mu_q^+$, $\mu_q^{-}$, and $\mu_G$ are absolutely continuous
with respect to the Radon measure $\fm$,\footnote{Here, we treat $\nu$ as a measure on $\R$
with $\nu(\{ 0\}) =0$, then we view $\fm$ as a measure on $\R^3$.}
where
\[
\mu_G(dr, dy, dz) = G_{t-r}(x-y)|z| drdy \nu(dz). 
\]
Therefore, we can use the  ``{\it differentiation of Radon measures}''\footnote{One can refer
to the book \cite{EV15} by Evans and Gariepy for more details,
and in particular Theorem 1.30, Theorem 1.29,
Definition 1.20, and Definition 1.21.}
that 
for $\fm$-almost every $(r, y, z)\in\mathbf{Z}$,
as the corresponding   {\it Radon-Nikodym derivatives},  
$q^+(r, y, z)$ and $    G_{t-r}(x-y) |z|$
 satisfy: for   $\fm$-almost every
  $\xi=(r,y,z)\in \mathbf{Z}$,
\noi
\begin{align}
\begin{aligned}
q^+(\xi) :=  \max\{q(\xi), 0 \} 
& =  \lim_{\eps\downarrow  0}
  \frac{ \mu_q^+ \big( B_{\eps}(\xi)    \big)  }{ \fm  \big( B_{\eps}(\xi)    \big) }
    =  \lim_{\eps\downarrow  0}   \frac{ \mu_q \big( B_{\eps}(\xi) \cap P^{+}  \big)   }{ \fm  \big( B_\eps(\xi)    \big) }
   \\
  &\leq \mathcal{C}_{t,p,\s} 
    \lim_{\eps\downarrow  0}
  \frac{ \mu_G \big( B_{\eps}(\xi)    \big)     }{ \fm  \big( B_{\eps}(\xi)    \big) } \\
  &=\mathcal{C}_{t,p,\s} 
    G_{t-r}(x-y) |z|,
\end{aligned}
\label{Radon2}
\end{align}

\noi
where    $B_\eps(\xi)$ denotes the usual Euclidean open ball with center at $\xi$ and radius $\eps$,
and
the above inequality follows from the bound \eqref{Radon1}:

\noi
\begin{align*}
0 & \leq \mu_q \big(B_{\varepsilon}(\xi)\cap P^{+}\big)
=\int_{B_{\varepsilon}(\xi)\cap P^{+}}
\E\big[ \W D_{r', y', z'} u(t,x) \big]  dr'dy' \nu(dz')\\
&\leq C_{t,p,\s } \int_{B_{\varepsilon}(\xi)\cap P^{+}}  G_{t-r'}(x-y') |z'|dr'dy' \nu(dz')\\
&= C_{t,p,\s} \, \mu_G \big( B_\eps(\xi)\cap P^{+}\big) \leq
C_{t,p,\s} \, \mu_G \big( B_\eps(\xi)\big).
\end{align*}

The same argument yields

\noi
\begin{align}
\begin{aligned}
q^-(\xi) :=  \max\{ - q(\xi), 0 \} 
& =  \lim_{\eps\downarrow  0}
  \frac{ \mu_q^- \big(  B_\eps(\xi)    \big)    }{ \fm  \big(B_\eps(\xi)    \big) }
  =  \lim_{\eps\downarrow  0}   \frac{ -\mu_q \big( B_\eps(\xi) \cap P^{-}  \big)    }{ \fm  \big( B_\eps(\xi)    \big) }
   \\
  &\leq   \mathcal{C}_{t,p,\s} 
    \lim_{\eps\downarrow  0}
  \frac{ \mu_G \big( B_\eps (\xi)    \big)     }{ \fm  \big( B_\eps(\xi)    \big) } \\
  &= \mathcal{C}_{t,p,\s} 
    G_{t-r}(x-y) |z|,
\end{aligned}
\label{Radon3}
\end{align}
where for the inequality above, we used the fact that
\begin{align*}
0 & \leq -\mu_q \big(B_{\eps}(\xi)\cap P^{-}\big)
=\left|\int_{B_{\eps}(\xi)\cap P^{+}}\E\big[ \W D_{r', y', z'} u(t,x) \big]  dr'dy' \nu(dz')\right|\\
&\leq C_{t,p,\s } \int_{B_{\eps}(\xi)\cap P^{-}}  G_{t-r'}(x-y') |z'|dr'dy' \nu(dz')\\
&= C_{t,p,\s }  \mu_G \big(B_\eps(\xi)\cap P^{-}\big) 
\leq  C_{t,p,\s}  \mu_G \big(B_\eps(\xi)\big).
\end{align*}

Then, combining \eqref{Radon2} and \eqref{Radon3}
with $|q| = \max\{ q^+, q^-\}$,
 leads to 

\noi
\begin{align}\label{Radon4}
\Big| \E\big[ \W D_{r, y, z} u(t,x) \big]  \Big| =|q(r,y,z)| \leq   \mathcal{C}_{t,p,\s}   G_{t-r}(x-y) |z|
\end{align}

\noi
 with $\mathcal{C}_{t,p,\s}$ as in \eqref{CTPS}. Note that the above bound holds
 for any $\W$ with finite second moment and $\| \W\|_{p^\ast} \leq 1$.
 Finally, by duality

 \noi
 \begin{align}
 \| D_{r,y, z} u(t,x) \|_p 
 &= \sup_{ \| \W \|_{p^\ast}\leq 1} \Big| \E\big[ \W D_{r, y, z} u(t,x) \big]  \Big| \notag  \\
 &=  \sup_{k\geq 1} \sup_{ \| \W \|_{p^\ast}\leq 1} \Big| \E\big[ T_k(\W) D_{r, y, z} u(t,x) \big]  \Big|   \notag \\
 &=   \sup_{ \| \W \|_{p^\ast}\leq 1} \Big| \E\big[ \W D_{r, y, z} u(t,x) \big]  \Big| \ind_{\{\| \W\|_2 < \infty \}},
  \label{Radon5}
 \end{align}
where $T_k(x) =\max\{  \min\{ x, k\} , -k \}$. 
Hence, the bound \eqref{Radon4}, together with 
 \eqref{Radon5}, implies \eqref{Z1}
 for $\fm$-almost every $(r, y, z)\in\mathbf{Z}$.
 The proof is then completed. 
 \qedhere

 \end{proof}

As a corollary of Proposition  \ref{prop_MD},
we derive the similar bound for $D_\xi \s^2( u(t,x) )$.
Note that the function $\s^2$ may not be Lipschitz so that 
a direct application of Proposition  \ref{prop_MD}
and \eqref{add1b} is not possible. 
The following result will be used in \eqref{use_DSU}
to establish the  \textsf{QCLT}  in  Theorem \ref{thm_main}.

\begin{corollary}
\label{cor_DSU}
Let the assumptions in Proposition \ref{prop_MD} hold with 
$m_4<\infty$. Assume that $\s:\R\to\R$ is Lipschitz. 
Then, for any $(t, x)\in(0, \infty) \times\R$, $\s^2(u(t,x)) \in \dom(D)$ 
and 
 for $\fm$-almost every $(r,y,z)\in [0,t] \times \R \times \R_0$,
 we have 
 
 \noi
\begin{align}
\label{Dsu}
\|D_{r,y,z} \s^2(u(t,x))\|_2 \leq C_t G_{t-r}(x-y)|z|(1+|z|) ,
\end{align}

\noi
where $C_t=\big({\rm Lip}^2(\s)+2{\rm Lip}(\s)K_{t,4,\s}^{1/4}\big)\big(\mathcal{C}_{t,4,\s}^2  +\mathcal{C}_{t,4,\s}\big)$
with $K_{t,4,\s}$, $\mathcal{C}_{t,4,\s}$ as in \eqref{KTPs}, \eqref{CTPS}.
\end{corollary}

\begin{proof}
Let $F=u(t,x)$. Then, $F\in\dom(D)$ with $D^+_\xi F = D_\xi F$.
 We need to show $\s^2(F) \in  \dom(D)$ when $\s^2$ is not necessarily Lipschitz. 
In view of  Lemma \ref{rem_BZ-ii},  
it suffices to prove that $\E \int_{\mathbf{Z}} |D_{\xi}^{+}\s^2(F)|^2 \, \fm(d\xi)<\infty$.
Let $\xi=(r,y,z)$ with $r<t$. Then
\begin{align*}
D_{\xi}^{+}\s^2(F)
&=\s^2(F+D_{\xi}^{+}F)-\s^2(F) \\
&=\big(\s(F+D_{\xi}^{+}F)-\s(F)\big)^2 + 2\s(F) \big(\s(F+D_{\xi}^{+}F)-\s(F)\big).
\end{align*}

\noi
Using the Lipschitz property of $\s$  and $D_{\xi}^{+}F=D_{\xi}F$ for $F \in  \dom(D)$, 
we obtain
\[
|D_{\xi}^{+}\s^2(F)|\leq {\rm Lip}^2(\s)|D_{\xi}F|^2 + 2{\rm Lip}(\s)|\s(F)| \cdot  |D_{\xi}F|.
\]
By Minkowski's inequality, followed by Cauchy-Schwarz inequality 
and \eqref{KTPs}, 
we obtain

\noi
\begin{align*}
\|D_{\xi}^{+}\s^2(F)\|_2 
& \leq {\rm Lip}^2(\s)\|D_{\xi}F\|_4^2 
+ 2{\rm Lip}(\s)\|\s(F)\|_4 \cdot \|D_{\xi} F\|_4\\
& \leq C_t' (\|D_{\xi}F\|_4^2 +\|D_{\xi}F\|_4),
\end{align*}
where $C_t'= {\rm Lip}^2(\s)+2{\rm Lip}(\s)K_{t,4,\s}^{1/4}$.
Using \eqref{Z1} and $G_t^2(x)=\frac{1}{2}G_t(x)$, 
we have for $\fm$-almost every $\xi = (r, y, z)\in\mathbf{Z}$,

\noi
\begin{align} \label{Dsu-plus}
\begin{aligned}
\|D_{\xi}^{+}\s^2(F)\|_2 
&\leq C_t (G_{t-r}^2(x-y)|z|^2  +G_{t-r}(x-y)|z|) \\
& \leq C_t G_{t-r}(x-y)|z|(1+|z|),
\end{aligned}
\end{align}

\noi
where  $C_t=C_t' (\mathcal{C}_{t,4,\s}^2  +\mathcal{C}_{t,4,\s})$.
 Therefore, with $m_2, m_4<\infty$,
 we have 
\[
\E\int_{Z}|D_{\xi}^{+}\s^2(F)|^2 \fm(d\xi) \leq C_t^2 \left(\int_0^t \int_{\R} G_{t-r}^2(x-y)dydr\right)\left(\int_{\R_0} |z|^2 (1+|z|)^2 \nu(dz)\right)<\infty.
\]
Hence $\s^2(F) \in  \dom(D)$. 
The bound   \eqref{Dsu} follows from \eqref{Dsu-plus} and  Lemma \ref{rem_BZ-ii}.
\qedhere

\end{proof}

\section{Malliavin-Stein bounds in the Poisson setting}
\label{section-MS}

In this section, we provide the Malliavin-Stein bounds
 in the Poisson setting,
 which will be used for the proof of  \textsf{QCLT}, convergence of finite-dimensional
  distributions (when we prove the functional CLT), 
  and  the result  on asymptotic independence.
 The Malliavin-Stein approach was invented by Nourdin and Peccati 
 in  the Gaussian setting \cite{NP_ptrf} and 
 the general theory has been developed in their monograph \cite{blue}.
 In a paper  \cite{PSTU10} by  Peccati, Sol\'e, Taqqu, and Utzet,
 the Malliavin-Stein approach found its Poisson analogue;
 see also the survey \cite{BP16}.

  In what follows, we first present the univariate Malliavin-Stein bound
  in Section \ref{SEC4_1}
  and  present a multivariate version in 
  Section \ref{SEC4_2}
  that is motivated by \cite{pimentel, tudor23}.

\subsection{Univariate Malliavin-Stein bound}\label{SEC4_1}

Let us begin with the {\it Stein's lemma} 
asserting that for a real  integrable random variable $Z$,
\begin{center}
$Z\sim\cN(0,1)$ if and only if $\bE[ f'(Z)]  = \bE[ Z f(Z)]$
\end{center}
for any absolutely continuous   function 
$f: \R\to\R$ such that $f'(Z)\in L^1(\O)$;
see, e.g., \cite[Lemma 2.1]{Ross11}.
This hints that
if one obtains $\bE[ f'(X) - Xf(X)]\approx 0$ for a large enough family 
of test functions $f$, one can say something about the proximity 
of law of $X$ to that of a standard normal random variable. 
In this paper, we use the Wasserstein distance \eqref{WASS_def}
to describe the distance between random variables,
where the corresponding test functions are $1$-Lipschitz.  
Taking any  $1$-Lipschitz function  $h:\R\to\R$, 
 we solve the following ordinary differential equation (ODE,  known as 
 the {\it Stein's equation})

\noi
\begin{align}
\label{Stein-eq}
 f'(x) - x f(x)=  h(x)-\bE[h(Z)],
\end{align}

\noi
where  $h$ is known and $f$ is the unknown with $Z\sim\cN(0,1)$.
One can easily solve the above ODE:
\[
f_h(x) = e^{\frac{x^2}{2}} \int_{-\infty}^x \big[ h(t) - \E( h(Z) ) \big] e^{\frac{-t^2}{2}} dt,
\]
which is the unique solution to \eqref{Stein-eq} subject to 
$\lim_{|x|\to+\infty} e^{\frac{-x^2}{2}} f(x) =0$. We call $f_h$ the 
{\it Stein's solution} to  the Stein's equation  \eqref{Stein-eq}. 
Via an involved careful analysis, one can derive the regularity properties
of $f_h$: with $\|h'\|_\infty \leq 1$,

\noi
\begin{align} \label{FhSol}
\text{$\| f_h\|_\infty \leq 2 $, 
 $\|f_h'\|_\infty \leq \sqrt{2/\pi}$, and    $\|f_h''\|_\infty \leq 2$};
\end{align}
see, e.g.,   \cite[Lemma 2.5]{Ross11}. 
Now replacing the dummy variable $x$ by $X$ in \eqref{Stein-eq}
and taking  expectation on both sides yields the {\it Stein's bound}: 
\begin{align}\label{Stein_bdd}
d_{\rm Wass}(X, Z)
\leq \sup_{f\in\F_W} \big|  \bE[ f'(X) - Xf(X)]  \big|,
\end{align}
where $\F_W$ is the set of functions $f\in C^2(\R)$ with 
$\| f\|_\infty \leq 2 $, 
 $\|f'\|_\infty \leq \sqrt{2/\pi}$, and    $\|f''\|_\infty \leq 2$
 as in \eqref{FhSol}; see also Proposition 2.1.1 in \cite{GZ_thesis}.

In the following, we present  the univariate Malliavin-Stein bound,
one of our main technical tools,
which is  a variant of \cite[Theorem 3.1]{PSTU10}
that suits our application on SPDEs.
%The way we represent the random variable $X$ as a divergence 
%goes back to the paper \cite{NZhou21} by Nualart and Zhou. 

\begin{theorem}
\label{PSTU-new}
Let $X = \dl(v)\in\dom(D)$ for some $v\in \dom(\dl)$
such that $\Var(X)=1$.\footnote{By definition of $\dom(\dl)$, 
$X = \dl(v)$
is a centered random variable with finite variance. }
Then,  

\noi
\begin{align}
d_{\rm Wass} \big(  X, \cN(0,1) \big) 
&\leq  \sqrt{2/\pi}
 \big\|  1-\jb{DX, v}_{\fH}\big\|_1 
            +     \int_{\mathbf{Z}}  \|  (D_\xi X)^2  v(\xi) \|_1 \fm(d\xi)  \label{MSB1}  \\
 &\leq   \sqrt{2/\pi}  \sqrt{\Var\big(\jb{DX, v}_{\fH}\big)} 
     +     \int_{\mathbf{Z}}  \|  (D_\xi X)^2  v(\xi) \|_1 \fm(d\xi). \label{MSB2}
\end{align}
\end{theorem}

Let us also record a useful fact before moving to the proof of Theorem \ref{PSTU-new}:
for a function $f\in\F_W$ as in \eqref{Stein_bdd}, 
we have

\noi
\begin{align}\label{FTC}
f(x+a) - f(x) - f'(x) a = a \int_0^1 \big( f'(x+ t a) - f'(x) \big) dt
\end{align}

\noi
for any $x, a\in\R$. It follows from $\|f''\|_\infty \leq 2$ that the left side of \eqref{FTC} is bounded by 
$a^2$.

 \begin{proof}[Proof of Theorem \ref{PSTU-new}]
  To obtain a bound 
 for $d_{\rm Wass}( X, Z)$, with $Z\in\cN(0,1)$,
 it suffices to bound $\E[ X f(X) - f'(X)]$  for $f\in\F_W$ as in \eqref{Stein_bdd}.
 Note first that due to $\| f' \|_\infty \leq  \sqrt{2/\pi}$, the above expectation is well defined. 
 Using $X = \dl(v)$ and the duality relation \eqref{dualR}, we can write 
 
 \noi
 \begin{align}
& \E[ X f(X) ]
 =  \E[  \dl(v) f(X) ] = \E\big[ \langle v, Df(X) \rangle_\fH \big] \notag  \\
 &= \E\big[ \langle v, DX  \rangle_\fH  f'(X)  \big] 
 + \E\bigg[  \int_{\mathbf{Z}}  v(\xi)  \big[ f(X + D_\xi X) - f(X) - f'(X) D_\xi X         \big]  \fm(d\xi) \bigg].
 \label{FTC2}
 \end{align}
It follows from \eqref{FTC} with $\|f'' \|_\infty \leq 2$ that 
the second expectation in \eqref{FTC2} is bounded by 

\noi
\begin{align} \label{pre_coup}
\E\bigg[  \int_{\mathbf{Z}}   |v (\xi) |  ( D_\xi X   )^2     \fm(d\xi) \bigg] 
=  \int_{\mathbf{Z}}\|  v(\xi)  ( D_\xi X   )^2 \|_1     \fm(d\xi). 
\end{align}

\noi
It follows that 
\begin{align*}
d_{\rm Wass}( X, Z)
&\leq \sup_{f\in\F_W} \big| \E[ X f(X) - f'(X)]  \big| \\
&\leq  \sqrt{2/\pi} \| 1 - \jb{DX,  v  }_\fH \|_1 + 
 \int_{\mathbf{Z}}\|  v(\xi)  ( D_\xi X   )^2 \|_1     \fm(d\xi).
\end{align*}
Hence, the desired Malliavin-Stein bound  \eqref{MSB1}
is proved, while
the second bound \eqref{MSB2}
follows from Cauchy-Schwarz inequality with
\begin{align}\label{IBPv}
\E[ \jb{DX, v} ] = \E[ X\dl(v)] = \E[ X^2] = 1.
\end{align}
Hence, the proof is completed. 
\qedhere

\end{proof}

\subsection{Multivariate Malliavin-Stein bound}
\label{SEC4_2}
Recall from \eqref{WASS_def} that the Wasserstein distance between 
two random vectors $\mathbf{Y}_1, \mathbf{Y}_2$ on $\R^{d+1}$ is given by 

\noi
\begin{align}\label{WASS_def2}
d_{\rm Wass}( \mathbf{Y}_1, \mathbf{Y}_2 ) =
 \sup\big\{   \big| \E[ h( \mathbf{Y}_1)] -  \E[ h(\mathbf{Y}_2)] \big| : \,  \Lip( h) \leq 1 \big\},
\end{align}

\noi
that is, the supremum runs over all $1$-Lipschitz functions $h: \R^{d+1}\to\R$.
For any $1$-Lipschitz function $h: \R^{d+1}\to\R$, we define $h_\eps(x) =\E[ h( x+ \eps \mathbf{G} ) ]$ 
with $\eps > 0$ and  $\mathbf{G}$ standard normal on $\R^{d+1}$. 
Then, 
$h_\eps$ is continuously differentiable with $\Lip(h_\eps) \leq 1$
and $\| h_\eps - h \|_\infty\to 0$ when $\eps\to 0$,
as one can easily verify. As a result, one can rewrite \eqref{WASS_def2} as

\noi
\begin{align}\label{WASS_def3}
d_{\rm Wass}( \mathbf{Y}_1, \mathbf{Y}_2 ) =
 \sup\big\{   \big| \E[ h( \mathbf{Y}_1)] - 
  \E[ h(\mathbf{Y}_2)] \big| : \,  h\in C^1(\R^{d+1}), \Lip( h) \leq 1 \big\}.
\end{align}

Now let us state   Pimentel's observation that generalizes Stein's lemma. 

\begin{lemma}[\cite{pimentel}]
Let $X$ be a real, integrable random variable
and $\mathbf{Y}=(Y_1, ... , Y_d)$ be a $d$-dimensional 
random vector. Then, 
 $X \sim \cN(0, 1)$ and $X$ is independent of $\mathbf{Y}$   if and only if 
\[
  \bE[\partial_x f(X, \mathbf{Y})]-\bE[Xf(X,\mathbf{Y})]=0,
\]
for all differentiable functions $f:\bR^{1+d}\to \R$ 
with $\dx f(X, \mathbf{Y})  \in L^1(\O)$; see also Lemma 1 and Lemma 2 in \cite{tudor23}.
\end{lemma}

Heuristically speaking, when 
$ \bE[\partial_x f(X, \mathbf{Y})]-\bE[ Xf(X,\mathbf{Y})] \approx 0$
for a large enough family of test functions $f$, 
we have approximate normality for $X$
and the approximate independence between $X$ and $\mathbf{Y}$.
To measure the approximate normality and the approximate independence, 
we can use the   Wasserstein distance
$
d_{\rm Wass}\big(  (X, \mathbf{Y}),   (Z', \mathbf{Y})   \big)
$
with $Z'\sim\cN(0,1)$ independent of $\mathbf{Y}$.
See also the discussion before Theorem \ref{thm3}.
As in the Stein's method, one can consider the following (generalized) 
Stein's equation: for $(x,\pmb{y}) \in \bR^{1+d}$,

\noi
\begin{align}
\label{Stein2}
  \partial_x f(x,\pmb{y})-x f (x,\pmb{y})=h(x,\pmb{y})-\bE[h(Z,\pmb{y})]
  \quad\text{with  $Z \sim \cN(0,1)$},
\end{align}

\noi
where $h\in C^1(\R^{1+d})$ with $\Lip(h)\leq 1$ is given
and $f$ is the unknown. After we solve the above Stein's equation \eqref{Stein2},
we can replace $(x,\pmb{y})$ by $(X,  \mathbf{Y})$ and take expectations at both sides
 of \eqref{Stein2}, to get
 
 \noi
\begin{align}
\label{Stein3}
  \bE[ \partial_x f (X,\mathbf{Y})]-\bE[X f(X,\mathbf{Y})]
=\bE[h(X,\mathbf{Y})]-\bE[h(Z',\mathbf{Y})],
\end{align}

\noi
where $Z' \sim \cN(0,1)$ is independent of $\mathbf{Y}$.
Note that for any fixed $\pmb{y}\in\R^d$,
the above Stein's equation is simply the one in \eqref{Stein-eq}.
As a result, we immediately deduce from 
Section \ref{SEC4_1} that 

\noi 
\begin{align}\label{fh2}
f_h(x, \pmb{y}) 
= e^{\frac{x^2}{2}}\int_{-\infty}^x \big( h(t, \pmb{y}) - \E[ h(Z, \pmb{y})] \big) e^{\frac{-t^2}{2}}dt
\quad\text{with  $Z \sim \cN(0,1)$}
\end{align}
solves the Stein's equation \eqref{Stein2} and the Stein's solution in \eqref{fh2} satisfies 
 
 \noi
 \begin{align}\label{Stein_sol2}
\text{$\| f_h \|_\infty \leq 2$,   $\| \partial_x f_h \|_\infty \leq \sqrt{2/\pi}$,  
and $\|\partial^2_{xx}f_h \|_\infty \leq 2$}.
 \end{align}
Now taking partial derivative $\partial_{y_j} f_h$ with \eqref{fh2},
we have 
\[
\partial_{y_j} f_h(x, \pmb{y}) 
= e^{\frac{x^2}{2}}\int_{-\infty}^x \big(  \partial_{y_j} h(t, \pmb{y}) - \E[  \partial_{y_j} h(Z, \pmb{y})] \big) e^{\frac{-t^2}{2}}dt,
\]

\noi
where $\phi(t) := \partial_{y_j} h(t, \pmb{y})$ is uniformly bounded by $1$
for any $\pmb{y}\in\R^d$ and $j\in\{1,..., d\}$.
In view of the standard Stein's method,
$\partial_{y_j} f_h(x, \pmb{y})$ is the Stein's solution to 
the Stein's equation \eqref{Stein-eq} with $h$ replaced by $\phi$,
so that we have

\noi
\begin{align} \label{2ndbdd}
\| \partial_{y_j} f_h \|_\infty \leq  \sqrt{2\pi}
\quad
{\rm and}\quad
\| \partial_x \partial_{y_j} f_h  \|_\infty \leq  4;
\end{align}

\noi
see, e.g., Lemma 2.5 in the survey \cite{Ross11}. Let us summarize 
the above discussion in the following lemma. See also \cite[Lemma 5.1]{pimentel}
and \cite[Proposition 1]{tudor23}.
We would like to point out that 
the second bound in \eqref{2ndbdd} is new (and of independent interest although 
we do not need it in this work)
and our argument avoids any direct  computations 
done in \cite{pimentel,tudor23}.

\begin{lemma}[Stein's bound] \label{lem_bdd}
{\rm (i)} Consider the multivariate Stein's equation \eqref{Stein2} with $h\in C^1(\R^{1+d})$
and $\Lip(h) \leq 1$. Then, $f_h$ given in \eqref{fh2} solves  \eqref{Stein2} and it satisfies
\eqref{Stein_sol2} and \eqref{2ndbdd}.

\smallskip

\noi
{\rm (ii)} Let $X$ be a real integrable random variable and $\mathbf{Y}$ be any
random vector on $\R^d$. Then,

\noi
\begin{align}\label{Stein_bdd3}
d_{\rm Wass}\big(  (X, \mathbf{Y}),  (Z', \mathbf{Y})     \big)
\leq \sup_{f\in\mathcal{A}_W} \big|    \bE[ \partial_x f (X,\mathbf{Y})]-\bE[X f(X,\mathbf{Y})] \big|,
\end{align}

\noi
where $Z'\sim\cN(0,1)$ is independent of $\mathbf{Y}$,
$\mathcal{A}_W$ is the set of twice-differentiable functions 
$(x, \pmb{y})\in\R^{1+d}\mapsto f(x, \pmb{y})\in\R$ such that 
 $\| f \|_\infty \leq 2$,   $\| \partial_x f \|_\infty \leq \sqrt{2/\pi}$,  
 $\|\partial^2_{xx}f \|_\infty \leq 2$, 
 $\| \partial_{y_j} f \|_\infty \leq  \sqrt{2\pi}$,
and $\| \partial_x \partial_{y_j} f  \|_\infty \leq  4$
for any $j\in\{1,..., d\}$.
\end{lemma}

\begin{proof} We only prove (ii) here. Indeed, 
the bound \eqref{Stein_bdd3} follows from 
\eqref{WASS_def3}, \eqref{Stein3},
and part (i) of  Lemma \ref{lem_bdd}. \qedhere

\end{proof}

Next, we present the multivariate Malliavin-Stein bound in the Poisson setting
that measures both the approximate normality of one marginal
and the approximate independence between  components.
See  \cite[Theorem 1]{tudor23} for an analogous result in the  Gaussian setting.  

\begin{theorem}
\label{stein-th2}
Let $X=\dl(v)\in\dom(D)$ for some $v\in\dom(\dl)$
such that $\Var(X) =1$.
 Consider $\mathbf{Y}=(Y_1,\ldots,Y_d)$
with $Y_j\in\dom(D)$ for each $j\in\{1,..., d\}$.
Let $Z'\sim\cN(0,1)$ be independent of $\mathbf{Y}$, then 
we have 

\noi
\begin{align}
\label{MSB_new}
\begin{aligned}
d_{\rm Wass}\big(  (X, \mathbf{Y}),  ( Z', \mathbf{Y})     \big)
& \leq 
 \sqrt{2/\pi}  
 \| 1 -\langle DX, v \rangle_{\fH} \|_1   
 +  \int_{\mathbf{Z}}  \big\| v(\xi)   ( D_\xi X )^2  \big\|_1 \, \fm(d\xi)
\\
&\qquad\qquad + 
 \sqrt{2\pi} \sum_{j=1}^d   \int_{\mathbf{Z}}  \big\| v(\xi)   ( D_\xi Y_j )  \big\|_1 \, \fm(d\xi) .
\end{aligned}
\end{align}
Moreover, one can further bound 
$ \| 1-\langle DX, v \rangle_{\fH} \|_1$ by $\sqrt{ \Var(\langle DX, v \rangle_{\fH} )}$.

\end{theorem}

\begin{remark}\rm
\label{rem_MSB_new}
In view of Theorem \ref{PSTU-new},
the first two terms in \eqref{MSB_new} quantifies the approximate
normality of $X$, while the third term in  \eqref{MSB_new} 
quantifies the approximate independence between $X$
and $\mathbf{Y}$.
\end{remark}

\begin{proof}[Proof of Theorem \ref{stein-th2}]

Let $f\in\mathcal{A}_W$ be given as in Lemma \ref{lem_bdd}. We need to bound 
the following expectation 
\begin{align} \label{coup1}
 \bE[ \partial_x f (X,\mathbf{Y})]   -\bE[X f(X,\mathbf{Y})]. 
\end{align}

First using the duality relation \eqref{dualR},
we have
 \begin{align} 
 \E[X f(X,\mathbf{Y})] 
 & = \E[ \dl(v) f(X,\mathbf{Y})]  = \E\big[ \jb{ v, D f(X,\mathbf{Y}) }_\fH \big]  \notag  \\
 &=  \E\big[ \jb{ v, DX}_\fH \partial_x f(X, \mathbf{Y}) \big]
  +  \E\big[ \jb{ v,  Df(X, \mathbf{Y}) -\partial_x f(X, \mathbf{Y})   DX}_\fH \big].  \label{coup2}
 \end{align}
 Then, the first term in \eqref{coup2} can be coupled with $ \bE[ \partial_x f (X,\mathbf{Y})]$
 in \eqref{coup1} so that we have 
 
 \noi
 \begin{align} \label{coup1a}
  \begin{aligned}
 \big|  \bE[ \partial_x f (X,\mathbf{Y})]  -  \E\big[ \jb{ v, DX}_\fH \partial_x f(X, \mathbf{Y}) \big] \big|
& \leq \sqrt{2/\pi}\, \big\| 1 -  \jb{ v, DX}_\fH  \big\|_1 \\
&\leq  \sqrt{2/\pi}\,  \sqrt{ \Var( \jb{ v, DX}_\fH  )},
 \end{aligned}
 \end{align}
 
 \noi
 where we used $\| \partial_x f\|_\infty \leq  \sqrt{2/\pi}$
 and Cauchy-Schwarz inequality with  \eqref{IBPv}.
 It remains to bound the second expectation in \eqref{coup2}.
 
 Recall from   Lemma \ref{chain} that for $\fm$-almost every $\xi\in\mathbf{Z}$,
 we can write 
 
 \noi
 \begin{align}
& D_\xi f(X, \mathbf{Y}) -\partial_x f(X, \mathbf{Y})   D_\xi X \notag \\
 =&  f(X + D_\xi X, Y_1 + D_\xi Y_1, ... , Y_d + D_\xi Y_d) -   f(X, \mathbf{Y}) -
 \partial_x f(X, \mathbf{Y})   D_\xi X  \notag \\
 = & f(X + D_\xi X, Y_1 + D_\xi Y_1, ... , Y_d + D_\xi Y_d) - f( X + D_\xi X, \mathbf{Y} )   \label{coup3a} \\
 &\quad  +  f( X + D_\xi X, \mathbf{Y} )
 -   f(X, \mathbf{Y}) -
 \partial_x f(X, \mathbf{Y})   D_\xi X.   \label{coup3b}
 \end{align}
 
 \noi
 Using  $\| \partial_{y_j} f \|_\infty \leq  \sqrt{2\pi}$ for each $j$,
 we can deal with  the term in \eqref{coup3a} as follows:
 note that for any $\pmb{a}, \pmb{b}\in\R^{1+d}$ with $a_1=b_1$,
 \begin{align*}
| f(\pmb{b}) - f(\pmb{a}) | 
&= \bigg| \sum_{j=1}^{1+d} \int_0^1   \partial_j f( \pmb{a} + t(\pmb{b}-\pmb{a})) (b_j-a_j) dt \bigg| \\
&\leq \sqrt{2\pi} \sum_{j=2}^{1+d} |b_j-a_j|,
 \end{align*}
 
 \noi
 from which we deduce that 
 
 \noi
 \begin{align*}
|  f(X + D_\xi X, Y_1 + D_\xi Y_1, ... , Y_d + D_\xi Y_d) - f( X + D_\xi X, \mathbf{Y} ) |
\leq \sqrt{2\pi} \sum_{j=1}^d | D_\xi Y_j|,
 \end{align*}
 and thus
 
 \noi
  \begin{align}   \label{coup3c}
    \begin{aligned}  
&\left|  \E\bigg(  \int_{\mathbf{Z}} v(\xi)  \Big[  f(X + D_\xi X, Y_1 + D_\xi Y_1, ... , Y_d + D_\xi Y_d)
      - f( X + D_\xi X, \mathbf{Y} )  \Big] \fm(d\xi) \bigg) \right| \\
& \leq \sqrt{2\pi} \sum_{j=1}^d   \int_{\mathbf{Z}}  \big\| v(\xi)   ( D_\xi Y_j )  \big\|_1 \, \fm(d\xi).
 \end{aligned}
 \end{align}

\noi
For the term in \eqref{coup3b}, the same arguments in \eqref{FTC2} and 
\eqref{pre_coup} with $\|\partial^2_{xx} f\|_\infty \leq 2$
lead to 

  \noi
  \begin{align}   \label{coup3d}
    \begin{aligned}  
& \left| \E\bigg(  \int_{\mathbf{Z}} v(\xi)  \Big[     f( X + D_\xi X, \mathbf{Y} )
 -   f(X, \mathbf{Y}) -
 \partial_x f(X, \mathbf{Y})   D_\xi X   \Big] \fm(d\xi) \bigg) \right| \\
& \leq    \int_{\mathbf{Z}}  \big\| v(\xi)   ( D_\xi X )^2  \big\|_1 \, \fm(d\xi).
 \end{aligned}
 \end{align}
 Therefore, it follows from Lemma \eqref{lem_bdd}
 with  \eqref{coup1}, \eqref{coup2}, \eqref{coup1a}, \eqref{coup3c},
 and \eqref{coup3d} that 

\noi
\begin{align}
 \notag% \label{MSB_newa}
\begin{aligned}
d_{\rm Wass}\big(  (X, \mathbf{Y}),  ( Z', \mathbf{Y})     \big)
& \leq 
    \sqrt{\frac2\pi}  \cdot
 \| 1 -\langle DX, v \rangle_{\fH} \|_1   + 
 \sqrt{2\pi} \sum_{j=1}^d   \int_{\mathbf{Z}}  \big\| v(\xi)   ( D_\xi Y_j )  \big\|_1 \, \fm(d\xi) \\
&\qquad\qquad +  \int_{\mathbf{Z}}  \big\| v(\xi)   ( D_\xi X )^2  \big\|_1 \, \fm(d\xi).
\end{aligned}
\end{align}

\noi
This proves the bound \eqref{MSB_new}.

 Finally, recalling the relation \eqref{IBPv}, 
 we can further bound 
$ \| 1 -\langle DX, v \rangle_{\fH} \|_1$ by $\sqrt{ \Var(\langle DX, v \rangle_{\fH} )}$;
see also \eqref{coup1a}.

\smallskip

Hence, the proof is completed. 
\qedhere

\end{proof}

\section{Main proofs} \label{SEC5A}

This section is mainly devoted to proving our main Theorem \ref{thm_main}
and Theorem \ref{thm3}.
In Section \ref{SEC5_1}, we first establish 
the strict stationarity of the solution to \eqref{SWE};
in Section \ref{SEC5_2}, 
we present the proof of  Theorem \ref{thm_main}
while the last Section \ref{SEC5_3} contains
the proof of Theorem \ref{thm3}.

\subsection{Stationarity of the solution}
\label{SEC5_1}

Recall from \eqref{def-L} and Remark \ref{REM23} that 
\[
L_t(B) = L([0,t]\times B) = \int_0^t \int_B \int_{\R_0} z \wh{N}(ds, dy, dz),
\]
$(t, B)\in\R_+\times \cB_b(\R)$,
defines  a {\it martingale measure}; see, e.g., \cite[Section 2]{dalang99}.
The following definition
   was introduced in \cite[Definition 5.1]{dalang99}.

\begin{definition}
Let $g=\{g(t,x): (t,x)\in\R_+\times \R\}$ be a random field  defined on the same probability space as 
$L$.
We say that $g$  has property {\rm (S)}  with respect to $L$,
 if the finite-dimensional distributions of $(g^{(z)}, L^{(z)} )$
 do not depend on $z\in \R$,
where 
\begin{align}\label{XLz}
\text{$g^{(z)}(t,x)=g(t,x+z)$ and $L_t^{(z)}(B)=L_t(B+z)$}
\end{align}

\noi
for any $t\in\R_+$, $x\in\R$, and $B\in\cB_b(\R)$ with $B+z:=\{ x+z: x\in B\}$.
\end{definition}

The goal of this subsection is to show that the solution $u(t,x)$ to \eqref{SWE}
satisfies the property (S) with respect to $L$. 
Let us first establish a weaker result.

\begin{lemma}
\label{spat-homog}
The process $L$ is spatially homogeneous, 
in the sense that the finite-dimensional distributions 
of $L^{(z)} = \{L_t^{(z)}(B)\}_{(t, B)\in\R_+  \times \cB_b(\bR)}$ do not depend on $z$.
\end{lemma}
 
\begin{proof}
Define $N^{(z)}(A \times B \times C)=N(A \times (B+z) \times C)$
 for bounded Borel sets $A \subset \R_{+}$,  $B \subset \R$, and $C \subset \R_0$. 
 Due to the translational invariance of Lebesgue measure on $\R$,
 we see that   $N^{(z)}$ is a Poisson random measure on $\mathbf{Z} =\R_+\times\R\times\R_0$
 with the intensity measure $\fm$.
 That is, $N^{(z)}$ has  the same distribution as $N$.
 It follows that
\begin{align*}
L_t^{(z)}(B)=&\int_0^t \int_{\bR}\ind_{B+z}(x)L(ds,dx)=\int_0^t \int_{\bR} \int_{\bR_0}\ind_{B}(x-z) z_1 \wh{N}(ds,dx,dz_1)\\
= &\int_0^t \int_{\bR} \int_{\bR_0}\ind_{B}(x) z_1 \wh{N}^{(z)}(ds,dx,dz_1)\\
 \stackrel{{\rm (law)}}{=} & 
\int_0^t \int_{\bR} \int_{\bR_0}\ind_{B}(x) z_1 \wh{N}(ds,dx,dz_1)=L_t(B).
\end{align*}

\noi
By the same argument, we can show that 
  for any $t_1, \ldots,t_k \in \bR_{+}$ and $B_1,\ldots,B_k \in \cB_b(\bR)$,
\[
\sum_{j=1}^k \al_j  L_{t_j}^{(z)}(B_j)  \stackrel{{\rm (law)}}{=}  \sum_{j=1}^k \al_j  L_{t_j}(B_j) 
\]
for any $\al_1, ... , \al_k\in\R$, $k\in\N_{\geq 1}$. This proves
the desired spatial homogeneity of $L$. \qedhere

\end{proof}

\begin{lemma}
\label{lem:stat}
Let  $u$ denote the solution to  the  equation \eqref{SWE}.
Then, the random field $u$  satisfies the  property {\rm (S)}.
In particular, $u$ is strictly {\rm(}spatially{\rm)} stationary in the sense that 
 for any $(t_1,x_1),\ldots,(t_k,x_k)\in \bR_{+} \times \bR$ $(k\in\N_{\geq 1})$ and for any $z \in \bR$,
\begin{equation}
\label{station}
\big(u(t_1,x_1+z),\ldots,u(t_k,x_k+z)\big) \stackrel{{\rm (law)}}{=} \big(u(t_1,x_1),\ldots,u(t_k,x_k)\big).
\end{equation}
%where $\stackrel{(law)}{=}$ denotes equality in distribution.
\end{lemma}

\begin{proof}
Let $(u_n)_{n\geq 0}$ be the Picard iteration sequence given by \eqref{nPicard}.
Since $u_n(t,x) \to u(t,x)$ in $L^2(\O)$ for all $(t,x)\in \R_{+}\times \R$, 
it suffices to show that $u_n$ satisfies  the property (S) 
 for any $n \geq 0$.
We will prove this by induction on $n\geq 0$.

For $n=0$, the result is clear, since $u_0=1$ and     $L$ is spatially homogeneous
in view of  Lemma \ref{spat-homog}.
For the induction step, we use the same argument as in the proof of 
\cite[Lemma 18]{dalang99}. 
This argument is based on the fact that the stochastic integral with respect to $L$ 
has the following homogeneity property: 
for any predictable process  $g \in L^2(\Omega \times \bR_{+}\times \bR)$,

\noi
\begin{align}
\label{homog-int}
\int_0^{\infty}\int_{\bR}g(s, y)L(ds, dy)=\int_0^{\infty}\int_{\bR}g^{(z)}(s, y) L^{(z)}(ds, dy),
\end{align}

\noi
where $g^{(z)}, L^{(z)}$ are defined as in \eqref{XLz}.
To prove \eqref{homog-int}, we first consider the case where 
$g$ is an elementary process of form \eqref{elem-X}.  
In this case, we have  $g^{(z)}(s,y)=X \ind_{(a,b]}(s) \ind_{B-z}(y)$
and
\[
\mbox{RHS of \eqref{homog-int}}
= X L^{(z)}\big((a,b] \times (B-z)\big)
=X L\big((a,b] \times B\big)
=\mbox{LHS of \eqref{homog-int}}.
\]
That is,  \eqref{homog-int} is proved for elementary processes, while
the general case follows by a standard approximation argument. 
Now suppose that $u_n$ satisfies the property (S), i.e.,

\noi
\begin{align}\label{eqd1}
\text{$(u_n^{(z)}, L^{(z)})$ has the same finite-dimensional distributions
as $(u_n, L)$}.
\end{align}

\noi
From
\eqref{nPicard}
\[
u_{n+1}(t,x) = 1 + \int_0^\infty \int_\R g_{t,x}(s, y) L(ds, dy).
\]
with $g_{t,x}(s,y) = \ind_{[0,t)}(s) G_{t-s}(x-y) \s( u_n(s,y)  )$,
we see that $g_{t,x}\in L^2(\O\times\R_+\times\R)$ is predictable 
and     we can deduce from  \eqref{homog-int} and \eqref{eqd1}
 that 
for any $m\in\N_{\geq 1}$, for any $\al_j, \be_j\in\R_0$,
and for any $(s_j, t_j, x_j, B_j)\in\R_+^2\times\R\times\cB_b(\R)$,
we have 

\noi
\begin{align*}
&\sum_{j=1}^m \big[ \al_j u^{(z)}_{n+1}(t_j, x_j) + \be_j L^{(z)}(s_j, B_j) \big] \\
 =& \sum_{j=1}^m  \al_j   + \int_{\R_+\times\R}    \sum_{j=1}^m \big[  \al_j g_{t_j, x_j +z}(s,y) 
+\be_j \ind_{[0,s_j]\times (z+B_j)}(s,y) \big] L(ds, dy) \\
=&
 \sum_{j=1}^m  \al_j  + \int_{\R_+\times\R}    \sum_{j=1}^m \big[  \al_j g_{t_j, x_j +z}(s,y+z) 
+\be_j \ind_{[0,s_j]\times (z+B_j)}(s,y+z) \big] L^{(z)}(ds, dy)  \\
=&  \sum_{j=1}^m  \al_j   + \int_{\R_+\times\R}    
\sum_{j=1}^m \big[  \al_j   G_{t_j-s}( x_j -y )   \s( u_n^{(z)}(s,y) ) 
+\be_j \ind_{[0,s_j]\times B_j}(s,y) \big] L^{(z)}(ds, dy) \\
 \stackrel{{\rm (law)}}{=}  &
  \sum_{j=1}^m  \al_j   + \int_{\R_+\times\R}    
\sum_{j=1}^m \big[  \al_j   G_{t_j-s}( x_j -y )   \s( u_n(s,y) ) 
+\be_j \ind_{[0,s_j]\times B_j}(s,y) \big] L(ds, dy) \\
=& \sum_{j=1}^m \big[ \al_j u_{n+1}(t_j, x_j) + \be_j L(s_j, B_j) \big] .
\end{align*} 
This concludes our proof. 
\qedhere
\end{proof}

\subsection{Proof of Theorem \ref{thm_main}}
\label{SEC5_2}

The proof consists of four parts as indicated by the statement of Theorem \ref{thm_main}.

\smallskip

\noi
$\bul$  \textbf{Part (i): \textsf{spatial ergodicity}.} 
For any fixed $t\in\R_+$,
the strict stationarity of $\{u(t,x)\}_{x \in \R}$ follows from
Lemma  \ref{lem:stat} immediately. 
The spatial ergodicity follows immediately from   \cite[Lemma 4.2]{BZ23}
and the key estimate \eqref{Z1} (with $p=2$). The argument is identical to that in 
 in the proof of \cite[Theorem 1.1-(i)]{BZ23} and thus omitted here. 
 See also \cite[Remark 1.4]{BZ23} for the discussion on the almost sure convergence and 
 $L^2(\O)$-convergence in \eqref{LLN}. Note that the proof of part (i) only requires 
 $m_2 < \infty$.
 \hfill $\square$

\smallskip

Recall from \eqref{FRT} that 
\[
F_R(t) = \int_{-R}^R [u(t,x) -1] dx.
\]
In the following, we will establish the limiting covariance structure of the process $F_R$.

\smallskip

\noi
$\bul$  \textbf{Part (ii): \textsf{limiting covariance structure}.} 
Fix any $t,s \in \R_{+}$. Using  \eqref{station} and $\E[ u(t,x) ] = 1$,
we can write 
\[
\bE\big[F_R(t) F_R(s)\big]
=\int_{-R}^R \int_{-R}^R {\rm Cov}\big(   u(t,x),   u(s,y) \big) dxdy
=\int_{-R}^R \int_{-R}^R  \rho_{t,s}(x-y) dxdy,
\]

\noi
where   $\rho_{t,s}(x-y) := {\rm Cov}\big(   u(t,x),   u(s,y) \big) $ depends on $x, y$
only through their difference. 
By the dominated convergence theorem, we can obtain \eqref{lim-cov}, since
\begin{align*}
\frac{1}{R} \bE[F_R(t)F_R(s)] 
& =\int_{|x|<2R} \rho_{t,s}(x)\left(1-\frac{|x|}{2R} \right)dx \\
&\xrightarrow{R\to+\infty}  2\int_{\bR}\rho_{t,s}(x)dx=:K(t,s)
\end{align*}

\noi
provided that 
\begin{equation}
\label{rho-integr}
\int_{\bR}|\rho_{t,s}(x)|dx<\infty.
\end{equation}

\noi
Let us verify the condition \eqref{rho-integr} now. 
From \eqref{2-Poincare} and \eqref{Z1}, we deduce that 

\noi
\begin{align*}
|\rho_{t,s}(x)| & \leq \int_0^{\infty}\int_{\R} \int_{\R_0}\| D_{r,y,z}u(t,x)\|_2 \| D_{r,y,z}u(s,0)\|_2dx \\
&\les m_2 \int_0^{\infty}\int_{\bR} G_{t-r}(x-y) G_{s-r}(y)dydr,
\end{align*}

\noi
and here we assume $m_2 < \infty$ in part (ii).  With $G_{t-r} = 0$ for $r\geq t$
and $\int_\R G_t(x)dx = t$ for any $t\in\R_+$,
we deduce  immediately that 

\noi
\begin{align*}
\int_{\R}|\rho_{t,s}(x)|dx
&\les \int_\R \bigg(  \int_0^{t\wedge s}\int_{\bR} G_{t-r}(x-y) G_{s-r}(y)dydr  \bigg) dx \\
&= \int_0^{t\wedge s} (t-r) (s-r)dr < +\infty.
\end{align*} 
Therefore, the limit \eqref{lim-cov} holds true.  
Next, 
we prove that $\s_R^2(t): = \E[  F^2_R(t)] >0$ for all $t>0$ and $R>0$. 
We proceed as in \cite[Lemma 3.4]{DNZ20}. 
Assume that $\s_{R_0}^2(t_0)=0$ for some $t_0>0$ and $R_0>0$. 
Using \eqref{mild}, the isometry property \eqref{m_iso}, 
and the stationarity of $\{u(t,x)\}_{x \in \bR}$, we have

\noi
\begin{align}
0 = \s_{R_0}^2(t_0)&=  \int_{-R_0}^{R_0} \int_{-R_0}^{R_0}  \E\big[\big(u(t_0,x_1)-1\big)\big(u(t_0, x_2)-1\big)\big]dx_1 dx_2  \notag\\
&=m_2  \int_{-R_0}^{R_0} \int_{-R_0}^{R_0} \int_0^{t_0} 
\int_{\R} G_{t_0-s}(x_1-y) G_{t_0-s}(x_2-y) \E\big[|\s(u(s,y))|^2\big] dyds dx_1 dx_2 \notag \\
&=m_2  \int_0^{t_0} \bE\big[|\s(u(s,0))|^2\big]   
\bigg[ \int_\R   \Big( \int_{-R_0}^{R_0} G_{t_0 -s}(x-y)  dx \Big)^2   dy \bigg]     ds. \label{bigB}
\end{align}

\noi
It is easy to see that the integral (with respect to $dy$) in the big bracket in \eqref{bigB}
is strictly positive unless $s \geq t_0$.
Recall that the function $s\in\R_+ \mapsto \bE\big[|\s(u(s,0))|^2\big]  $ is continuous
due to the $L^2$-continuity of $u$ and Lipschitz property of $\s$,
and thus,  we have $\bE\big[|\s(u(s,0))|^2\big] =0$ for any $s\in[0, t_0].$
In particular, 
\[
0 =  \bE\big[|\s(u(0,0))|^2\big]=|\s(1)|^2,
\] 
 which is a contradiction to our standing assumption that $\s(1) \neq 0$.
 This concludes the proof of part (ii). 
  \hfill $\square$

\smallskip

\noi
$\bul$ 
 \textbf{Part (iii): \textsf{quantitative  CLT}.} 
 Assume $m_4 < + \infty$ in the following. We will apply Theorem \ref{PSTU-new}
 to prove the bound  \eqref{QCLT-rate}.
 
Let $X = F_R(t) / \s_R(t)$. Since
 \begin{align*}
 u(t, x) - 1 
 &= \int_0^t \int_{\R\times\R_0} G_{t-s}(x-y)z \s( u(s,y) ) \wh{N}(ds, dy, dz) \\
  &= \dl( V_0  )\quad\text{with $V_0(s,y, z) = G_{t-s}(x-y)z \s( u(s,y) )$}
 \end{align*}
 in view of Lemma \ref{lem25}-(ii) and convention \eqref{def-G}.
Recall that $\varphi_{t,R}(r,y)$ is defined by \eqref{VTR}, and vanishes when $r\geq t$.

Now,   we can write by stochastic  Fubini that  
 
 \noi
 \begin{align}\label{vectorV}
 X =  \dl (V_{t, R})/ \s_R(t)
  \quad\text{with $V_{t, R}  (s,y, z) = \varphi_{t, R}(s, y) z \s( u(s,y) )$;}
 \end{align}
 
 \noi
 see \cite[Lemma 2.6]{BZ23}.

Note that $X$ is centered with variance one. Then, we apply  
Theorem \ref{PSTU-new}   to write 
\begin{align}  \label{Qa}
\begin{aligned} 
&d_{\rm Wass}\big( X,  \cN(0,1)  \big) \\
&\les  \sqrt{\Var( \jb{ DX,  \tfrac{V_{t, R}}{ \s_R(t) }  }_\fH)} +
\frac{1}{\s^3_R(t)} \int_{\mathbf{Z}} \| (D_{r, y,z} F_R(t) )^2 V_{t, R} (r, y, z) \|_1 \, drdy\nu(dz) \\
&=: \mathbf{T}_1 +  \mathbf{T}_2.
\end{aligned}
\end{align}

\noi
Let us first deal with the second term $\mathbf{T}_2 $.

\noi
\begin{align}\label{DFR}
D_{r,y,z}F_R(t) = \int_{-R}^{R}D_{r,y,z}u(t,x)dx,
\end{align}

\noi
and thus by Proposition \ref{prop_MD}, we have 

\noi
\begin{align}  \notag%\label{Qa2}
\| D_{r,y,z}F_R(t)  \|_4 \leq   \int_{-R}^{R} \| D_{r,y,z}u(t,x )\|_4 dx 
\les |z|   \varphi_{t, R}(r, y)  .
\end{align}

\noi
Therefore, it follows from the above discussion, Cauchy-Schwarz, and  \eqref{KTPs} that 

\noi
\begin{align}\label{bdd_T2}
\begin{aligned}
 \mathbf{T}_2
& \leq \frac{1}{\s^3_R(t)}\int_{\mathbf{Z}}  \| D_{r,y,z}F_R(t)  \|^2_4 
 \cdot   \varphi_{t, R}(r, y)   |z| \cdot \| \s( u(s,y) ) \|_2 \, drdy \nu(dz) \\
 &\les R^{-\frac32} \int_0^t\int_{\R\times\R_0} |z|^3  \varphi^3_{t, R}(r, y) 
 \, drdy \nu(dz).
\end{aligned}
\end{align}
Then, using $m_3 = \int_{\R_0} |z|^3 \nu(dz) < \infty$ (due to $m_2, m_4 <\infty$)
and 
$  \varphi^3_{t, R}(r, y) \leq  t^2   \varphi_{t, R}(r, y)$,
we deduce from \eqref{bdd_T2} that 

\noi
\begin{align}\label{bdd_T2b}
\begin{aligned}
 \mathbf{T}_2
& \les R^{-\frac32} \int_0^t \int_\R   \bigg(  \int_{-R}^R G_{t-r}(x-y) dx  \bigg)
 \, drdy \\
& \les R^{-\frac12},
\end{aligned}
\end{align}

\noi
where the last inequality follows by performing integration in the order
of $dy, dx$, $dr$, and the implicit constants in  \eqref{bdd_T2b}
  depend on $\s, t$ but not  on $R$.

\smallskip

Now let us bound the first term   in \eqref{Qa}:
\begin{align} \notag
\mathbf{T_1} = \sqrt{\Var( \jb{ DX,  V_{t, R}/\s_R(t) }_\fH)}
= \frac{1}{\s^2_R(t)} \sqrt{  \Var( \jb{ DF_R(t),  V_{t, R} }_\fH)}.
\end{align}
Then, in view of $\s^2_R(t)\asymp R$,
 it suffices to prove
\begin{align} \label{bdd_T1a}
 \Var( \jb{ DF_R(t),  V_{t, R} }_\fH) \les  R.
\end{align}
First, we decompose the inner product using \eqref{DFR} and
 the integral equation \eqref{eq_MD}
for $D_{r, y, z} u(t,x)$:

\noi
\begin{align} \label{B_decomp}
\begin{aligned}
 \jb{ DF_R(t),  V_{t, R} }_\fH 
 &=    \int_{\mathbf{Z}}  \bigg(\int_{-R}^R D_{r, y, z} u(t,x) dx \bigg) 
 \varphi_{t, R}(r,y) z  \s\big( u(r, y) \big) drdy \nu(dz) \\
  &=   \mathbf{B}_1 +  \mathbf{B}_2 ,
\end{aligned}
\end{align}

\noi
where,  with $\varphi_{t,R}$ as in \eqref{VTR},

\noi
\begin{align}  \notag%\label{B12a}
 \mathbf{B}_1:&= \int_{0}^t \int_{\R\times\R_0} \big|  \varphi_{t, R}(r,y) z  \s\big( u(r, y) \big) \big|^2 
 drdy \nu(dz)  \notag  \\
 &= m_2\int_{0}^t \int_{\R} \big|  \varphi_{t, R}(r,y)   \s\big( u(r, y) \big) \big|^2 
 drdy   ,    \notag%
\end{align}
and 
\begin{align}  
 \mathbf{B}_2:&= 
 \int_{0}^t \int_{\R\times\R_0}
\mathcal{J}(r, y, z)   \varphi_{t, R}(r,y) z  \s\big( u(r, y) \big) 
 drdy \nu(dz)  \label{B12b}
\end{align}

\noi
with
\begin{align}   \label{def_cJ}
\begin{aligned} 
\mathcal{J}(r, y, z) &:= 
\int_{-R}^R\bigg[  \int_{r}^t \int_{\R\times\R_0} G_{t-s_1}(x-y_1)z_1 D_{r,y,z} \s( u(s_1, y_1)) 
 \wh{N}(ds_1,dy_1, dz_1) \bigg] dx \\
 &= \int_{r}^t \int_{\R\times\R_0}   \varphi_{t, R}(s_1, y_1)   z_1 D_{r,y,z} \s( u(s_1, y_1)) 
 \wh{N}(ds_1,dy_1, dz_1).
\end{aligned}
\end{align}

\noi
Then, in view of 
\begin{align}\label{VAR_sum}
\sqrt{ \Var(  \mathbf{B}_1 + \mathbf{B}_2 )} \leq  \sqrt{ \Var(  \mathbf{B}_1   )  }
+  \sqrt{ \Var(  \mathbf{B}_2 )},
\end{align} 

\noi
we only need to prove 
$\Var(  \mathbf{B}_k) \les R$ for $k=1,2$.

\smallskip

 $\star$ {\bf Estimation of  $  \Var(  \mathbf{B}_1) $.} 
 Observe first that for a jointly measurable 
 random field  $\{ \Psi(r,y): (r,y)\in\R_+\times\R\}$
 such that $\int_{\R_+\times\R} \| \Psi(r,y)\|_2^2 drdy <\infty$,
 we have 
 \begin{align}\label{VAR0}
  \begin{aligned}
 \sqrt{ \Var \bigg(  \int_{\R_+} \int_\R \Psi(r,y) dr dy\bigg) }
 &= \bigg\|  \int_{\R_+} \int_\R  \big( \Psi(r,y) - \E[ \Psi(r,y) ] \big) dr dy \bigg\|_2 \\
 &\leq  \int_{\R_+}  \Big\|  \int_\R  \big( \Psi(r,y) - \E[ \Psi(r,y) ] \big) dy \Big\|_2 dr  \\
 & =   \int_{\R_+}   \bigg(  \int_{\R^2} {\rm Cov}( \Psi(r,y),  \Psi(r,y')    \big) dy dy' \bigg)^{\frac12} dr.
 \end{aligned} 
 \end{align} 

\noi
Note that with $\Psi(r,y) =  \big|  \varphi_{t, R}(r,y)   \s\big( u(r, y) \big) \big|^2$,
we can deduce from Lemma \ref{2-Poincare} and Corollary \ref{cor_DSU}
that 

\noi
\begin{align}\label{use_DSU}
\begin{aligned}
&{\rm Cov}( \Psi(r,y),  \Psi(r,y')    \big) 
= \varphi^2_{t, R}(r,y) \varphi^2_{t, R}(r,y')  {\rm Cov} \big(  \s^2(u(r, y)),  \s^2(u(r, y'))   \big) \\
&\leq   \varphi^2_{t, R}(r,y) \varphi^2_{t, R}(r,y')  \cdot
C_t \int_0^r \int_{\R\times\R_0} G_{r-r_0}(y-y_0)  \\
&\qquad\qquad\qquad \times G_{r-r_0}(y'-y_0)  |z_0|^2 (1+ |z_0|)^2 dr_0 dy_0 \nu(dz_0)\\
&\leq C_t (m_2 + m_4)
\varphi^2_{t, R}(r,y) \varphi^2_{t, R}(r,y')  \cdot
\int_0^r \int_{\R} G_{r-r_0}(y-y_0) G_{r-r_0}(y'-y_0)   dr_0 dy_0,
\end{aligned}
\end{align}

\noi
where $C_t$ is the constant  in \eqref{Dsu}. Using the triangle inequality,
we have the following rough bound
\[
G_{r-r_0}(y-y_0) G_{r-r_0}(y'-y_0) \leq G_{2r}(y-y') G_r(y-y_0), 
\]
from which, we deduce that 
\[
\int_0^r \int_{\R} G_{r-r_0}(y-y_0) G_{r-r_0}(y'-y_0)   dr_0 dy_0
\leq r^2 G_{2r}(y-y') \leq t^2 G_{2r}(y-y').
\]
That is, we have 
\[
{\rm Cov}( \Psi(r,y),  \Psi(r,y')    \big) \les 
\varphi^2_{t, R}(r,y) \varphi^2_{t, R}(r,y') G_{2r}(y-y'),
\]
with the implicit constant depending on $t$. Thus, it follows from 
\eqref{VAR0}    that 

\noi
\begin{align}\label{SB1}
\sqrt{ \Var(  \mathbf{B}_1) }
\les 
  \int_{0}^t   \bigg(  \int_{\R^2} \varphi^2_{t, R}(r,y) \varphi^2_{t, R}(r,y') G_{2r}(y-y')
   dy dy' \bigg)^{\frac12} dr.
\end{align}

\noi
Since $\varphi_{t,R}(r,y) = \int_{-R}^R G_{t-r}(x-y)dx \leq t \ind_{\{r < t\}}$
and  $\int_{\R} \varphi_{t,R}(r,y) dy = 2(t-r)R \ind_{\{r < t\}}$,
we have 
\begin{align}\label{SB1b}
\begin{aligned}
 \int_{\R^2} \varphi^2_{t, R}(r,y) \varphi^2_{t, R}(r,y') G_{2r}(y-y')
   dy dy'
  &\leq t^3  \int_{\R^2} \varphi_{t, R}(r,y)  G_{2r}(y-y')
   dy dy' \\
  & \leq 4 t^5 R,
   \end{aligned}   
   \end{align}
where the last step follows by performing integration first in $y'$ and then in $y$.
Combining \eqref{SB1} and \eqref{SB1b} yields
\begin{align} \label{SB1c}
  \Var(  \mathbf{B}_1)  \les  R.
\end{align}

 \smallskip

 $\star$ {\bf Estimation of  $  \Var(  \mathbf{B}_2) $.} 
 Recall from \eqref{B12b} and \eqref{def_cJ} the expressions of $\mathbf{B}_2$
 and $\mathcal{J}(r, y,z)$.
 Then, it follows from \eqref{VAR0}  that

 \noi
 \begin{align}\label{VARB21}
  \begin{aligned}
 \sqrt{ \Var(  \mathbf{B}_2) }
 & \leq \int_0^t 
\bigg\|  \int_{\R\times\R_0}
\mathcal{J}(r, y, z)   \varphi_{t, R}(r,y) z  \s\big( u(r, y) \big) 
dy \nu(dz)  \bigg\|_2  \,  dr \\
&=  \int_0^t 
 \sqrt{ \mathcal{H}_r } \, dr,
 \end{aligned} 
 \end{align}
 
 \noi
where $ \mathcal{H}_r$ denotes the second moment of the integral over $\R\times\R_0$,
i.e.,  
 \begin{align*}
 \mathcal{H}_r:&=
  \int_{(\R\times\R_0)^2}
\E[ \mathcal{J}(r, y, z)\mathcal{J}(r, y', z')   \s( u(r, y) )  \s( u(r, y') )   ] \\
&\qquad\qquad\times 
 \varphi_{t, R}(r,y) z     \varphi_{t, R}(r,y') z'     
dydy' \nu(dz) \nu(dz').
\end{align*}
Let us first estimate the term 
$\E[ \mathcal{J}(r, y, z)\mathcal{J}(r, y', z')   \s( u(r, y) )  \s( u(r, y') )   ]$.
 Since  $  \s(u(r,y))$,   $\s( u(r, y') )$ are $\F_r$-measurable, 
 we deduce from  Lemma \ref{Ito-int} and Lemma \ref{lem25}-(iv) that 
 
 \noi
  \begin{align*}
& \E\big[   \s(u(r,y)) \s( u(r, y') ) \mathcal{J}(r, y',z')\mathcal{J}(r, y,z) \big] \\
& =\E\bigg[  \Big( \int_r^t \int_{\R\times\R_0}  \s(u(r,y)) \varphi_{t, R}(s_1, y_1)  z_1 D_{r, y, z}   \s(u(s_1,y_1)) 
 \wh{N}(ds_1, dy_1, dz_1) \Big) \\
 &\qquad\qquad\times
  \int_r^t \int_{\R\times\R_0}  \s(u(r,y')) \varphi_{t, R}(s_1, y_1)  z_1 D_{r, y', z'}   \s(u(s_1,y_1)) 
 \wh{N}(ds_1, dy_1, dz_1) \bigg] \\
 &=m_2  \int_r^t \int_{\R\times\R_0} \E\Big[  \s(u(r,y))    \s(u(r,y'))      
  \big(  D_{r, y, z}   \s(u(s_1,y_1))  \big)  D_{r, y', z'}   \s(u(s_1,y_1))    \Big] \\
   &\qquad\qquad\times
         \varphi^2_{t, R}(s_1, y_1)   ds_1 dy_1,   
 \end{align*} 
 whose modulus  is bounded by
 
 \noi
 \begin{align*}
 \les  |zz' |   \int_r^t \int_{\R}     
 G_{s_1-r}(y_1-y) G_{s_1-r}(y_1-y')  
         \varphi^2_{t, R}(s_1, y_1)   ds_1 dy_1
 \end{align*}
 
 \noi
 using H\"older's inequality,  \eqref{Z1} with $p=4$,  \eqref{add1b}, and \eqref{KTPs}.
 Then, we have 
  
  \noi
  \begin{align*} 
 \mathcal{H}_r 
 & \les
   \int_r^t \int_{\R^3} \varphi^2_{t, R}(s_1, y_1)   
  G_{s_1-r}(y_1-y)   G_{s_1-r}(y_1-y')  
   \varphi_{t, R}(r,y)      \varphi_{t, R}(r,y') 
  ds_1 dy_1 dy dy' \\
  &\les \int_r^t \int_{\R^3} \varphi_{t, R}(s_1, y_1)   
  G_{s_1-r}(y_1-y)   G_{s_1-r}(y_1-y')  
  ds_1 dy_1 dy dy'  \quad\text{(using $\varphi_{t,R}\leq t$)} \\
&=     \int_r^t  (s_1-r)^2 \bigg( \int_{\R} \varphi_{t, R}(s_1, y_1)     dy_1  \bigg)   ds_1 
=  2R  \int_r^t  (s_1-r)^2  (t-s_1)  ds_1  \les R.
 \end{align*}

\noi
 Note that the above implicit constant does not depend on $r$, so that we deduce 
 easily from \eqref{VARB21} that 
 
 \noi
 \begin{align}\label{SB2c}
 \Var(  \mathbf{B}_2)   \les  R.
 \end{align}
 
 Hence, the desired bound \eqref{bdd_T1a}
 follows from \eqref{B_decomp}, \eqref{VAR_sum}, \eqref{SB1c}, and
 \eqref{SB2c}.
 That is, we have $\mathbf{T}_1 \les 1/\sqrt{R}$,
 which together with \eqref{bdd_T2b} and \eqref{Qa},
 implies the bound \eqref{QCLT-rate}.
 This concludes our proof of the quantitative CLT. 
  \hfill $\square$

 \bigskip

\noi
$\bul$ 
\textbf{Part (iv): \textsf{functional CLT}.} 
We first establish the  convergence of  finite-dimensional distributions (f.d.d. convergence)
and then prove the tightness.

\smallskip

$\star$  \textbf{f.d.d. convergence.}
For the first step, it suffices to show that as $R \to \infty$,
\[
\left(\frac{1}{\sqrt{R}}F_R(t_1),\ldots, \frac{1}{\sqrt{R}}F_R(t_m)\right) 
\stackrel{{\rm (law)}}{\longrightarrow} \big(\cG(t_1),\ldots,\cG(t_m) \big)
\]
for any $0\leq t_1<\ldots<t_m$. By the Cram\'er-Wold theorem, 
it is enough  to prove that 

\noi
\begin{align}  \notag  %\label{fdd-XR}
X_R:=\frac{1}{\sqrt{R}}\sum_{j=1}^{m}b_j F_R(t_j) 
\stackrel{{\rm (law)}}{\longrightarrow} X:=\sum_{j=1}^{m}b_j \cG(t_j)
\end{align}

\noi
for any $b_1,\ldots,b_m \in \bR$. From  \eqref{lim-cov}, 
it follows that 
\[
\tau_R^2:={\rm Var}(X_R)
=\frac{1}{R}\sum_{i,j=1}^{m}b_i b_j \bE[F_R(t_i)F_R(t_j)] \xrightarrow{R\to+\infty}
 \sum_{i,j=1}^{m}b_i b_j K(t_i,t_j)=:\tau^2.
\]
The rest of the proof is trivial if $\tau^2=0$. We assume that $\tau^2 >0$,
and thus, $\tau_R>0$ for $R$ large. In this case, it suffices to show 
\[
\frac{X_R}{\tau_R} \stackrel{{\rm (law)}}{\longrightarrow} 
Z \sim \mathcal{N}(0,1) \quad \mbox{as $R \to \infty$},
\]
and in view of \cite[Proposition C.3.1]{blue}, 
we will prove the following stronger statement:
\begin{equation}
\label{dW-Xtau}
\lim_{R\to+\infty} d_{\rm Wass}\big( X_R/\tau_R,  \cN(0,1) \big) = 0. 
\end{equation}

\noi
To prove \eqref{dW-Xtau}, we apply the Malliavin-Stein bound \eqref{MSB2} to the variable 
$G_R=X_{R}/\tau_R$, which is centered and of variance $1$. 
Moreover, as in \eqref{vectorV}, we can write 
\[
\text{$F_R(t)=\delta(V_{t,R})$ with $V_{t,R}(r, y, z): =  \varphi_{t, R}(r, y) z \s( u(r, y))$,}
\]

\noi
where $\varphi_{t,R}$ is given in \eqref{VTR}.
Then, we have  $G_R=\delta(U_R)$ with
\[
U_R=\frac{1}{\tau_R \sqrt{R}}\sum_{j=1}^{m}b_j V_{t_j,R}.
\]
Therefore, the  Malliavin-Stein bound \eqref{MSB2} implies that
\begin{align}\label{FCLT1}
d_{\rm Wass}\big( X_R/\tau_R,  \cN(0,1) \big) 
 \leq \sqrt{{\rm Var}\big(\langle DG_R, U_R \rangle_{\fH}\big)}
 +  \int_{\mathbf{Z}} \big\| (D_{\xi} G_R)^2  U_R(\xi) \big\|_1  \fm(d\xi).
\end{align}

\noi
Since  
$ \langle DG_R, U_R \rangle_{\fH}=\frac{1}{\tau_R^2 R}\sum_{i,j=1}^{m}b_i b_j \langle 
DF_R(t_i), V_{t_j,R} \rangle_{\fH}
$,
we deduce from Minkowski's inequality (i.e., 
$\sqrt{ \Var(\sum_{i=1}^{n}Y_i )}\leq \sum_{i=1}^{n} \sqrt{{\rm Var}(Y_i)}$
for square-integrable random variables $Y_i$'s)
  that 
 
 \noi
 \begin{align} \notag
\sqrt{{\rm Var}\big(\langle DG_R, U_R \rangle_{\fH}\big)}
\leq \frac{1}{\tau_R^2 R} \sum_{i,j=1}^{m}|b_i b_j| \sqrt{{\rm Var}\big( \langle 
DF_R(t_i), V_{t_j,R} \rangle_{\fH} \big)}.
\end{align}

Recall from \eqref{bdd_T1a} that 
\[
   \Var\big( \langle 
DF_R(t_i), V_{t_j,R} \rangle_{\fH} \big) 
 \les  R
\]
for $t_i = t_j = t$,
provided that $m_4<\infty$.
Indeed,  the same argument (with minor modifications) can be used to show
the same bound for $t_i \neq t_j$, so that 
with $\tau_R\to \tau > 0$, 
we have
 
 \begin{align} \label{FCLT1b}
   \sqrt{{\rm Var}\big(\langle DG_R, U_R \rangle_{\fH}\big)}
\les  R^{-1/2}
\end{align}     as $R$ tends to infinity. 

Next, we examine the second term in \eqref{FCLT1}. 
First, we write using the definitions of $G_R, U_R$
and the elementary inequality $ (b_1 + ... + b_m)^2 \leq m (b_1^2 +... + b_m^2)$
that 

\noi
\begin{align*}
  \int_{\mathbf{Z}} \big\| (D_{\xi} G_R)^2  U_R(\xi) \big\|_1  \fm(d\xi)
  &\leq \frac{m}{\tau_R^3 R^{3/2}} \sum_{i, j=1}^m |b_i| b_j^2   \int_{\mathbf{Z}} \big\|  ( D_{\xi}F_R(t_j) )^2   V_{t_i,R}(\xi)\big\|_1 \fm(d\xi).
\end{align*}

\noi
Note that we have proved  
that
\[
  \int_{\mathbf{Z}} \big\|  ( D_{\xi}F_R(t_j) )^2   V_{t_i,R}(\xi)\big\|_1 \fm(d\xi) \les  R 
  \]
for $t_i = t_j = t$;
see \eqref{Qa} and \eqref{bdd_T2b} with $\s_R(t) \asymp \sqrt{R}$.
The same argument with very minor modifications
can lead to the same bound for $t_i \neq t_j$, and thus, 

\noi
 \begin{align} \label{FCLT1c}
   \int_{\mathbf{Z}} \big\| (D_{\xi} G_R)^2  U_R(\xi) \big\|_1  \fm(d\xi)\les R^{-1/2}
   \end{align}
as $R$ tends to infinity. 

\smallskip

Hence,   \eqref{dW-Xtau} follows from \eqref{FCLT1},  \eqref{FCLT1b}, and  \eqref{FCLT1c},
which verifies  the convergence of  finite-dimensional distributions.  

\smallskip 

$\star$  \textbf{Tightness}. In this step, 
we show that $\{\frac{1}{\sqrt{R}}F_R\}_{R>0}$ is tight in $C([0,T]; \R)$
 for any finite  $T>0$. 
 By \cite[Theorem 12.3]{billingsley68}, it is enough to prove that  
\begin{align}\label{tight1}
\bE[|F_R(t)-F_R(s)|^p] \leq C R^{p/2}|t-s|^p 
\end{align}
for all $t,s \in [0,T]$,
where $p\geq 2$ and
 the constant $C$ does not depend on $t, s,$ nor $R$.
In \cite[Proposition 2.10-(ii)]{BZ23} that deals with the linear case (i.e.,  $\sigma(u)=u$), 
we  established  \eqref{tight1}
under the assumption that  $m_p<\infty$ for some $p \geq 2$.
The exactly same argument, together with Remark \ref{rem_MD},
can lead to the same bound \eqref{tight1}.\footnote{More precisely, 
 the minor change we need here is to use 
 \eqref{KTPs}
 in place of  $\sup\{ \| u(t,x)\|_2 : (t,x)\in[0, T] \times\R \} <\infty$. } 
For our purpose, if we only assume $m_2 <\infty$, the bound 
 \eqref{tight1} holds with $p=2$, which is enough to conclude the tightness. 
We omit the details and refer interested readers to 
 \cite[Proposition 2.10]{BZ23}.

 \smallskip
 
 Hence, the functional CLT in Theorem \ref{thm_main}-(iv) is established
 and thus the proof of Theorem \ref{thm_main} is completed.
 \hfill $\square$

\subsection{Proof of Theorem \ref{thm3}} \label{SEC5_3}

We spilt the section into two parts: in the first part, we prove the bound
\eqref{qindep} and address the asymptotic independence between the spatial
integral (for a fixed time) and any finite-dimensional marginals of the random field 
solution $u$; in the second part, we establish the   asymptotic independence 
result for the finite-dimensional marginals of $\{ F_R(t)/\sqrt{R}: t\in\R_+\}$ and those of $u$.

\smallskip
\noi
$\bul$ {\bf Part $\I$.}
We apply the multivariate Malliavin-Stein bound \eqref{MSB_new}
with 
\begin{align}\label{bfY}
\text{$X= X_R=\frac{F_R(t)}{\s_{R}(t)} = \dl(V_{t, R}/\s_R(t))$ as in \eqref{vectorV} 
and
$\mathbf{Y}=\big(u(t_1,x_1),\ldots,u(t_d,x_d)\big).$}
\end{align}
That is, we have, 
with $Z'\sim\cN(0,1)$ independent of  $\mathbf{Y}$,

\noi
\begin{align*}
& d_{\rm Wass}\big(  (X, \mathbf{Y}),  ( Z', \mathbf{Y})     \big) \\
& \leq 
 \sqrt{2/\pi}  
 \| 1 -\langle DX, V_{t, R}/\s_R(t) \rangle_{\fH} \|_1   
 +  \frac{1}{\s^3_R(t)} \int_{\mathbf{Z}}  \big\| V_{t, R}(\xi)   ( D_\xi F_R(t) )^2  \big\|_1 \, \fm(d\xi)
\\
&\qquad\qquad + 
 \frac{ \sqrt{2\pi} }{\s_R(t)} \sum_{j=1}^d   \int_{\mathbf{Z}}  \big\| V_{t,R}(\xi)    D_\xi  u(t_j, x_j)   \big\|_1 \, \fm(d\xi) .
\end{align*}
As already pointed out in Remark \ref{rem_MSB_new},
the first two terms quantifies the proximity to normality of $X$
and we have already established that they are bounded by $\les 1/\sqrt{R}$.
It remains to bound the third term.

From \eqref{vectorV} and \eqref{VTR}, 
we have 
$V_{t,R}(s,y,z) = \varphi_{t, R}(s, y) z \s( u(s, y))$,
so that 
\[
\|V_{t,R}(s,y,z)\|_2 \les  \varphi_{t, R}(s, y) |z|, 
\]
in view of \eqref{KTPs}.
Next, applying  Cauchy-Schwarz inequality and  the key estimate \eqref{Z1},
 we can obtain that 
\begin{align*}
  \int_{\mathbf{Z}}  \big\| V_{t, R}(\xi)    D_\xi  u(t_j, x_j)   \big\|_1 \, \fm(d\xi) 
   & \leq 
   \int_{\mathbf{Z}}   \| V_{t,R}(\xi) \|_2   \| D_\xi  u(t_j, x_j)  \|_2 \, \fm(d\xi) \\
   &\les   \int_0^{t\wedge t_j} \int_\R \int_{\R_0}   \varphi_{t, R}(s, y) 
   G_{t_j-s}(x_j -y) |z|^2    \, dsdy \nu(dz) 
   \les 1,
\end{align*}
where the last step follows
from $\varphi_{t,R}(s,y) \leq t$, $m_2 = \int_{\R_0} |z|^2\nu(dz)  <\infty$,
and  $\int_\R  G_{t_j-s}(x_j -y)dy = t_j -s$ for $s< t_j$.
Thus, with $\s_R(t) \asymp \sqrt{R}$, we can prove the bound \eqref{qindep}.
Then the asymptotic independence of $X_R$ and $\mathbf{Y}$ (in the sense of 
\eqref{adp})
 follows easily by a triangle inequality: 
 
 \noi
 \begin{align}\label{adpB}
 \begin{aligned}
  d_{\rm Wass}\big(   (X_R ,   \mathbf{Y} ), 
 \big(\wh{X}_R,   \mathbf{Y}   )  \big)  
& \leq  d_{\rm Wass}\big(   (X_R ,   \mathbf{Y} ), 
 \big(Z',   \mathbf{Y}   )  \big)   +  d_{\rm Wass}\big(   (Z' ,   \mathbf{Y} ), 
 \big(\wh{X}_R,   \mathbf{Y}   )  \big)   \\
 &\leq \frac{C}{\sqrt{R}} + d_{\rm Wass} (\wh X_R,   Z' )    \les R^{-\frac12},
 \end{aligned}
 \end{align}

\noi
where  (i) $\wh X_R$ is a  copy of $X_R$ that is independent of $\mathbf{Y}$,
 (ii)  the second inequality follows from the definition of Wasserstein distance \eqref{WASS_def},
and (iii)  the last inequality follows from Theorem \ref{thm_main}.
It is then clear that the condition \eqref{adp} is satisfied,
 as a consequence of  \eqref{adpB}.

\smallskip
\noi
$\bul$ {\bf Part $\II$.} We  proceed by proving the   asymptotic independence 
result for  the finite-dimensional distributions of the spatial integral process
and the random field solution. 
This can be done by applying the Malliavin-Stein bound as in 
 {\bf Part $\I$} with $X$ replaced by finite linear combinations 
 of marginals of $F_R(\bul)/\sqrt{R}$ as in the proof of f.d.d. convergences
 (see {\bf Part (iv)} in Section \ref{SEC5_2}).
In what follows, we proceed differently by directing proving \eqref{adp} using
the Poincar\'e inequality (Lemma \ref{2-Poincare}). 

Put $\mathbf X_R= R^{-\frac12}  (  F_R(s_1), ... ,    F_R(s_j)   )$
for any given $s_1, ..., s_j\in\R_+$ and $j\in\N_{\geq 1}$. Let $\mathbf{Y}$ be as in \eqref{bfY},
and let $f, g$ be  Lipschitz functions on $\R^j, \R^d$ respectively (not necessarily bounded). 
Next, we show
\begin{align}\label{531}
\text{\rm Cov}\big( f(\mathbf{X}_R), g(\mathbf{Y})  \big) \xrightarrow{R\to+\infty} 0.
\end{align}
Indeed, it  first 
follows from Lemma \ref{2-Poincare} and Lemma \ref{chain} with \eqref{2_LIP}
that

\noi
\begin{align*}
&\text{\rm Cov}\big( f(\mathbf{X}_R), g(\mathbf{Y})  \big) 
\leq \int_{\R_+\times\R\times\R_0}  \| D_{r, y, z}  f(\mathbf{X}_R)  \|_2
\| D_{r, y, z}  g(\mathbf{Y})  \|_2 drdy \nu(dz) \\
& \qquad  \les  \sum_{i=1}^j  \sum_{i'=1}^d  \int_{\R_+\times\R\times\R_0} \frac{1}{\sqrt{R}} \| D_{r, y, z} F_R(s_i)  \|_2
\| D_{r, y, z} u(t_{i'}, x_{i'})   \|_2 drdy \nu(dz).
\end{align*}
Using $ \| D_{r, y, z} F_R(s_i)  \|_2 \leq \int_{-R}^R   \| D_{r, y, z}  u(s_i, x)  \|_2 dx$
and \eqref{Z1}, we can further write 

\noi
\begin{align*}
&\text{\rm Cov}\big( f(\mathbf{X}_R), g(\mathbf{Y})  \big)   \\
&  \les  \frac{1}{\sqrt{R}}   \sum_{i=1}^j  \sum_{i'=1}^d  
 \int_{\R_+\times\R\times\R_0} \bigg( \int_{-R}^R
 G_{s_i - r}(x-y) dx \bigg) G_{t_{i'} -r }( x_{i'} -y) |z|^2  drdy \nu(dz) \\
 & \les  \frac{1}{\sqrt{R}}   \sum_{i=1}^j  \sum_{i'=1}^d  
 \int_0^{ \min\{ s_i, t_{i'} \}  }\int_{\R} \bigg( \int_{-R}^R
 G_{s_i - r}(x-y)dx \bigg) G_{t_{i'} -r }( x_{i'} -y)  drdy  
  \les  \frac{1}{\sqrt{R}}, 
\end{align*}

\noi
where the last step follows from the fact that  $\int_\R G_t(x)dx = t$.
This proves \eqref{531}
and concludes our proof of Theorem \ref{thm3}.
 \hfill $\square$

\end{document}